\documentclass{amsart}

\usepackage{amssymb}
\usepackage{amsfonts}
\usepackage{amsmath}
\usepackage{amsthm}
\usepackage{bm}
\usepackage{color}      
\usepackage{graphicx}
\usepackage{subfig}
\usepackage{url}

\newtheorem{lemma}{Lemma}

\newtheorem{theorem}{Theorem}
\newtheorem{proposition}{Proposition}

\newtheorem{remark}{Remark}
\newtheorem{hypothesis}{Hypothesis}
\newtheorem{example}{Example}

\newcommand{\tr}{\mathrm{tr}}        
\newcommand{\Curl}{\mathrm{curl}\,}       
\newcommand{\Div}{\mathrm{div}\,}         
\newcommand{\dd}{:} 
\newcommand{\tp}{^{T}} 
\newcommand{\dij}{\delta_{ij}}

\DeclareMathOperator*{\argmin}{{\arg \min}}
\newcommand{\atan}{\mathrm{atan2}}

\newcommand{\vb}{\mathbf{b}}

\newcommand{\vu}{\mathbf{u}}
\newcommand{\vv}{\mathbf{v}}
\newcommand{\vw}{\mathbf{w}}

\newcommand{\vp}{\mathbf{p}}

\newcommand{\vg}{\mathbf{g}}

\newcommand{\vq}{\mathbf{q}}

\newcommand{\vQ}{\mathbf{Q}}
\newcommand{\vr}{\mathbf{r}}

\newcommand{\vm}{\mathbf{m}}
\newcommand{\vn}{\mathbf{n}}
\newcommand{\vt}{\mathbf{t}}

\newcommand{\vx}{\mathbf{x}}

\newcommand{\vnu}{\bm{\nu}}

\newcommand{\vLambda}{\mathbf{\Lambda}}

\newcommand{\vA}{\mathbf{A}}

\newcommand{\vD}{\mathbf{D}}
\newcommand{\vE}{\mathbf{E}}
\newcommand{\vU}{\mathbf{U}}

\newcommand{\vW}{\mathbf{W}}
\newcommand{\vV}{\mathbf{V}}

\newcommand{\vM}{\mathbf{M}}

\newcommand{\vP}{\mathbf{P}}
\newcommand{\vZ}{\mathbf{Z}}
\newcommand{\vI}{\mathbf{I}}
\newcommand{\vR}{\mathbf{R}}
\newcommand{\vT}{\mathbf{T}}
\newcommand{\vF}{\mathbf{F}}

\newcommand{\vzero}{\mathbf{0}}
\newcommand{\vgamma}{\pmb{\gamma}}
\newcommand{\vGamma}{\pmb{\Gamma}}

\newcommand{\Om}{\Omega}
\newcommand{\dOm}{\partial \Omega}
\newcommand{\Gm}{\Gamma}

\newcommand{\Gmdir}{\Gamma_{D}}
\newcommand{\Oc}{\Om_{\mathrm{c}}} 
\newcommand{\Ochat}{\hat{\Om}_{\mathrm{c}}} 
\newcommand{\bdys}{\Gamma_{s}}
\newcommand{\bdyvn}{\Gamma_{\vn}}
\newcommand{\bdyvu}{\Gamma_{\vu}}
\newcommand{\bdyvU}{\Gamma_{\vU}}
\newcommand{\bdyvNN}{\Gamma_{\vNN}}
\newcommand{\Van}{Z} 

\newcommand{\iO}{\int_{\Omega}}

\newcommand{\iG}{\int_{\Gamma}}

\newcommand{\V}{\mathbb{V}}
\newcommand{\Vperp}{\mathbb{V}^{\perp}}

\newcommand{\Q}{\mathbb{Q}}   
\newcommand{\Hbdy}[1]{H^1_{#1}(\Omega)}   
\newcommand{\Sh}{\mathbb{S}_h}   
\newcommand{\Nh}{\mathbb{N}_h}   
\newcommand{\THh}{\mathbb{T}_h}   
\newcommand{\Uh}{\mathbb{U}_h}   
\newcommand{\UUh}{\overline{\mathbb{U}}_h}   
\newcommand{\Sp}{\mathbb{S}}  
\newcommand{\Vh}{\V_h}   
\newcommand{\LL}{L}  

\newcommand{\R}{\mathbb{R}}   

\newcommand{\Sing}{\mathcal{S}}   

\newcommand{\dt}{\delta t}
\newcommand{\eps}{\varepsilon}

\newcommand{\phase}{\phi}
\newcommand{\phaseref}{\phi^\mathrm{ref}}

\newcommand{\EOF}{E_{\mathrm{OF}}} 
\newcommand{\EOFone}{E_{\mathrm{OF},\mathrm{one}}}
\newcommand{\EOFfunc}{\mathcal{W}_{\mathrm{OF}}}

\newcommand{\symmtraceless}{\vLambda}
\newcommand{\ELdG}{E_{\mathrm{LdG}}}
\newcommand{\ELdGone}{E_{\mathrm{LdG},\mathrm{one}}}
\newcommand{\ELdGfunc}{\mathcal{W}_{\mathrm{LdG}}}
\newcommand{\ELdGform}[2]{a_{\mathrm{LdG}} \left( #1 , #2 \right)}
\newcommand{\LdGspace}[1]{\V (#1)}
\newcommand{\LdGsurf}{f_{\Gm}}
\newcommand{\surfcoef}{\eta_{\Gm}}
\newcommand{\LdGrhs}{\chi_{\mathrm{LdG}}}
\newcommand{\Li}{L}
\newcommand{\Ebulk}{E_{\mathrm{LdG},\mathrm{bulk}}}
\newcommand{\Bulkfunc}{\psi_{\mathrm{LdG}}}
\newcommand{\Bulkimp}{\psi_{c}}
\newcommand{\Bulkexp}{\psi_{e}}
\newcommand{\Bulkcoef}{\eta_{\mathrm{B}}}
\newcommand{\BulkA}{A}
\newcommand{\BulkB}{B}
\newcommand{\BulkC}{C}
\newcommand{\Bulkstab}{D}
\newcommand{\BulkK}{K}
\newcommand{\vQdir}{\vQ_{D}}
\newcommand{\vQhdir}{\vQ_{D,h}}
\newcommand{\vQinit}{\vQ_{0}}

\newcommand{\consistLdG}{\mathcal{E}} 
\newcommand{\consistULdG}{\widetilde{\mathcal{E}}} 

\newcommand{\vNN}{\mathbf{\Theta}}

\newcommand{\Euni}{E_{\mathrm{uni}}}

\newcommand{\Euniring}{\mathring{E}_{\mathrm{uni}}}
\newcommand{\Admisuni}{\mathcal{A}_{\mathrm{uni}}}   
\newcommand{\Eunimain}{E_{\mathrm{uni}-\mathrm{m}}}
\newcommand{\EuniUmain}{\widetilde{E}_{\mathrm{uni}-\mathrm{m}}}

\newcommand{\ek}{k}
\newcommand{\eb}{b}

\newcommand{\Eerk}{E_{\mathrm{erk}}}

\newcommand{\Eerkfunc}{\mathcal{W}_{\mathrm{erk}}}

\newcommand{\Eerkbulk}{E_{\mathrm{erk},\mathrm{bulk}}}

\newcommand{\Eerkring}{\mathring{E}_{\mathrm{erk}}}
\newcommand{\Eerkmain}{E_{\mathrm{erk}-\mathrm{m}}}
\newcommand{\EerkUmain}{\widetilde{E}_{\mathrm{erk}-\mathrm{m}}}
\newcommand{\Eerkone}{E_{\mathrm{erk},\mathrm{one}}}
\newcommand{\EerkUone}{\widetilde{E}_{\mathrm{erk},\mathrm{one}}}

\newcommand{\DWfunc}{\psi_{\mathrm{erk}}}
\newcommand{\DWfuncimp}{\psi_{c}}
\newcommand{\DWfuncexp}{\psi_{e}}
\newcommand{\Admiserk}{\mathcal{A}_{\mathrm{erk}}}   

\newcommand{\anchorN}{K_{\mathrm{a}}^{\vnu}}
\newcommand{\anchorPone}{K_{\mathrm{a},1}^{\perp}}
\newcommand{\anchorPtwo}{K_{\mathrm{a},2}^{\perp}}

\newcommand{\Eerkanch}{E_{\mathrm{erk},\mathrm{a}}}
\newcommand{\Eunianch}{E_{\mathrm{uni},\mathrm{a}}}

\newcommand{\ipanchvn}{a_h^{\vn}}
\newcommand{\ipanchs}{a_h^{s}}
\newcommand{\linanchsA}{\omega_h^{s}}
\newcommand{\linanchsB}{\zeta_h^{s}}

\newcommand{\Eerkext}{E_{\mathrm{erk},\mathrm{ext}}}
\newcommand{\Euniext}{E_{\mathrm{uni},\mathrm{ext}}}
\newcommand{\ipelec}{e_h}

\newcommand{\fieldcoef}{K_{\mathrm{ext}}}

\newcommand{\etens}{\bm{\varepsilon}}

\newcommand{\ebar}{\bar{\varepsilon}}
\newcommand{\epar}{\varepsilon_{\parallel}}
\newcommand{\eperp}{\varepsilon_{\perp}}
\newcommand{\ea}{\varepsilon_{\mathrm{a}}}
\newcommand{\ga}{\gamma_{\mathrm{a}}}

\newcommand{\ipOm}[2]{\left( #1, #2 \right)}        
\newcommand{\ipGm}[2]{\left( #1, #2 \right)_{\Gm}}        
\newcommand{\ipvt}[2]{a_{\vn} \left( #1, #2 \right)}    
\newcommand{\ipkvt}[2]{a_{\vn}^{k} \left( #1, #2 \right)}    
\newcommand{\ips}[2]{a_{s} \left( #1, #2 \right)}    
\newcommand{\ipvT}[2]{a_{\vNN} \left( #1, #2 \right)}    

\newcommand{\Tk}{\mathcal{T}}

\newcommand{\Nk}{\mathcal{N}}

\newcommand{\Pk}{\mathcal{P}}

\newcommand{\interp}{\mathcal{I}_{h}} 
\newcommand{\trunc}[1]{\left\lfloor #1 \right\rfloor} 

\begin{document}

\title[The $\vQ$-tensor Model with Uniaxial Constraint]{The $\vQ$-tensor Model with Uniaxial Constraint}
\author[J.P.~Borthagaray]{Juan Pablo~Borthagaray}
\address[J.P.~Borthagaray]{Department of Mathematics, University of Maryland, College Park, MD 20742, USA, and Departamento de Matem\'atica y Estad\'istica del Litoral, Universidad de la Rep\'ublica, Salto, Uruguay}
\email{jpborthagaray@unorte.edu.uy}
\thanks{JPB has been supported in part by NSF grant DMS-1411808}

\author[S.W.~Walker]{Shawn W.~Walker}
\address[S.W.~Walker]{Department of Mathematics and Center for Computation and Technology (CCT) Louisiana State University, Baton Rouge, LA 70803}
\email{walker@math.lsu.edu}
\thanks{SWW has been supported in part by NSF grant DMS-1555222 (CAREER)}

\begin{abstract}
This chapter is about the modeling of nematic liquid crystals (LCs) and their numerical simulation.  We begin with an overview of the basic physics of LCs and discuss some of their many applications.  Next, we delve into the modeling arguments needed to obtain macroscopic order parameters which can be used to formulate a continuum model.  We then survey different continuum descriptions, namely the Oseen-Frank, Ericksen, and Landau-deGennes ($\vQ$-tensor) models, which essentially model the LC material like an anisotropic elastic material.  In particular, we review the mathematical theory underlying the three different continuum models and highlight the different trade-offs of using these models.

Next, we consider the numerical simulation of these models with a survey of various methods, with a focus on the Ericksen model.  We then show how techniques from the Ericksen model can be combined with the Landau-deGennes model to yield a $\vQ$-tensor model that \emph{exactly} enforces uniaxiality, which is relevant for modeling many nematic LC systems.  This is followed by an in-depth numerical analysis, using tools from $\Gamma$-convergence, to justify our discrete method. We also show several numerical experiments and comparisons with the standard Landau-deGennes model.
\end{abstract}

\maketitle


%
%

\section{Physics of Liquid Crystals}\label{sec:physics_LCs}

\subsection{Fundamentals}\label{sec:fundamentals}

The name ``liquid crystal'' appears self-contradictory.  A crystal has a rigid molecular structure and so is associated with being a solid.  How could a crystal be liquid?  Thinking more broadly, a crystal is matter that possesses some kind of macroscopic order, such as having individual molecules arranged in a lattice.  On the other hand, a liquid has no macroscopic order.

Liquid crystals (LCs) are a \emph{meso-phase} of matter, having a degree of macroscopic order that is \emph{between} a liquid and a solid \cite{Virga_book1994}.  A classic solid crystal has both translational order (points in the lattice do not move) and orientational order (neighboring molecules have similar orientation).  A LC has no strong translational order, i.e. the molecules are free to slide about, but they must roughly maintain the same orientation with neighboring molecules.  Thus, LCs have \emph{partial orientational order}.

The initial, accidental discovery of LCs is classically attributed to the Austrian botanist Friedrich Reinitzer \cite{Reinitzer_chem1888,Reinitzer_LC1989}, who was studying carrots.  While heating cholesteryl benzoate, he saw the material exhibit an LC phase.  In order to better understand this, he sought the help of German physicist, Otto Lehmann \cite{Lehmann_ZPC1889} who had experimental apparatus capable of better analysis.  After this initial work, Lehmann continued to study LCs, while Reinitzer moved on.  Further information on the history of LCs can be found in \cite{Sluckin_book2004}, which contains translations of Reinitzer's and Lehmann's initial papers.  Moreover, one can consult \cite{deGennes_book1995,LC_Elastomers_book2012,Luckhurst_book1994,Mottram_arXiv2014} for more details on the basic physics of LCs. 

The LC state may be obtained as a function of temperature between the crystalline and isotropic liquid phases; in such a case, the material is called a thermotropic LC. Other classes include lyotropic and metallotropic LCs, in which concentration of the LC molecules in a solvent or the ratio between organic and inorganic molecules determine the phase transitions, respectively. 

Let us consider thermotropic LCs. In a crystalline solid, molecules exhibit both long-range ordering of the positions of the centers and orientation of the molecules. As the substance is heated, the molecules gain kinetic energy and large molecular vibrations usually make these two ordering types disappear. This results in a fluid phase. In substances capable of producing an LC meso-phase,
the long-range orientational ordering survives until a higher temperature than the long-range positional ordering.  Whenever long-range positional ordering is completely absent, but orientational order remains, the LC is regarded as {\em nematic}. 
At lower temperatures, the molecules may order along a preferred direction, forming layered structures: this is called the {\em smectic} phase. In turn, smectic mesophases may be classified into subclasses (such as smectics A, smectics C and hexatic smectics), depending on the type and degree of positional and orientational order \cite{deGennes_book1995}. Some substances with chiral molecules (i.e. different from their mirror image) may give rise to a {\em cholesteric} mesophase, in which the structure possesses a helical distortion.

This chapter will only consider nematic LCs, and regard their molecules as rods, elongated in one direction and thin in the other two directions.  Imagining a bunch of thin rods packed together, it is natural to expect the orientation of neighboring rods to be similar, but the rods are free to slide along each other. Indeed, the partial order of LCs is essentially due to the anisotropic shape of the LC molecules. Naturally, most LC molecules do not possess axial symmetry. If the molecules resemble more laths than rods, it is expected that the energy interaction can be minimized if the molecules are fully aligned; this necessarily involves a certain degree of {\em biaxiality}. Roughly, this was the rationale behind the prediction of the biaxial nematic phase by Freiser \cite{Freiser_PRL1970}.

Since that seminal work, empirical evidence of biaxial states in certain lyotropic LCs has been well documented (see \cite{Yu_PRL1980}, for example). Nevertheless, for thermotropic LCs the nematic biaxial phase remained elusive for a long period, and was first reported long after Freiser's original prediction \cite{Acharya_PRL2004,Madsen_PRL2004,Prasad_JACS2005}. As pointed out by Sonnet and Virga \cite[Section 4.1]{Sonnet_book2012}, 
\begin{quote}
The vast majority of nematic liquid crystals do not, at least in homogeneous equilibrium states, show any sign of biaxiality.
\end{quote}

Therefore, in this chapter, we shall focus mainly on uniaxial LCs. We discuss three models for the equilibrium configurations. Uniaxiality is naturally built into both the Oseen-Frank and the Ericksen models, as the LC orientation is modeled by a vector field. In contrast, the Landau-deGennes model represents molecular orientation by means of a tensor field, and thus accounts for biaxiality. However, one can enforce uniaxiality and obtain a model that is not equivalent to either the Ericksen or the Oseen-Frank models. In particular, the uniaxially constrained Landau-deGennes model represents the LC molecule orientation by a {\em line field}, and thus it allows for non-orientable configurations. We will discuss this with more detail in Section \ref{sec:uniaxial_Q-tensor}.

\subsection{Applications}

The most well-known application of LCs is in electronic displays \cite{Hoogboom_RSA2007,Roques-CarmesFeenstra_JAP2004}, which is due to an LC's birefringence property.  Indeed, some LC materials polarize light (depending on the orientation of molecules) and this can be controlled through external fields, such as electric fields. This, combined with sophisticated engineering, delivers the flat panel LC display.

However, many newer uses for LCs are being found in material science, that either further build on LCs ability to manipulate light or take advantage of the mechanical properties of the material's anisotropy \cite{Lagerwall_CAP2012}.  For example, \cite{Humar_OE2010} demonstrates three dimensional LC droplets that act as lasers, which may be used as bio-sensors. Self-assembly of rigid particles (inclusions) immersed in an LC medium \cite{Copar_Mat2014,Copar_PNAS2015} has the potential to make new materials. Clever optical effects with LC droplets \cite{Schwartz_AM2018} provide novel means of creating secure ``markers'' that cannot be counterfeit.

Furthermore, LC models provide a test bed for investigating continuum models of complex fluids (e.g. Ericksen-Leslie \cite{Cruz_JCP2013,Mottram_LC2013,Walkington_M2AN2011}), especially swarms of bacteria \cite{Vicsek_PR2012,Marchetti_RMP2013,Giomi_PRX2015} which is sometimes called \emph{active matter} \cite{Ramaswamy_ARCMP2010,Ramaswamy_JSMTE2017,Doostmohammadi_NC2018,Miles_PRL2019}.  The shape of a bacterium is very reminiscent of LC molecules (elongated rods) so it is not surprising that LC models, coupled with fluid dynamics, may be reused.  Therefore, LC research is a very active field within physics, mathematics, biology, and soft-matter in general.

\section{Modeling of Nematic Liquid Crystals}\label{sec:modeling_LCs}

We review three models for the equilibrium states of LCs.  Specifically, we show how to obtain a continuum description by an appropriate averaging over molecules.  This leads to continuum mechanics type models that derive from minimizing an energy; we refer to \cite{deGennes_book1995,Virga_book1994,Mottram_arXiv2014} for more details on the modeling of LCs.

\subsection{Order Parameters}\label{sec:order_param}
Modeling individual LC molecules is certainly viable via molecular dynamics or Monte Carlo methods \cite{Biscarini_PRL1995,Vanzo_SM2012,Moreno-Razo_Nat2012,Whitmer_PRL2013,Teixeira-Souza_PRE2015,Changizrezaei_PRE2017} and has the advantage of being based on first principles (i.e. completely ``correct''), but is also very expensive computationally. In order to build on these models for, say, coupling to fluids or doing optimal design, we require a macroscopic description of LCs. Usually, the transition between phases of different symmetry is described in terms of an order parameter, that represents the extent to which the configuration of the more symmetric phase differs from that of the less symmetric phase.

In the following discussion, we fix the spatial dimension to be $d = 3$. As a first step, suppose we have an ensemble of LC molecules in a small region where the state of each molecule is defined by its orientation in $\R^{3}$, i.e. let $\vp \in \Sp^{2}$ be a vector in the unit sphere that indicates the orientation of an LC molecule.  Let $\rho (\vp) \geq 0$ be the probability distribution of the orientation of LC molecules. Obviously, the distribution of LC molecules may vary in space (and in time), e.g. $\rho \equiv \rho(\vx,\vp)$, but we shall omit the dependence on these variables for simplicity of notation.
It is reasonable to assume that an LC molecule is just as likely to be observed with orientation $\vp$ as $-\vp$; hence, $\rho$ satisfies $\rho(\vp) = \rho(-\vp)$. Based on this {\em head-to-tail symmetry}, it must be that
\begin{equation}\label{eqn:prob_zero_mean}
	\int_{\Sp^{2}} \vp \rho(\vp) \, d\vp = \vzero,
\end{equation}
so direct averaging does not yield a useful order parameter.

Therefore, the first nontrivial information on the molecule distribution is given by the second moments of $\rho$, namely:
\begin{equation}\label{eqn:order_param_second_moment}
\int_{\Sp^{2}} \vp \vp\tp \rho(\vp) \, d\vp = \vM,
\end{equation}
where $\vM \in \R^{3 \times 3}$ is symmetric and effectively captures the average state of the LC molecules (i.e. this is a useful order parameter).  For a uniformly random (isotropic) distribution of LC molecules, $\rho(\vp) = 1 / (4 \pi)$ and $\vM = \vM_{\mathrm{iso}} = (1/3) \vI$.  Thus, it is convenient to define an auxiliary matrix
\begin{equation}\label{eqn:defn_Q_matrix}
	\vQ := \vM - \vM_{\mathrm{iso}},
\end{equation}
which is symmetric and traceless, i.e.
\begin{equation}\label{eqn:defn_symm_traceless}
	\vQ \in \symmtraceless := \left\{ \vQ \in \R^{3 \times 3} \mid \vQ = \vQ\tp, ~ \tr \; \vQ = 0 \right\}.
\end{equation}
Clearly, for an isotropic distribution of LC molecules, $\vQ = \vzero$. However, as pointed out in \cite[Sec. 1.3.4]{Virga_book1994}, $\vQ = \vzero$ is a necessary but not sufficient condition for isotropy. For the sake of obtaining a continuum theory in which microscopic order is described by $\vQ$ only, we shall regard all distributions satisfying $\vQ = \vzero$ as isotropic.

We can further characterize $\vQ$ by its eigenframe and is often written in the form:
\begin{equation}\label{eqn:Q_matrix_biaxial}
\begin{split}
	\vQ = s_1 (\vn_1 \otimes \vn_1) + s_2 (\vn_2 \otimes \vn_2) - \frac{1}{3} (s_1 + s_2) \vI,
\end{split}
\end{equation}
where $\vn_1$, $\vn_2$ are orthonormal eigenvectors of $\vQ$, with eigenvalues given by
\begin{equation}\label{eqn:Q_matrix_eigenvalues}
\begin{split}
	\lambda_1 = \frac{2 s_1 - s_2}{3}, \quad \lambda_2 = \frac{2 s_2 - s_1}{3}, \quad \lambda_3 = -\frac{s_1 + s_2}{3},
\end{split}
\end{equation}
where $\lambda_3$ corresponds to the eigenvector $\vn_3 \perp \vn_1, \vn_2$. The eigenvalues of $\vQ$ are constrained by
\begin{equation} \label{eqn:restriction_eigenvalues}
-\frac13 \le \lambda_i = \int_{\Sp^{2}} (\vp \cdot \vn_i)^2 \rho(\vp) \, d\vp - \frac13 \le \frac23, \quad i = 1,2,3.
\end{equation}
When all eigenvalues are equal, since $\vQ$ is traceless, we must have $\lambda_1 = \lambda_2 = \lambda_3 = 0$ and $s_1 = s_2 = 0$, i.e. the distribution of LC molecules is isotropic.  If two eigenvalues are equal, i.e.
\begin{equation}\label{eqn:Q_matrix_eigenvalues_uniaxial}
\begin{split}
	\lambda_1 &= \lambda_2 \quad \Leftrightarrow \quad s_1 = s_2, \\
	\lambda_1 &= \lambda_3 \quad \Leftrightarrow \quad s_1 = 0, \\
	\lambda_2 &= \lambda_3 \quad \Leftrightarrow \quad s_2 = 0,
\end{split}
\end{equation}
then we encounter a uniaxial state, in which either molecules prefer to orient in alignment with the simple eigenspace (in case it corresponds to a positive eigenvalue) or perpendicular to it (in case it corresponds to a negative eigenvalue). If all three eigenvalues are distinct, then the state is called biaxial. We recall that, as discussed in Section \ref{sec:fundamentals}, most nematic LCs can be effectively modeled as uniaxial states.

Let us consider a uniaxial state for $\vQ$, that can be written in the equivalent form
\begin{equation}\label{eqn:Q_matrix_uniaxial}
	\vQ = s \left( \vn \otimes \vn - \frac{1}{3} \vI \right).
\end{equation}
Above, $\vn$ is the main eigenvector with eigenvalue $\lambda = 2s/3$; the other two eigenvalues equal $-s/3$. The scalar field $s$ is called the {\em degree of orientation} of the LC molecules. Taking into account identity \eqref{eqn:Q_matrix_eigenvalues} and the restriction \eqref{eqn:restriction_eigenvalues}, it follows that the physically meaningful range is $s \in [-1/2, 1]$. As Figure \ref{fig:meaning_of_s} illustrates, the $s$ variable characterizes the local order. In case $s = 1$, the molecular long axes are in perfect alignment with the direction of $\vn$, whereas $s = -1/2$ represents the state in which all molecules are perpendicular to $\vn$.

\begin{figure}[ht]
\includegraphics[width=0.9\linewidth]{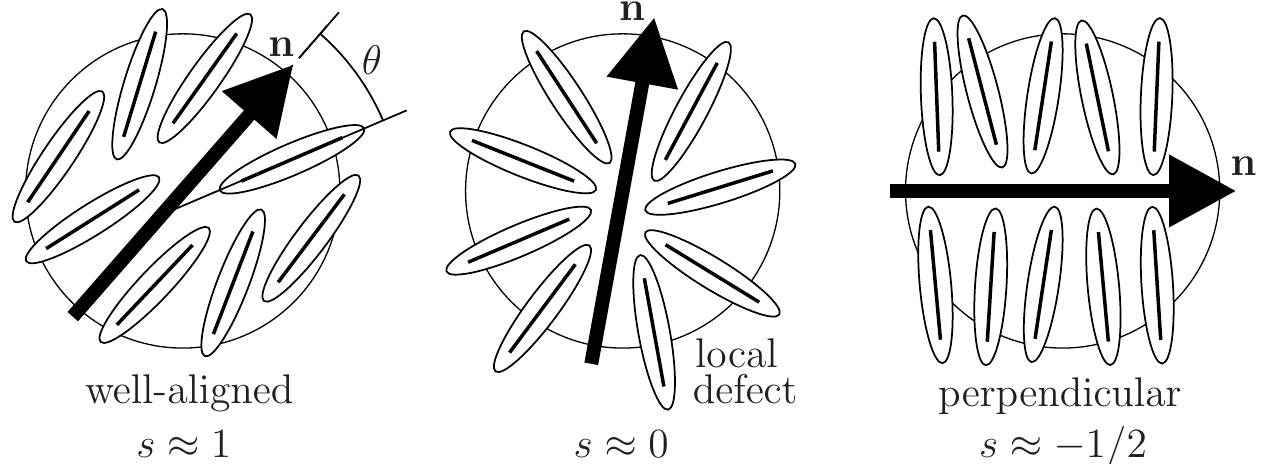} \hspace{-0.5cm}
\caption{The degree-of-orientation variable $s(x)$ provides information on the probability distribution $\rho(x,\cdot)$ over a local ensemble. The case $s=1$ corresponds to a Dirac delta distribution, and represents the state of perfect alignment in which all molecules in the ensemble are parallel to $\vn(x)$. Likewise, $s = -1/2$ represents the state in which $\rho(x,\cdot)$ is a uniform distribution on the unit circle perpendicular to $\vn(x)$. Finally, when $s(x) = 0$, $\rho(x,\cdot)$ is a uniform distribution on $\Sp^2$; such a state is regarded as a {\em defect} in the LC. 
}
\label{fig:meaning_of_s}
\end{figure}

\begin{remark}[problems in $2d$] \label{rmk:QTensor_2d}
The discussion above simplifies considerably when $d=2$. Indeed, the fact that $\vQ$ is a zero-trace tensor forces it to be uniaxial. Writing $\vQ$ as \eqref{eqn:Q_matrix_uniaxial}, we deduce that its eigenvalues are $\lambda_1 = s/2$, with eigenvector $\vn$, and $\lambda_2 = -\lambda_1$, with eigenvector $\vn^\bot$. Because eigenvalues are constrained to satisfy $\lambda_i \in (-1/2,1/2)$, we deduce that the physically meaningful range is $s \in (-1,1)$. Actually, one can further simplify to $s \in [0,1)$ by noting that a state with director $\vn$ and degree of orientation $s<0$ is equivalent to a state with director $\vm \perp \vn$ and degree of orientation $-s$.
\end{remark}

\subsection{Continuum Mechanics}\label{sec:continuum_mech}

Given the order parameter, we still need a model to determine its state as a function of space.  For modeling equilibrium states, this amounts to finding minimizers of an energy functional.  A common approach from continuum mechanics \cite{Holzapfel_book2000,Temam_ContMechBook2005,Truesdell_Book1976} is to construct the ``simplest'' functional possible that is quadratic in gradients of the order parameter whilst obeying standard laws of physics, such as frame indifference and material symmetries.

\subsubsection{Oseen-Frank}\label{sec:oseen-frank_model}

If we assume $\vQ$ has the form \eqref{eqn:Q_matrix_uniaxial}, and assume $s$ is constant, then $\vQ$ is fully determined by $\vn$.  Hence, we may take $\vn$ to be the order parameter, which is usually called the \emph{director}.  If we now seek the simplest energy functional that is quadratic in gradients of $\vn$, and obeying symmetry relations such as $\vn \equiv -\vn$, then we obtain the Oseen-Frank energy \cite{Virga_book1994}:
\begin{equation}\label{eqn:OseenFrank_energy}
\begin{split}
\EOF [\vn] &:= \frac{1}{2} \iO \EOFfunc(\vn,\nabla \vn) \, d\vx, \\
\EOFfunc(\vn,\nabla \vn) &:= \ek_1 (\Div \vn)^2 + \ek_2 (\vn \cdot \Curl \vn)^2 + \ek_3 |\vn \times \Curl \vn|^2 \\
&\qquad + (\ek_2 + \ek_4) [\tr ( [\nabla \vn]^2 ) - (\Div \vn)^2], \end{split}
\end{equation}
where $|\vn|=1$ and $\{ \ek_{i} \}_{i=1}^{4}$ are independent material parameters.  Note that minimizing $\EOF$ subject to $|\vn|=1$ is a non-convex optimization problem.  We give more discussion on mathematical issues, such as regularity of minimizers, in Section \ref{sec:math_background}.  The Oseen-Frank model has been used extensively in the modeling of LC-based flat panel pixel displays, so in that sense has been very successful.  As a simplification, one can take $\ek_1 = \ek_2 = \ek_3 = 1$, $\ek_4 = 0$ to obtain
\begin{equation}\label{eqn:OseenFrank_energy_one_constant}
\EOFone [\vn] := \frac{1}{2} \iO |\nabla \vn|^2 \, d\vx,
\end{equation}
which is known as the \emph{one-constant approximation}.

There are two main drawbacks to using \eqref{eqn:OseenFrank_energy} or \eqref{eqn:OseenFrank_energy_one_constant}. First, the state of the LC molecules is unaffected by the sign of $\vn$.  For example, if $\vn = \vn(\vx)$ is a minimizer of $\EOF$, then arbitrarily changing the sign of $\vn$ on any subset of the domain $\Om$ does not affect the state of the LC molecules.  However, changing the sign of $\vn$ at points arbitrarily close together leads to very large gradients in $\vn$, which of course affects the energy.  Thus, the energy in \eqref{eqn:OseenFrank_energy} does not fully respect the basic symmetry condition $\vn \equiv -\vn$ of nematic LCs.

However, even allowing for the smoothest possible configurations of the director, boundary conditions may force another problem.  For example, suppose $\Om$ is a simply connected open set containing the origin, and suppose we fix $\vn = \vx / |\vx|$ on the boundary $\dOm$. 
Then, by the Poincar\'{e}-Hopf Theorem, every smooth vector field $\vv$ in $\Om$ that coincides with $\vn$ on $\dOm$ must have at least a zero with non-zero index. Therefore, the unit-length vector field $\vn = \vv/|\vv|$ must have a point of discontinuity.
Moreover, if $\Om \subset \R^{2}$, then $\EOFone[\vn] = \infty$ by basic arguments.  Since defects naturally occur in many LC systems, this is a major problem with the Oseen-Frank model.

If $\Om \subset \R^{3}$, then $\EOFone[\vn]$ for a point defect is actually finite.  However, for \emph{line defects} in $\R^3$ such as 
\begin{equation}\label{eqn:line_defect_example}
	\vn (x_1,x_2,x_3) = \frac{(x_1,x_2,0)\tp - (a_1,a_2,0)\tp}{|(x_1,x_2,0) - (a_1,a_2,0)|},
\end{equation}
which is a two dimensional point defect extruded in the $x_3$-direction, one has $\EOFone[\vn] = \infty$.  The same holds true for curvilinear defects in $\R^{3}$.

\begin{remark}\label{rem:line_defects}
This discussion on defects is not merely academic; defects do occur in LC systems.  For example, \cite{Gu_PRL2000} gives experimental evidence for the ``Saturn-ring'' defect which is a closed curvilinear loop in $\R^{3}$ that surrounds a rigid inclusion.  Other examples of LC defects can be found in \cite{Dierking_PRE2005,Goodby_inbook2012,Kinderlehrer_CNAbook1993,Tojo_EPJE2009,Copar_PNAS2015}.
\end{remark}

\subsubsection{Landau-deGennes}\label{sec:landau-degennes_model}

When defects are relevant to an LC system, the Oseen-Frank model is not appropriate.  A better model, using $\vQ$ as the order parameter, is the Landau-deGennes energy \cite{deGennes_book1995,Sonnet_book2012}:
\begin{equation}\label{eqn:Landau-deGennes_energy_example}
\begin{split}
	\ELdG [\vQ] &:= \iO \ELdGfunc (\vQ,\nabla \vQ) \, d\vx + \frac{1}{\Bulkcoef} \iO \Bulkfunc (\vQ) \, d\vx, \\
	\ELdGfunc(\vQ,\nabla \vQ) &:= \frac{1}{2} \left( \Li_{1} |\nabla \vQ|^2 + \Li_{2} |\nabla \cdot \vQ|^2 + \Li_{3} (\nabla \vQ)\tp \dd \nabla \vQ \right),
\end{split}
\end{equation}
where $\{ \Li_{i} \}_{i=1}^{3}$, $\Bulkcoef$ are material parameters, $\Bulkfunc$ is a ``bulk'' (thermotropic) potential and
\begin{equation}\label{eqn:Landau-deGennes_invariants}
\begin{split}
&	|\nabla \vQ|^2 := (\partial_{k} Q_{ij}) (\partial_{k} Q_{ij}), \quad |\nabla \cdot \vQ|^2 := (\partial_{j} Q_{ij})^2, \\ 
&		(\nabla \vQ)\tp \dd \nabla \vQ := (\partial_{j} Q_{ik}) (\partial_{k} Q_{ij}),
\end{split}
\end{equation}
and we use the convention of summation over repeated indices. This is a relatively simple form for $\ELdGfunc$; more complicated models can also be considered \cite{Mottram_arXiv2014,deGennes_book1995,Sonnet_book2012}.

The bulk potential $\Bulkfunc$ is a double-well type of function that confines the eigenvalues of $\vQ$ to the physically meaningful range $\lambda_i \in [-1/d, 1-1/d]$, where the simplest form is given by
\begin{equation}\label{eqn:Landau-deGennes_bulk_potential}
\begin{split}
	\Bulkfunc (\vQ) = \BulkK + \frac{\BulkA}{2} \tr (\vQ^2) - \frac{\BulkB}{3} \tr (\vQ^3) + \frac{\BulkC}{4} \left( \tr (\vQ^2) \right)^2.
\end{split}
\end{equation}
Above, $\BulkA$, $\BulkB$, $\BulkC$ are material parameters such that $\BulkA$ has no sign, and $\BulkB$, $\BulkC$ are positive; $\BulkK$ is a convenient constant.

Stationary points of $\Bulkfunc$ are either uniaxial or isotropic $\vQ$-tensors \cite{Majumdar_EJAM2010}; for such, $\Bulkfunc$ is a quartic polynomial on the degree of orientation $s$ in \eqref{eqn:Q_matrix_uniaxial}, which has a local extremum at $s=0$. A straightforward calculation shows that $s=0$ is a maximum if and only if $\BulkA \leq 0$ (because $\BulkB$, $\BulkC >0$).  Thus, in three dimensions it is typical to let $\BulkA \leq 0$ in order to favor uniaxial states over isotropic states, so throughout this paper we assume that
\begin{equation}\label{eqn:Landau-deGennes_bulk_params}
\begin{split}
	\BulkA \leq 0, \quad \BulkB, \BulkC > 0,
\end{split}
\end{equation}
which implies that $\Bulkfunc (\vQ) \geq 0$ assuming $\BulkK$ is suitably chosen.  In two dimensions, $\tr(\vQ^3) = 0$, because $\vQ^2 = \frac{s^2}{4} \vI$.  Hence, $\BulkB$ is irrelevant when $d=2$, and it is \emph{necessary} that $\BulkA$ be strictly negative in order to have a stable nematic phase.  This also implies that $\Bulkfunc$ is an \emph{even} function of $s$ if $\vQ$ is uniaxial (see Remark \ref{rmk:QTensor_2d}). 

In the same spirit as in \eqref{eqn:OseenFrank_energy_one_constant}, one can take $\Li_{1} = 1$, $\Li_{2} = \Li_{3} = 0$ to obtain a \emph{one-constant approximation}
\begin{equation}\label{eqn:Landau-deGennes_energy_one_const}
\begin{split}
\ELdGone [\vQ] &:= \frac{1}{2} \iO |\nabla \vQ|^2 \, d\vx + \frac{1}{\Bulkcoef} \iO \Bulkfunc (\vQ) \, d\vx.
\end{split}
\end{equation}

\subsubsection{Remarks on Uniaxiality}\label{sec:uniaxial_remarks}

As we discussed in Section \ref{sec:fundamentals}, the biaxial phase is elusive among thermotropic nematic LCs: it took about 30 years after Freiser's prediction \cite{Freiser_PRL1970} to empirically observe a nematic biaxial phase \cite{Acharya_PRL2004,Madsen_PRL2004,Prasad_JACS2005}.
Moreover, in the Landau-deGennes theory, there is no a priori constraint on the eigenvalues of the tensor $\vQ$, in contrast with the probabilistic definition from \eqref{eqn:order_param_second_moment} and \eqref{eqn:defn_Q_matrix}. 

In \cite{Majumdar_EJAM2010} it is shown that, in the low-temperature regime, the Landau-deGennes model can lead to $\vQ$ having physically unrealistic eigenvalues. 
As a remedy, Ball and Majumdar \cite{Ball_MCLC2010} propose a continuum energy functional that interpolates between the Landau-deGennes energy and the mean-field Maier-Saupe energy. This gives rise to the bulk potential 
\[
\psi_B(\vQ) := T \inf_{\rho \in \mathcal{A}_\vQ} \int_{\Sp^{d-1}} \rho(\vp) \ln \rho(\vp) d\vp - \kappa |\vQ|^2,
\]
where $T$ denotes the absolute temperature, $\kappa$ is a constant related to the strength of intramolecular interactions and $\mathcal{A}_\vQ$ is the set of probability distributions that yield the tensor $\vQ$,
\[
\mathcal{A}_\vQ = \left\{ \rho : \Sp^{d-1} \to [0, \infty) ,  \ \int_{\Sp^{d-1}} \rho(\vp)\, d\vp = 1, \ \vQ = \int_{\Sp^{d-1}} \left( \vp \vp\tp - \frac1d \vI \right) \rho(\vp) \, d\vp \right\}.
\]
This potential satisfies the key property $\psi_B (\vQ) \to \infty$ as any of the eigenvalues $\lambda_i$ approaches the boundary of the physically meaningful range.

Clearly, if uniaxiality is built into the model, for instance by forcing $\vQ$ to have the form \eqref{eqn:Q_matrix_uniaxial}, then keeping the eigenvalues within the physically meaningful range reduces to guaranteeing that a single parameter ($s$) lies in a suitable range. This is an important simplification if, for example, the energy has the form \eqref{eqn:Landau-deGennes_energy_example} or if there is a large external forcing. Clearly, another approach to enforce uniaxiality is to consider a director model. If we allow for a variable degree of orientation, then we are led to the Ericksen model, that we discuss in the next section.

\subsubsection{Ericksen Model}\label{sec:ericksen_model}
Though the Landau-deGennes model is quite general, it can be fairly expensive when $d = 3$.  In such a case, since $\vQ \in \R^{3 \times 3}$ and symmetric, it has five independent variables.  Moreover, the bulk potential $\Bulkfunc$ is a non-linear function of $\vQ$, which couples all five variables together when seeking a minimizer of $\ELdG$.

Therefore, we present the Ericksen model of LCs, which is an intermediate model between Oseen-Frank and Landau-deGennes.  Assuming that $\vQ$ is uniaxial \eqref{eqn:Q_matrix_uniaxial}, we can take $s$ \emph{and} $\vn$ as order parameters and obtain an energy analogous to \eqref{eqn:OseenFrank_energy} \cite{Ericksen_ARMA1991,Virga_book1994}:
\begin{equation}\label{eqn:Ericksen_energy}
\begin{split}
\Eerk [s,\vn] &:= \frac{1}{2} \iO \Eerkfunc(s, \nabla s, \vn, \nabla \vn) \, d\vx + \frac{1}{\Bulkcoef} \iO \DWfunc(s) \, d\vx, \\
\Eerkfunc (s, \nabla s, \vn, \nabla \vn) &= \ek_1 s^2 (\Div \vn)^2 + \ek_2 s^2 (\vn \cdot \Curl \vn)^2 + \ek_3 s^2 |\vn \times \Curl \vn|^2 \\
 + (\ek_2 + \ek_4) s^2 & [\tr ( [\nabla \vn]^2 ) - (\Div \vn)^2] + \eb_1 |\nabla s|^2 + \eb_2 (\nabla s \cdot \vn)^2 \\
&\qquad + \eb_3 s (\Div \vn) (\nabla s \cdot \vn) + \eb_4 s \nabla s \cdot [\nabla \vn] \vn,
\end{split}
\end{equation}
where $|\vn|=1$, and $\{ \ek_i \}^4_{i=1}$ and $\{ \eb_i \}^4_{i=1}$ are material constants.  Moreover, we have the \emph{one-constant} version of \eqref{eqn:Ericksen_energy}:
\begin{equation}\label{eqn:Ericksen_energy_TEMP_one_const}
\begin{split}
\Eerkone [s,\vn] &:= \frac{\kappa}{2} \iO |\nabla s|^2 \, d\vx + \frac{1}{2} \iO s^2 |\nabla \vn|^2 \, d\vx + \frac{1}{\Bulkcoef} \iO \DWfunc(s) \, d\vx,
\end{split}
\end{equation}
where $\kappa > 0$ is a single material parameter, and $\DWfunc$ is a double-well potential like $\Bulkfunc$, except acting on $s$.  

We point out that both \eqref{eqn:Ericksen_energy} and \eqref{eqn:Ericksen_energy_TEMP_one_const} are \emph{degenerate} in the sense that $s$ may vanish, which allows for $\vn$ to have discontinuities (i.e. defects) with finite energy.  Indeed, the hallmark of this model is to regularize defects using $s$, but still retain part of the Oseen-Frank model. Discontinuities in $\vn$ may still occur in the singular set
\begin{equation}
    \label{eqn:singular_set}
\Sing := \{ x \in \Om :\; s(x) = 0 \}.
\end{equation}

For problems in $\R^3$, because $\vn \in \Sp^{2}$, it is uniquely defined by two parameters. Thus, in such a case the Ericksen model only has three scalar order parameters, as opposed to five in the Landau-deGennes model. Another advantage of the Ericksen model is that $s$ and $\vn$ are easy to decouple when searching for a minimizer numerically. 

Additionally, the parameter $\kappa$ in \eqref{eqn:Ericksen_energy_TEMP_one_const} plays a major role in the occurrence of defects. Assuming that $s$ equals a sufficiently large positive constant on $\dOm$, if $\kappa$ is large, then $\iO \kappa |\nabla s|^2 dx$ dominates the energy and $s$ stays close to such a positive constant within the domain $\Om$. Thus, defects are less likely to occur. If $\kappa$ is small (say $\kappa < 1$), then $\iO s^2 |\nabla \vn|^2 dx$ dominates the energy, and $s$ may vanish in regions of $\Om$ and induce a defect.  This is confirmed by the numerical experiments in \cite{NochettoWalker_MRSproc2015,Nochetto_SJNA2017}.

\begin{remark}[orientability]
\label{rem:vector_vs_line_field}
Director field models --either Oseen-Frank or Ericksen-- are more than adequate in some situations, although in general they introduce a {\em nonphysical orientational bias} into the problem. Even though LC molecules may be polar, in nematics one always finds that the states with $\vn$ and $-\vn$ are equivalent \cite{Gramsbergen_PR1986}. At the molecular level, this means that the same number of molecules point ``up'' and ``down''. Therefore, \emph{line-fields} are more appropriate for modeling nematic LCs.

Another issue with the use of the vector field $\vn$ as an order parameter instead of the matrix $\vQ$ is that the only allowable defects in such a case are \emph{integer-order} defects. On the other hand $\vQ$, specifically $\vn \otimes \vn$ in \eqref{eqn:Q_matrix_uniaxial}, is able to represent \emph{line fields} having half-integer defects.  These have been largely observed and documented in experiments, see for example \cite{Brinkman_PT1982, Ohzono_SR2017} and references therein. We point out that, if a line field is \emph{orientable}, then a vector field representation is essentially equivalent \cite{Ball_PAMM2007, Ball_ARMA2011}.
\end{remark}

\subsection{Dynamics}\label{sec:dynamics_blurb}

Dynamic LC models become more complicated than the ones discussed in Section \ref{sec:continuum_mech}.  The simplest setting is to assume the dynamics are dictated by a gradient flow \cite{Burger_IFB2002,DoganNochetto_2007CMAME,Braides_book2014}.
Let $t$ represent ``time'' and suppose that $u \equiv u(\vx,t)$ is an evolving solution such that $\lim_{t \to \infty} u(\cdot,t) =: u_{*}$ is a local minimizer of some energy functional $J(u)$ that is bounded below by a constant.  If $u(x,0) := u_{0}$ is the initial guess, then the energy minimizing evolution can be found via (steepest) gradient descent:
\begin{equation}\label{eqn:generic_L2_grad_flow}
\begin{split}
	\ipOm{\partial_{t} u(\cdot,t)}{v} = - \delta_{u} J(u;v),
\end{split}
\end{equation}
for all perturbations $v$ in an appropriate space, where $\ipOm{\cdot}{\cdot}$ is the $L^2(\Om)$ inner product, and $\delta_{u} J(u;v)$ is the variational derivative of $J(u)$ with respect to $u$ in the direction $v$.
Formally, the solution of \eqref{eqn:generic_L2_grad_flow} will converge to a local minimizer of $J$ depending on the initial guess $u_{0}$ \cite{DoganNochetto_2007CMAME}.

We can approximate this evolution by a time semi-discrete scheme known as minimizing movements (see \cite[Ch. ``New problems on minimizing movement'']{DeGiorgi_collection2006}, \cite[Ch. 7]{Braides_book2014}).  Discretizing in time, we let $u_{k} (\vx) \approx u(\vx, k \dt)$, where $\dt > 0$ is a finite time step and $k$ is the time index.  Next, define an auxiliary functional
\begin{equation}\label{eqn:generic_min_movement_functional}
\begin{split}
	F_{k} (u) := J(u) + \frac{1}{2 \dt} \| u - u_{k} \|_{\Om}^2.
\end{split}
\end{equation}
Treating $u_{0}$ as given, setting $\dt > 0$, and initializing $k=0$, we iterate the following scheme:
\begin{enumerate}
	\item Let $u_{k+1} := \argmin F_{k} (u)$.
	\item Update $k := k + 1$.
\end{enumerate}
At each iteration, we solve the variational problem $\delta_{u} F_{k} (u;v) = 0$ for all $v$, i.e.
\begin{equation}\label{eqn:generic_min_movement_weak_form}
\begin{split}
	\ipOm{\frac{u_{k+1} - u_{k}}{\dt}}{v} = - \delta_{u} J(u_{k+1};v), \quad \forall v,
\end{split}
\end{equation}
which is a backward Euler discretization of \eqref{eqn:generic_L2_grad_flow}.  The following properties of this scheme are immediate:
\begin{equation}\label{eqn:generic_min_movement_properties}
\begin{split}
	J(u_{k+1}) \leq J(u_{k+1}) + \frac{1}{2 \dt} & \| u_{k+1} - u_{k} \|_{\Om}^2 = \min_{u} F_{k} (u) \leq J(u_{k}), \\
	\| u_{k+1} - u_{k} \|_{\Om}^2 &\leq - 2 \dt \left( J(u_{k+1}) - J(u_{k}) \right).
\end{split}
\end{equation}
Both \eqref{eqn:generic_L2_grad_flow} and the minimizing movement scheme may be applied to any of the energy functionals we have discussed.

More realistic dynamics can be derived by generalizing the $L^2(\Om)$ inner product and coupling in other PDE constraints (such as Stokes flow) using Onsager's variational framework of \emph{minimum energy dissipation} \cite{Onsager_1_PR1931,Onsager_2_PR1931,LinLiu_CPAM1995,LinLiu_JPDE2001,QianWang_JFM2006,Hyon_DCDS2010}.  We refer the interested reader to \cite{Liu_SJNA2000,Walkington_M2AN2011,Cruz_JCP2013,Guillen-Gonzalez_M2AN2013,Yang_JCP2013,Zhao_JCP2016} for more information on dynamical theories for LCs.

\section{Mathematical Formulation}\label{sec:math_background}

We revisit the LC models discussed in Section \ref{sec:continuum_mech} with an emphasis on their mathematical formulation.  Functional minimization must include an admissible set in which to find the minimizer.  In other words, we describe the function spaces over which to minimize the energies given in Section \ref{sec:continuum_mech}, as well as other mathematical issues. Although often overlooked, the function space is an important part of a continuum mathematical model \cite{Ball_MCLC2017,BallOtto_book2015}. For example, Lavrentiev gap phenomena  between Sobolev and special bounded variation (SBV) functions in certain nematic LC models are examined in \cite{Bedford_ARMA2016}.

\subsection{Oseen-Frank}\label{sec:math_oseen-frank}

Taking the first variation of $\EOF$ and setting to zero yields the first order optimality conditions, i.e. the PDE satisfied by a minimizer.  For simplicity, let us consider the one-constant energy $\EOFone$ whose minimization problem is as follows
\begin{equation}\label{eqn:OF_min_problem}
\begin{split}
	\min_{\vv \in H^1_{\vg} (\Om)} \EOFone[\vv], \quad \text{subject to } |\vv| = 1, \text{ a.e.},
\end{split}
\end{equation}
where $H^{1}_{\vg} (\Om) = \left\{ \vv \in H^{1}(\Om) \mid \vv |_{\dOm} = \vg \right\}$.  This is an instance of the \emph{harmonic map} problem \cite{Schoen_Book1994,Schoen_JDG1982,Brezis_CMP1986,Brezis_BAMS2003,Jost_Book2017,Struwe_Book2008}.  The Euler-Lagrange equation for a minimizer of \eqref{eqn:OF_min_problem}, in strong form, is given by
\begin{equation}\label{eqn:OF_Euler-Lagrange_strong}
\begin{split}
	-\Delta \vn - |\nabla \vn|^2 \vn = \vzero, \text{ in } \Om, \quad \vn = \vg, \text{ on } \dOm.
\end{split}
\end{equation}
The existence of minimizers is complicated if defects are present, unless the boundary data $\vg$ is sufficiently restricted (recall the discussion in Section \ref{sec:oseen-frank_model}).

There is an extensive literature on numerical methods to find a minimizer of \eqref{eqn:OF_min_problem}, e.g. \cite{Hu_SJNA2009,Cohen_CPC1989,Lin_SJNA1989,Alouges_SJNA1997,Bartels_SJNA2006,Bartels_Book2015}, all of which solve the non-convex minimization problem iteratively.  Their main contribution is in how they handle the unit length constraint, $|\vn|=1$, at each iteration.  For instance, Lagrange multipliers may be used by solving the linearized Euler-Lagrange equation in a saddle point framework \cite{Hu_SJNA2009}.  Alternatively, one can project the current (iterative) solution onto the constraint manifold (see the algorithm in Section \ref{sec:Erk_discrete_gradient_flow} for an example), which usually takes advantage of a discrete maximum principle that is built into the method \cite{Cohen_CPC1989,Alouges_SJNA1997,Bartels_Book2015}.  One can find more recent methods for Oseen-Frank type models coupled to other physics in \cite{Adler_SJSC2015,Adler_SJNA2015,Adler_SJSC2016,Gartland_SJNA2015}.

For this paper, our main interest is in modeling defects such that they have \emph{finite} energy, so we will not discuss more on the extensive literature of harmonic maps and their numerical approximation.

\subsection{Landau-deGennes}\label{sec:math_landau-degennes}
We follow \cite{Davis_SJNA1998} to outline the basic theory of the Landau-deGennes model. We also write down a numerical method to simulate Landau-deGennes motivated by \cite{Bajc_JCP2016,Davis_SJNA1998,Zhao_JSC2016}.

\subsubsection{Theoretical Background}\label{sec:LdG_theory}
First, we define the function space for $\vQ$ when seeking a minimizer:
\begin{equation}\label{eqn:Landau-deGennes_function_space}
\begin{split}
	\LdGspace{\vP} := \left\{ \vQ \in H^{1}(\Om) \mid \vQ \in \symmtraceless \mbox{ a.e. in } \Om, ~ \vQ |_{\Gmdir} = \vP \right\},
\end{split}
\end{equation}
where $\Lambda$ is defined by \eqref{eqn:defn_symm_traceless},  $\Gmdir \subset \Gm$ and $\vP \in H^{1}(\Om)$ is arbitrary such that $\vP(\vx) \in \symmtraceless$ for a.e. $\vx \in \Om$.
For the sake of generality, we slightly modify the energy $\ELdG$ in \eqref{eqn:Landau-deGennes_energy_example}:
\begin{equation}\label{eqn:Landau-deGennes_model_problem}
\begin{split}
	\ELdG [\vQ] &:= \iO \ELdGfunc (\vQ,\nabla \vQ) \, d\vx + \frac{1}{\Bulkcoef} \iO \Bulkfunc (\vQ) \, d\vx \\
	&\qquad + \surfcoef \iG \LdGsurf(\vQ) \, dS(\vx) - \iO \LdGrhs(\vQ) \, d\vx,
\end{split}
\end{equation}
where $\ELdGfunc$ is given in \eqref{eqn:Landau-deGennes_energy_example}, $\Bulkfunc$ is given in \eqref{eqn:Landau-deGennes_bulk_potential}, $\surfcoef \geq 0$, and a Rapini-Papoular type anchoring energy \cite{Barbero_JPF1986} is used:
\begin{equation}\label{eqn:Landau-deGennes_surf_energy}
\begin{split}
	\LdGsurf (\vQ) = \frac{1}{2} \tr \left( \vQ - \vQ_{\Gm} \right)^2 \equiv \frac{1}{2} |\vQ - \vQ_{\Gm}|^2,
\end{split}
\end{equation}
where $\vQ_{\Gm} \in H^{1}(\Om)$ is given and $\vQ_{\Gm}(\vx) \in \symmtraceless$ for all $\vx \in \Om$.
The extra term involving $\LdGsurf$ gives an energetic way to penalize boundary conditions, provided $\surfcoef$ is large enough.

The functional $\LdGrhs(\cdot)$ accounts for external forcing effects, e.g. from an electric field.  For example, the energy density of a dielectric with fixed boundary potential is given by $-1/2 \, \vD \cdot \vE$ \cite{Walker_arXiv2018}, where the electric displacement $\vD$ is related to the electric field $\vE$ by the linear constitutive law \cite{Feynman_Lectures1964,deGennes_book1995,Biscari_CMT2007}:
\begin{equation}\label{eqn:LC_dielectric}
	\vD = \etens \vE = \ebar \vE + \ea \vQ \vE, \quad \etens(\vQ) = \ebar \vI + \ea \vQ,
\end{equation}
where $\etens$ is the LC material's dielectric tensor and $\ebar$, $\ea$ are constitutive dielectric permittivities.  Thus, in the presence of an electric field, $\LdGrhs(\cdot)$ becomes
\begin{equation}\label{eqn:LC_electric_energy_functional}
	\LdGrhs (\vQ) = -\frac{1}{2} \vD \cdot \vE = -\frac{1}{2} \left[ \ebar |\vE|^2 + \ea \vE \cdot \vQ \vE \right].
\end{equation}

The minimization problem for the Landau-deGennes energy functional \eqref{eqn:Landau-deGennes_model_problem} is as follows
\begin{equation}\label{eqn:LdG_min_problem}
\begin{split}
	\min_{ \vQ \in \LdGspace{\vQdir}} \ELdG[\vQ],
\end{split}
\end{equation}
where $\vQdir \in H^{1}(\Om)$ is given and $\vQdir(\vx) \in \symmtraceless$ for a.e. $\vx \in \Om$.  This minimization problem is not as delicate as \eqref{eqn:OF_min_problem}; for instance, there is no non-convex, unit-length constraint.  Existence of a minimizer is guaranteed by the following results (taken from \cite[Lem. 4.1]{Davis_SJNA1998}).
\begin{theorem}[coercivity]\label{thm:Gartland_coercivity}
Let $\ELdGform{\cdot}{\cdot} : H^1(\Om) \times H^1(\Om) \to \R$ be the symmetric bilinear form defined by
\begin{equation}\label{eqn:LdG_form}
\begin{split}
	\ELdGform{\vQ}{\vP} = \iO \Li_{1} \nabla \vQ \dd \nabla \vP + \Li_{2} (\nabla \cdot \vQ) \cdot (\nabla \cdot \vP) + \Li_{3} (\nabla \vQ)\tp \dd \nabla \vP \, d\vx.
\end{split}
\end{equation}
Then $\ELdGform{\cdot}{\cdot}$ is bounded.  If $\Li_{1}$, $\Li_{2}$, $\Li_{3}$ satisfy
\begin{equation}\label{eqn:Gartland_coercivity}
\begin{split}
	0 < \Li_{1}, \quad - \Li_{1} < \Li_{3} < 2 \Li_{1}, \quad - \frac{3}{5} \Li_{1} - \frac{1}{10} \Li_{3} < \Li_{2},
\end{split}
\end{equation}
then there is a constant $C > 0$ such that $\ELdGform{\vQ}{\vQ} \geq C | \vQ |^2_{H^1(\Om)}$ for all $\vQ \in H^1(\Om)$.  Moreover, if $|\Gmdir| > 0$, then there is a constant $C' > 0$ such that $\ELdGform{\vQ}{\vQ} \geq C' \| \vQ \|^2_{H^1(\Om)}$ for all $\vQ \in \LdGspace{\vzero}$.
\end{theorem}
Combining Theorem \ref{thm:Gartland_coercivity} with the form of the energy in \eqref{eqn:Landau-deGennes_model_problem} and other basic results (see \cite[Lem. 4.2, Thm. 4.3]{Davis_SJNA1998}) we arrive at the following result.
\begin{theorem}[existence of a minimizer]\label{thm:Gartland_exist_min}
Let $\ELdG$ be of the form \eqref{eqn:Landau-deGennes_model_problem}, and assume that $\vQdir$ and $\Gmdir$ are defined as above and that $\LdGrhs$ is a bounded linear functional on $\LdGspace{\vQdir}$.  Then \eqref{eqn:LdG_min_problem} has at least one minimizer.
\end{theorem}

The Euler-Lagrange equation for a minimizer of \eqref{eqn:LdG_min_problem}, in weak form, is as follows.  Find $\vQ \in \LdGspace{\vQdir}$ such that
\begin{equation}\label{eqn:LdG_Euler-Lagrange_weak}
\begin{split}
	\delta_{\vQ} \ELdG[\vQ;\vP] &:= \ELdGform{\vQ}{\vP} \\
	+ \frac{1}{\Bulkcoef}& \iO \frac{\partial \Bulkfunc (\vQ)}{\partial \vQ} \dd \vP \, d\vx + \surfcoef \iG \frac{\partial \LdGsurf(\vQ)}{\partial \vQ} \dd \vP \, dS(\vx) \\
	&\qquad - \iO \frac{\partial \LdGrhs(\vQ)}{\partial \vQ} \dd \vP \, d\vx = 0, \quad \forall \vP \in \LdGspace{\vzero},
\end{split}
\end{equation}
where the variational derivatives are given by
\begin{equation}\label{eqn:LdG_energy_variational_deriv}
\begin{split}
	\frac{\partial \Bulkfunc (\vQ)}{\partial \vQ} \dd \vP &= \left[ \BulkA \, \vQ - \BulkB \, \vQ^2 + \BulkC \, \tr (\vQ^2) \, \vQ \right] \dd \vP, \\
	\frac{\partial \LdGsurf(\vQ)}{\partial \vQ} \dd \vP &= \left[ \vQ - \vQ_{\Gm} \right] \dd \vP \\
	\frac{\partial \LdGrhs(\vQ)}{\partial \vQ} \dd \vP &= -\frac{1}{2} \ea \vE \cdot \vP \vE,
\end{split}
\end{equation}
where we used
\begin{equation}\label{eqn:simple_variational_deriv_trace}
\begin{split}
	\frac{\partial \tr ( \vQ^2 )}{\partial \vQ} \dd \vP &= 2 \, \vQ \dd \vP, \quad \frac{\partial \tr ( \vQ^3 )}{\partial \vQ} \dd \vP = 3 \, \vQ^2 \dd \vP, \\
	\frac{ \partial \left( \tr ( \vQ^2 ) \right)^2 }{\partial \vQ} \dd \vP &= 4 \, \tr (\vQ^2) \, \vQ \dd \vP.
\end{split}
\end{equation}
The strong form of the Euler-Lagrange equation is given by
\begin{equation}\label{eqn:LdG_Euler-Lagrange_strong}
\begin{split}
	- \left( \Li_{1} \partial_{k} \partial_{k} Q_{ij} + (\Li_{2} + \Li_{3}) \partial_{j} \partial_{k} Q_{ik} \right) \qquad \qquad & \\
	+ \frac{1}{\Bulkcoef} \left[ \BulkA \, Q_{ij} - \BulkB \, Q_{ik} Q_{kj} + \BulkC \, \tr (\vQ^2) \, Q_{ij} \right] &= -\frac{1}{2} \ea E_{i} E_{j}, \text{ in } \Om, \\
	Q_{ij} &= Q_{0,ij}, \text{ on } \Gmdir, \\
	- \left( \Li_{1} \nu_{k} \partial_{k} Q_{ij} + \Li_{2} \nu_{j} \partial_{k} Q_{ik} + \Li_{3} \nu_{k} \partial_{j} Q_{ik} \right) &= Q_{ij} - Q_{\Gm,ij}, \text{ on } \Gm \setminus \Gmdir,
\end{split}
\end{equation}
for $ 1 \leq i,j \leq 3$, where $\vnu \equiv [\nu_{k}]_{k=1}^{3}$ is the unit outer normal of $\Gm$.  If $\Li_{2} = \Li_{3} = 0$, $\vE = \vzero$, and $\Gmdir = \Gm$, then the strong form simplifies to
\begin{equation}\label{eqn:LdG_Euler-Lagrange_strong_one_const}
\begin{split}
	- \Delta \vQ + \frac{1}{\Bulkcoef} \left[ \BulkA \, \vQ - \BulkB \, \vQ^2 + \BulkC \, \tr (\vQ^2) \, \vQ \right] &= \vzero, \text{ in } \Om, \quad 
\vQ = \vQdir, \text{ on } \Gmdir,
\end{split}
\end{equation}
which is an elliptic Dirichlet problem with non-linear lower order term due to the bulk potential.  From \cite[Thm 6.3]{Davis_SJNA1998}, we have
\begin{theorem}[regularity]\label{thm:Gartland_LdG_regularity}
Let $\Om$ be a bounded, open, connected set, and assume $\Om$ is either convex or $C^{1,1}$, and assume $\Gmdir \equiv \Gm := \dOm$.  Moreover, let $F(\vQ) := \iO \LdGrhs (\vQ) \, d\vx$ be a bounded linear functional on $L^2(\Om)$.  Then any solution of \eqref{eqn:LdG_Euler-Lagrange_weak} is in $H^{2} (\Om) \cap H^{1}_{0} (\Om)$.
\end{theorem}

\subsubsection{Numerical Method}\label{sec:LdG_numerics}

One can try to solve \eqref{eqn:LdG_Euler-Lagrange_weak} either directly \cite{Davis_SJNA1998,Gartland_SJNA2015}, or look for an energy minimizer \cite{Lee_APL2002,Ravnik_LC2009,Zhao_JSC2016,Bajc_JCP2016}.  To better compare with our method for the uniaxially constrained Landau-deGennes model (Section \ref{sec:uniaxial_Q-tensor}), we adopt the later approach and state a simple gradient flow method to find a minimizer of \eqref{eqn:LdG_min_problem}.

Let $t$ represent ``time'' and suppose that $\vQ \equiv \vQ(\vx,t)$ is an evolving solution such that $\lim_{t \to \infty} \vQ(\cdot,t) =: \vQ_{*}$ is a local minimizer of $\ELdG$, where $\vQ(\vx,0) = \vQinit$, and $\vQinit \in \LdGspace{\vQdir}$ is the initial guess for the minimizer.  Next, we evolve $\vQ(\cdot,t)$ according to the following $L^2(\Om)$ gradient flow:
\begin{equation}\label{eqn:LdG_L2_grad_flow}
\begin{split}
\ipOm{\partial_{t} \vQ(\cdot,t)}{\vP} = -\delta_{\vQ} \ELdG[\vQ;\vP], \quad \forall \vP \in \LdGspace{\vzero},
\end{split}
\end{equation}
where the variational derivative is given in \eqref{eqn:LdG_Euler-Lagrange_weak}.  Formally, the solution of \eqref{eqn:LdG_L2_grad_flow} will converge to $\vQ_{*}$.

We derive a numerical scheme for approximating \eqref{eqn:LdG_L2_grad_flow} by first discretizing in time by minimizing movements (see Section \ref{sec:dynamics_blurb}). Let $\vQ_{k} (\vx) \approx \vQ(\vx,k \dt)$, where $\dt > 0$ is a finite time-step, and $k$ is the time index.  Then \eqref{eqn:LdG_L2_grad_flow} becomes a sequence of variational problems.  Given $\vQ_{k}$, find $\vQ_{k+1} \in \LdGspace{\vQdir}$ such that
\begin{equation}\label{eqn:LdG_L2_grad_flow_semi-discrete}
\begin{split}
	\ipOm{\frac{\vQ_{k+1} - \vQ_{k}}{\dt}}{\vP} = -\delta_{\vQ} \ELdG[\vQ_{k+1};\vP], \quad \forall \vP \in \LdGspace{\vzero},
\end{split}
\end{equation}
which is equivalent to
\begin{equation}\label{eqn:LdG_L2_grad_flow_semi-discrete_min_movement}
\begin{split}
	\vQ_{k+1} = \argmin_{\vQ \in \LdGspace{\vQdir}} F(\vQ), \quad F(\vQ) := \frac{1}{2 \dt} \| \vQ - \vQ_{k} \|_{L^2(\Om)}^2 + \ELdG[\vQ],
\end{split}
\end{equation}
and yields the useful property $F(\vQ_{k+1}) \leq F(\vQ_{k})$.  However, \eqref{eqn:LdG_L2_grad_flow_semi-discrete} is a \emph{fully-implicit} equation and requires an iterative solution because of the non-linearity in $\Bulkfunc (\vQ)$.  For convenience, we shall, instead, use a \emph{semi-implicit} approach via convex splitting \cite{Wise_SJNA2009,Zhao_JSC2016,Xu_CMAME2019}.  Let us define the following split of \eqref{eqn:Landau-deGennes_bulk_potential}:
\begin{equation}\label{eqn:LdG_bulk_convex_split}
\begin{split}
	\Bulkimp (\vQ) &= \BulkK + \frac{\BulkA + \Bulkstab}{2} \tr (\vQ^2), \\
	\Bulkexp (\vQ) &= \frac{\Bulkstab}{2} \tr (\vQ^2) + \frac{\BulkB}{3} \tr (\vQ^3) - \frac{\BulkC}{4} \left( \tr (\vQ^2) \right)^2, \\
	\Rightarrow \qquad \Bulkfunc (\vQ) &\equiv \Bulkimp (\vQ) - \Bulkexp (\vQ),
\end{split}
\end{equation}
where $\Bulkstab > 0$ is chosen sufficiently large.  Indeed, for all $\vQ$ that satisfy the physical eigenvalue ranges \eqref{eqn:restriction_eigenvalues}, $\Bulkimp$ and $\Bulkexp$ are both convex functions if $\Bulkstab > 0$ is large enough.

Therefore, referring to \eqref{eqn:LdG_Euler-Lagrange_weak}, we obtain the following semi-implicit weak formulation of \eqref{eqn:LdG_L2_grad_flow_semi-discrete}.  Given $\vQ_{k}$, find $\vQ_{k+1} \in \LdGspace{\vQdir}$ such that 
\begin{equation}\label{eqn:LdG_L2_grad_flow_semi-implicit}
\begin{split}
	&\ipOm{\frac{\vQ_{k+1} - \vQ_{k}}{\dt}}{\vP} + \ELdGform{\vQ_{k+1}}{\vP} + \frac{1}{\Bulkcoef} \iO \frac{\partial \Bulkimp (\vQ_{k+1})}{\partial \vQ} \dd \vP \, d\vx \\
	&\quad + \surfcoef \iG \frac{\partial \LdGsurf(\vQ_{k+1})}{\partial \vQ} \dd \vP \, dS(\vx) = \frac{1}{\Bulkcoef} \iO \frac{\partial \Bulkexp (\vQ_{k})}{\partial \vQ} \dd \vP \, d\vx \\
	&\qquad\qquad + \iO \frac{\partial \LdGrhs(\vQ_{k})}{\partial \vQ} \dd \vP \, d\vx, \quad \forall \vP \in \LdGspace{\vzero},
\end{split}
\end{equation}
where the left-hand-side of \eqref{eqn:LdG_L2_grad_flow_semi-implicit} is linear in $\vQ_{k+1}$ and the right-hand-side is explicitly known at each iteration.

Next, we approximate \eqref{eqn:LdG_L2_grad_flow_semi-implicit} by a finite element method, so we introduce some basic notation and assumptions in that regard.  We assume that $\Om \subset \R^{3}$ is discretized by a conforming shape regular triangulation $\Tk_{h} = \{ T_{i} \}$ consisting of simplices, i.e. we define $\Om_{h} := \cup_{T \in \Tk_{h}} T$.  Furthermore, we define the space of continuous piecewise linear functions on $\Om_{h}$:
\begin{equation}\label{eqn:std_P1_FE_space}
\begin{split}
	\V_{h} (\Om_{h}) := \left\{ v \in C^{0} (\Om_{h}) \mid v |_{T} \in \Pk_{1} (T), ~\forall T \in \Tk_{h} \right\},
\end{split}
\end{equation}
where $\Pk_{k} (T)$ is the space of polynomials of degree $\leq k$ on $T$.

Next, we discretize \eqref{eqn:LdG_L2_grad_flow_semi-implicit} by a $\Pk_{1}$ approximation of the $\vQ$ variable denoted $\vQ_{h}$.  With the following notation
\begin{equation}\label{eqn:Q_tensor_components}
\vQ \in \symmtraceless \quad \Leftrightarrow \quad
\vQ =
\begin{bmatrix}
q_{11} & q_{12} & q_{13} \\ 
q_{12} & q_{22} & q_{23} \\ 
q_{13} & q_{23} & q_{33} 
\end{bmatrix}, ~~\text{where } q_{33} := -q_{11} - q_{22},
\end{equation}
we approximate $\vQ$ by $\vQ_{h} \in \Q_{h}$:
\begin{equation}\label{eqn:Q-tensor_FE_space}
\begin{split}
	\Q_{h} (\Om_{h}) := \big{\{} \vP \in C^{0}& (\Om_{h}) \mid \vP \in \symmtraceless, ~  p_{ij} |_{T} \in \Pk_{1} (T), \\
	&\quad 1 \leq i,j, \leq 3, ~\forall T \in \Tk_{h} \big{\}}.
\end{split}
\end{equation}
We point out that enforcing \eqref{eqn:Q_tensor_components} at the mesh nodes guarantees that $\vQ_h(x) \in \Lambda$ for all $x \in \Om$.
Additionally, other bases can be used to represent $\vQ$ \cite{Gartland_MCLC1991}. 
Therefore, the minimization problem \eqref{eqn:Landau-deGennes_model_problem} becomes
\begin{equation}\label{eqn:LdG_min_problem_FE_approx}
\begin{split}
\min_{ \vQ_{h} \in \Q_{h} \cap \LdGspace{\vQhdir}} \ELdG[\vQ_{h}],
\end{split}
\end{equation}
where $\vQhdir := \interp \vQdir$ and $\interp$ denotes the Lagrange interpolation operator.  We find a local minimizer of \eqref{eqn:LdG_min_problem_FE_approx} by solving a finite element approximation of \eqref{eqn:LdG_L2_grad_flow_semi-implicit}, i.e. given $\vQ_{h,k}$, find $\vQ_{h,k+1} \in \Q_{h} \cap \LdGspace{\vQhdir}$ such that 
\begin{equation}\label{eqn:LdG_L2_grad_flow_FE_approx}
\begin{split}
&\ipOm{\frac{\vQ_{h,k+1} - \vQ_{h,k}}{\dt}}{\vP_{h}} + \ELdGform{\vQ_{h,k+1}}{\vP_{h}} + \frac{\BulkA + \Bulkstab}{\Bulkcoef} \ipOm{\vQ_{h,k+1}}{\vP_{h}} \\
&\quad + \surfcoef \ipGm{\vQ_{h,k+1}}{\vP_{h}} = \surfcoef \ipGm{\interp \vQ_{\Gm}}{\vP_{h}} - \frac{1}{2} \ea \ipOm{\interp (\vE \otimes \vE)}{\vP_{h}} \\
&\frac{1}{\Bulkcoef} \ipOm{\left[ \Bulkstab \, \vQ_{h,k} + \BulkB \, \vQ_{h,k}^2 - \BulkC \, \tr (\vQ_{h,k}^2) \, \vQ_{h,k} \right]}{\vP_{h}}, \quad \forall \vP_{h} \in \Q_{h} \cap \LdGspace{\vzero},
\end{split}
\end{equation}
where we have written the scheme more explicitly.  We iterate this procedure until some stopping criteria is achieved. Numerical results for the standard LdG model are given in Section \ref{sec:LdG_half_degree_line_defect} and Section \ref{sec:standard_LdG_vs_uniaxial}.

\subsection{Ericksen}\label{sec:math_ericksen}

The (general) Ericksen model was originally proposed in \cite{Ericksen_ARMA1991}; see also \cite{Virga_book1994} for another description.  In this section, we concentrate on the \emph{one-constant model} of Ericksen and review its theoretical aspects, which can be found in \cite{Ambrosio_MM1990a,Ambrosio_MM1990b,Lin_CPAM1991}.  Moreover, we describe a robust numerical method for finding local minimizers  of the one-constant Ericksen model (see \cite{Nochetto_SJNA2017,Nochetto_JCP2018} for more details).

\subsubsection{Energy Minimization Framework}\label{sec:Erk_energy_min}

We review a few hypotheses required to have a well-posed energy minimization problem, and some key features of the one-constant Ericksen energy. For convenience, we re-state \eqref{eqn:Ericksen_energy_TEMP_one_const} here:
\begin{equation}\label{eqn:Ericksen_energy_one_constant}
\begin{split}
	\Eerkring [s,\vn] &:= \iO s^2 |\nabla \vn|^2 \, d\vx \equiv \ipOm{s \nabla \vn}{s \nabla \vn}, \\
	\Eerkmain [s,\vn] &:= \frac{\kappa}{2} \ipOm{\nabla s}{\nabla s} + \frac{1}{2} \Eerkring [s,\vn], \\
	\Eerkbulk[s] & := \frac{1}{\Bulkcoef} \ipOm{\DWfunc(s)}{1}\\
	\Eerkone [s,\vn] &:= \Eerkmain [s,\vn] + \Eerkbulk[s],
\end{split}
\end{equation}
where $\Eerkmain$ is the ``main'' part of Ericksen's energy.  Note that the double well potential $\DWfunc : (-1/2, 1) \to \R$ is a $C^2$ function that satisfies \cite{Ericksen_ARMA1991,Ambrosio_MM1990a,Lin_CPAM1991}:
\begin{equation}\label{eqn:bulk_potential_properties}
\begin{split}
	&\text{\textbf{1.} } ~ \lim_{s \rightarrow 1} \DWfunc (s) = \lim_{s \rightarrow -1/2} \DWfunc (s) = \infty, \\
	&\text{\textbf{2.} } ~ \DWfunc(0) > \DWfunc(s^*) = \min_{s \in [-1/2, 1]} \DWfunc(s) = 0, \text{ for some } s^* \in (0,1), \\
	&\text{\textbf{3.} } ~ \DWfunc'(0) = 0,
\end{split}
\end{equation}
where $s=0$ is a local maximum of $\DWfunc$.

If $s \neq 0$ and constant, then $\Eerkone [s,\vn]$ effectively reduces to the Oseen-Frank (one-constant) energy $\int_{\Om} |\nabla \vn|^2$.  When $s$ is variable, $\Eerkone [s,\vn]$ avoids singular energies when defects (discontinuities in $\vn$) are present by allowing $s$ to vanish wherever there are defects.  Hence, all defects must be contained in the \emph{singular set} (cf.\eqref{eqn:singular_set})
\begin{align*}
\Sing = \{ x \in \Om :\; s(x) = 0 \}.
\end{align*}

Existence of minimizers was shown in \cite{Ambrosio_MM1990a,Lin_CPAM1991} through the following clever trick.
By introducing the auxiliary variable $\vu := s \vn$, one can rewrite  $\Eerkmain [s,\vn]$ as
\begin{equation}\label{eqn:auxiliary_energy_identity}
	\Eerkmain [ s, \vn ] = \EerkUmain [ s , \vu ] := \frac{1}{2} \iO \left[ (\kappa - 1) | \nabla s |^2 + | \nabla  \vu |^2 \right] d\vx,
\end{equation}
which uses $\nabla \vu = \vn \otimes \nabla s + s \nabla \vn$ and the unit length constraint $| \vn | = 1$.  Thus, the total energy in terms of $s$ and $\vu$ is
\begin{equation}\label{eqn:Ericksen_energy_one_constant_U_var}
\begin{split}
\EerkUone [s,\vu] &= \EerkUmain [s,\vu] + \Eerkbulk[s].
\end{split}
\end{equation}
The advantage of \eqref{eqn:auxiliary_energy_identity} is that it is quadratic in terms of $s$ and $\vu$, which makes the (closed) admissible set of minimizers straightforward to define \cite{Ambrosio_MM1990a,Lin_CPAM1991}:
\begin{equation}\label{eqn:Erk_admissible_class} \begin{split}
	\Admiserk := \{(s,\vn) \in H^1(\Om) \times [L^{\infty}(\Om)]^d : (s,\vu,\vn) & \text{ satisfies \eqref{eqn:Erk_struct_condition},} \\
&	\text{ with } \vu \in [H^{1}(\Om)]^d \},
\end{split} \end{equation}
where
\begin{equation}\label{eqn:Erk_struct_condition}
\vu = s \vn, \quad - 1/2 \leq s \leq 1 \text{ a.e.~in } \Om, \text{ and } \vn \in \Sp^{d-1} \text{ a.e.~in } \Om,
\end{equation}
is called the \emph{structural condition} of $\Admiserk$.  If we write $(s,\vu,\vn)$ in $\Admiserk$, then this is equivalent to $(s,\vn)$ in $\Admiserk$, $\vu$ in $[H^{1}(\Om)]^d$, and $(s,\vu,\vn)$ satisfies \eqref{eqn:Erk_struct_condition}.  Indeed, \eqref{eqn:auxiliary_energy_identity} only holds for $(s,\vu,\vn)$ in $\Admiserk$.  Sometimes, we refer to the identity $\vu = s \vn$ in \eqref{eqn:Erk_struct_condition} as the \emph{cone constraint} for obvious reasons.

Boundary conditions are accounted for by functions $g : \R^d \to \R$, $\vr, \vq : \R^d \to \R^{d}$ so that the following is satisfied.
\begin{hypothesis}[regularity of boundary data] \label{hyp:basic_bdy_data}
There exists $g \in W^{1,\infty}(\R^d)$, $\vr \in [W^{1,\infty}(\R^d)]^{d}$, $\vq \in [L^{\infty}(\R^d)]^d$, such that $(g,\vr,\vq)$ satisfies \eqref{eqn:Erk_struct_condition} on $\R^d$, i.e. $\vr = g \vq$ and $\vq \in \Sp^{d-1}$ a.e. in $\R^d$.  Furthermore, we assume there is a fixed $c_0 > 0$ (small) such that
\begin{equation}\label{eqn:bdy_g_away_from_limits}
	c_0 \leq g \leq 1 - c_0,
\end{equation}
which implies that $\vq \in [W^{1,\infty}( \R^{d} )]^{d}$.
Moreover, let $\bdys$, $\bdyvu$, $\bdyvn$ be open subsets of $\dOm$ on which to enforce Dirichlet conditions for $s$, $\vu$, $\vn$ (respectively), and assume that $\bdyvu = \bdyvn \subset \bdys$.
\end{hypothesis}

The admissible class, with boundary conditions, is given by
\begin{align}\label{eqn:Erk_admissible_class_BC}
	\Admiserk(g,\vq) := \left\{ (s, \vn) \in \Admiserk : ~ s|_{\bdys} = g, \quad \vn |_{\bdyvn} = \vq \right\},
\end{align}
and Hypothesis \ref{hyp:basic_bdy_data} guarantees that setting boundary conditions for $(s,\vn)$ is meaningful.

For technical reasons, we require the following assumption on $\DWfunc$.
\begin{hypothesis}[growth of potential]\label{hyp:bulk_pot}
Let $c_0 > 0$ be taken from Hypothesis \ref{hyp:basic_bdy_data}.
The bulk potential $\DWfunc$ satisfies
\begin{equation}\label{eqn:bulk_potential_assume}
\begin{split}
	\DWfunc(s) &\ge \DWfunc(1-c_0) \quad \text{for } s \ge 1-c_0, \\
    \DWfunc(s) &\ge \DWfunc \left( -\frac{1}{2}+c_0 \right) \quad \text{for } s \le -\frac{1}{2}+c_0.
\end{split}
\end{equation}
which is consistent with property \textbf{1} of $\DWfunc$ in  \eqref{eqn:bulk_potential_properties}.
\end{hypothesis}

The existence of a minimizer $(s,\vn) \in \Admiserk(g,\vq)$ of $\Eerkone [s,\vn]$ is shown in \cite{Ambrosio_MM1990a,Lin_CPAM1991}, but is also a
consequence of the $\Gamma$-convergence theory that we review in Section \ref{sec:Erk_Gamma_convergence}.

\subsubsection{Finite Element Discretization}\label{sec:Erk_FE_discretization}
The Ericksen model is degenerate in the director field $\vn$. This feature, that makes it capable of capturing non-trivial defects, also makes its numerical analysis difficult. Consequently, references on numerical methods for such a model are scarce. We refer to \cite{Barrett_M2AN2006,Calderer_SJMA2002} and to \cite{Nochetto_SJNA2017,Nochetto_JCP2018} for finite element approximations to dynamics and equilibrium configurations, respectively.

In this section we review \cite[Sec. 2.2]{Nochetto_SJNA2017}.  First, discretize $\Om$ as we did in Section \ref{sec:LdG_numerics}, i.e. $\Om \subset \R^{d}$ is approximated by $\Om_{h}$ which comes from a conforming shape-regular mesh $\Tk_{h} = \{ T_{i} \}$ consisting of simplices.  For simplicity, we assume that $\Om \equiv \Om_{h}$, i.e. that there is no geometric error caused by the triangulation.  Furthermore, let $\Nk_{h}$ be the set of nodes (vertices) of $\Tk_{h}$ and, with some abuse of notation, let $N$ be the cardinality of $\Nk_{h}$.

Next, define continuous piecewise linear finite element spaces on $\Om$:
\begin{equation}\label{eqn:Erk_discrete_spaces}
\begin{split}
	\Sh &:= \{ s_h \in H^1(\Om) : s_h |_{T} \in \Pk_{1} (T), \forall T \in \Tk_h \}, \\
	\Uh &:= \{ \vu_h \in [H^1(\Om)]^d : \vu_h |_{T} \in \Pk_{1} (T), \forall T \in \Tk_h \}, \\
	\Nh &:= \{ \vn_h \in \Uh : |\vn_h(\vx_{i})| = 1, \forall \vx_{i} \in \Nk_h \},
\end{split}
\end{equation}
where the unit length constraint is enforced in $\Nh$ at the nodes (vertices) of the mesh. Dirichlet boundary conditions are included via the following discrete spaces (recall Hypothesis \ref{hyp:basic_bdy_data}):
\begin{equation}\label{eqn:Erk_discrete_spaces_BC}
\begin{split}
	\Sh (\bdys,g_h) &:= \{ s_h \in \Sh : s_h |_{\bdys} = g_h \}, \\
	\Uh (\bdyvu,\vr_h) &:= \{ \vu_h \in \Uh : \vu_h |_{\bdyvu} = \vr_h \}, \\
	\Nh (\bdyvn,\vq_h) &:= \{ \vn_h \in \Nh : \vn_h |_{\bdyvn} = \vq_h \},
\end{split}
\end{equation}
where $g_h := \interp g$, $\vr_h := \interp \vr$, and $\vq_h := \interp \vq$ is the discrete Dirichlet data.  This leads to the following discrete admissible class with boundary conditions:
\begin{equation}\label{eqn:Erk_admissible_class_BC_discrete}
\begin{split}
	\Admiserk^{h}(g_h,\vq_h) := \big\{ & (s_h,\vn_h) \in \Sh (\bdys, g_h) \times \Nh (\bdyvn,\vq_h) : \\
	&(s_h,\vu_h,\vn_n) \text{ satisfies \eqref{eqn:Erk_struct_condition_discrete}}, \text{ with } \vu_h \in \Uh (\bdyvu, \vr_h) \big\},
\end{split}
\end{equation}
where
\begin{equation}\label{eqn:Erk_struct_condition_discrete}
\vu_h = \interp (s_h \vn_h), \quad - 1/2 \leq s_h \leq 1 \text{ in } \Om, \quad \text{and } |\vn_h(\vx_{i})|=1, \forall \vx_{i} \in \Nk_{h},
\end{equation}
is called the \emph{discrete structural condition} of $\Admiserk^{h}$.  If we write $(s_h,\vu_h,\vn_h)$ in $\Admiserk^{h}$, then this is equivalent to $(s_h,\vn_h)$ in $\Admiserk^h$, $\vu_h$ in $\Uh$, and $(s_h,\vu_h,\vn_h)$ satisfies \eqref{eqn:Erk_struct_condition_discrete}.  We emphasize that the approximation we are considering is not conforming. Indeed, the inclusion $\Admiserk^h \subset \Admiserk$ fails because, at the discrete level, we only impose the structural condition $\vu = s \vn$ at the mesh nodes.

The discretization of $\Eerkring [s,\vn]$ in \eqref{eqn:Ericksen_energy_one_constant} is non-standard because of the delicate nature of the degenerate term $s^2 |\nabla \vn|^2$.  In fact, this requires us to make an additional assumption on the meshes. We shall denote by $\phi_i$ the standard piecewise linear ``hat'' function associated with a node $\vx_{i} \in \Nk_h$, so that $\{ \phi_i \}$ are basis functions of the spaces in \eqref{eqn:Erk_discrete_spaces}. Moreover, we indicate with $\omega_i = \text{supp} \;\phi_i$ the patch of a node $\vx_{i}$ (i.e. the ``star'' of elements in $\Tk_{h}$ that contain the vertex $\vx_{i}$).

\begin{hypothesis}[weak acuteness]\label{hyp:weakly-acute}
For all $h>0$, the mesh $\Tk_{h}$ is \emph{weakly acute}:
\begin{equation}\label{eqn:weakly-acute}
	k_{ij} := -\int_{\Om} \nabla \phi_i \cdot \nabla \phi_j dx \ge 0
\quad\text{for all } i \ne j.
\end{equation}
\end{hypothesis}

Condition \eqref{eqn:weakly-acute} imposes a severe geometric restriction
on $\Tk_{h}$ \cite{Ciarlet_CMAME1973,Strang_FEMbook2008}. We recall the
following characterization of \eqref{eqn:weakly-acute} for $d=2$.
\begin{proposition}[weak acuteness in two dimensions]
\label{prop:weak_acuteness_2D}
For any pair of triangles $T_1$, $T_2$ in $\Tk_h$ that share a common
edge $e$, let $\alpha_i$ be the angle in $T_i$ opposite to $e$ (for $i=1,2$).
If $\alpha_1 + \alpha_2 \leq \pi$ for every edge $e$, then
\eqref{eqn:weakly-acute} holds.
\end{proposition}
Generalizations of Proposition \ref{prop:weak_acuteness_2D} to three dimensions involve the interior dihedral angles of tetrahedra \cite{Brandts_LAA2008, Korotov_MC2001}. We also point out that a non-obtuse tetrahedral mesh is automatically weakly acute.

We now motivate our discretization of $\Eerkring [ s , \vn]$.  Note that for all $\vx_{i} \in \Nk_h$
\begin{equation*}
	\sum_{j=1}^N k_{ij} = - \sum_{j=1}^N \iO \nabla \phi_i \cdot \nabla \phi_j dx = 0,
\end{equation*}
because $\sum_{j=1}^N \phi_j = 1$ in the domain $\Om$ (i.e. $\{\phi_j\}_{j=1}^N$ is a partition of unity).  So, if $\Sh \ni s_h = \sum_{i=1}^N s_h(\vx_{i}) \phi_i$, then
\begin{align*}
	\iO | \nabla s_h |^2 dx = -\sum_{i=1}^N k_{ii} [s_h(\vx_{i})]^2 - \sum_{i, j = 1, i \neq j}^N k_{ij} s_h(\vx_{i}) s_h(\vx_{j}),
\end{align*}
and using $k_{ii} = - \sum_{j \neq i} k_{ij}$ and the symmetry $k_{ij}=k_{ji}$, we get
\begin{equation}\label{eqn:dirichlet_integral_identity}
\begin{aligned}
\int_{\Om} | \nabla s_h |^2 dx &= \sum_{i, j = 1}^N k_{ij} s_h(\vx_{i}) \big(s_h(\vx_{i}) - s_h(\vx_{j})\big)
\\
&= \frac{1}{2} \sum_{i, j = 1}^N k_{ij} \big(s_h(\vx_{i}) - s_h(\vx_{j})\big)^2 = \frac{1}{2} \sum_{i, j = 1}^N k_{ij} \big( \dij s_h \big)^2,
\end{aligned}
\end{equation}
where we define
\begin{equation}\label{eqn:delta_ij}
	\dij s_h := s_h(\vx_{i}) - s_h(\vx_{j}), \quad \dij \vn_h := \vn_h(\vx_{i}) - \vn_h(\vx_{j}).
\end{equation}
On the other hand, we discretize $\Eerkring [ s , \vn]$ by
\begin{equation}\label{eqn:Eerkring_discrete}
\begin{split}
	\Eerkring^h [s_h, \vn_h] &:= \frac{1}{2} \sum_{i, j = 1}^N k_{ij} \left(\frac{s_h(\vx_{i})^2 + s_h(\vx_{j})^2}{2}\right) |\dij \vn_h|^2,
\end{split}
\end{equation}
and the main part of the Ericksen energy by
\begin{equation}\label{eqn:Erk_main_discrete}
\begin{split}
	\Eerkmain^h [s_h, \vn_h] &:= \frac{\kappa}{2} \ipOm{\nabla s_h}{\nabla s_h} + \frac{1}{2} \Eerkring^h [s_h,\vn_h], \\
	&= \frac{\kappa}{2} \left(\frac{1}{2} \sum_{i, j = 1}^N k_{ij} \left( \dij s_h \right)^2 \right) + \frac{1}{2} \Eerkring^h [s_h,\vn_h].
\end{split}
\end{equation}
Equation \eqref{eqn:Eerkring_discrete} does \emph{not} come from applying the standard discretization of $\iO s^2 |\nabla  \vn |^2 dx$ by piecewise linear finite elements (though it is a first order approximation).  This special discretization of the energy preserves an important energy inequality (see Lemma \ref{lem:energydecreasing}), which is crucial to proving the $\Gamma$-convergence of our discrete energy with the degenerate coefficient $s^2$ \emph{without} regularization.

The double-well energy is discretized in the usual way,
\begin{equation}\label{discrete_energy_E2}
  \Eerkbulk^h [s_h] := \frac{1}{\Bulkcoef} \ipOm{\DWfunc(s_h)}{1} = \frac{1}{\Bulkcoef} \int_{\Om} \DWfunc (s_h(x)) dx.
\end{equation}

Therefore, our discrete minimization problem for the Ericksen model is as follows.  Find $(s_h^{*}, \vn_h^{*}) \in \Admiserk^{h}(g_h,\vq_h)$ such that
\begin{equation}\label{eqn:Erk_discrete_minimization}
\begin{split}
	(s_h^{*}, \vn_h^{*}) = \argmin_{(s_h, \vn_h) \in \Admiserk^{h}(g_h,\vq_h)} \Eerkone^h [s_h, \vn_h],
\end{split}
\end{equation}
where
\begin{equation}\label{eqn:Erk_total_discrete}
\begin{split}
	\Eerkone^h [s_h, \vn_h] &:= \Eerkmain^h [s_h, \vn_h] +\Eerkbulk^h [s_h].
\end{split}
\end{equation}

We close with a result showing that \eqref{eqn:Erk_main_discrete} preserves the key structure of \eqref{eqn:auxiliary_energy_identity} at the discrete level, and is a key component of the $\Gamma$-convergence of the method. First, we recall that $\interp$ denotes the Lagrange interpolation operator and introduce $\widetilde{s}_h := \interp|s_h|$ and two discrete versions of the vector field $\vu$,
\begin{equation}\label{eqn:vu_discrete}
	\vu_h := \interp(s_h \vn_h) \in \Uh, \quad \widetilde{\vu}_h := \interp (\widetilde{s}_h \vn_h) \in \Uh.
\end{equation}
Note that both triplets $(s_h, \vu_h, \vn_h), (\widetilde{s}_h, \widetilde{\vu}_h, \vn_h) \in \Sh \times \Uh \times \Nh$ satisfy \eqref{eqn:Erk_struct_condition_discrete}.  The following is taken from \cite[Lem. 2.2]{Nochetto_SJNA2017}.
\begin{lemma}[energy inequalities]\label{lem:energydecreasing}
Let the mesh $\Tk_{h}$ satisfy \eqref{eqn:weakly-acute}.
If $(s_h,\vn_h) \in \Admiserk^h(g_h,\vq_h)$, then,
for any $\kappa > 0$, the discrete energy \eqref{eqn:Erk_main_discrete} satisfies
\begin{equation}\label{energy_inequality}
	\Eerkmain^h [s_h, \vn_h] \geq (\kappa - 1) \iO |\nabla s_h|^2 dx + \iO |\nabla \vu_h|^2 dx =: \EerkUmain^h [s_h, \vu_h],
\end{equation}
and
\begin{equation}\label{abs_inequality}
	\Eerkmain^h [s_h, \vn_h] \geq (\kappa - 1) \iO |\nabla \widetilde{s}_h|^2 dx + \iO |\nabla \widetilde{\vu}_h|^2 dx =: \EerkUmain^h [\widetilde{s}_h,\widetilde\vu_h].
\end{equation}
\end{lemma}

\subsubsection{Continuous Gradient Flow}\label{sec:Erk_contin_gradient_flow}

We begin with a formal derivation of a gradient flow to find a local minimizer of $\Eerkone [s,\vn]$ in \eqref{eqn:Ericksen_energy_one_constant}.  Since $s$ and $\vn$ are coupled variables, we will have two coupled gradient flows.  

First, let $\ips{\cdot}{\cdot} : H^1(\Om) \times H^1(\Om) \to \R$ be a bounded bilinear form (inner product) for $s$.  For the sake of exposition, assume an $L^2(\Om)$ inner product, e.g. $\ips{s}{z} := \ipOm{s}{z}$, but other choices can be made.
Similar to \eqref{eqn:LdG_L2_grad_flow}, we define a gradient flow for $s$:
\begin{equation}\label{eqn:Erk_grad_flow_s}
\begin{split}
	\ips{\partial_{t} s(\cdot,t)}{z} = -\delta_{s} \Eerkone [s,\vn; z], \quad \forall z \in H^1_{\bdys} (\Om),
\end{split}
\end{equation}
where $\Hbdy{\bdys} := \{ z \in H^1(\Om) : z|_{\bdys} = 0 \}$ preserves the boundary condition for $s$, and the first variation is given by
\begin{equation}\label{eqn:Erk_energy_first_variation_s}
\begin{split}
	\delta_{s} \Eerkone [s,\vn; z] &= \kappa \iO \nabla s \cdot \nabla z \, d\vx + \iO s z |\nabla \vn|^2 \, d\vx + \frac{1}{\Bulkcoef} \iO \DWfunc'(s) z \, d\vx.
\end{split}
\end{equation}
Applying a formal integration by parts to \eqref{eqn:Erk_grad_flow_s} gives
\begin{equation}\label{eqn:Erk_grad_flow_s_int_by_parts}
	\iO \partial_t s z \, d\vx = - \iO \left( - \kappa \Delta s + |\nabla \vn|^2 s + \frac{1}{\Bulkcoef} \DWfunc'(s) \right) z \, d\vx, ~\text{ for all } z \in \Hbdy{\bdys},
\end{equation}
where we use the implicit Neumann condition $\vnu \cdot \nabla s = 0$
on $\dOm \setminus \bdys$, with $\vnu$ being the outer unit normal of $\dOm$.
Hence, $s$ satisfies the (nonlinear) parabolic PDE:
\begin{equation}\label{eqn:Erk_grad_flow_s_strong_form}
\begin{split}
	\partial_t s - \kappa \Delta s + |\nabla \vn|^2 s + \frac{1}{\Bulkcoef} \DWfunc'(s) &= 0, ~\text{ in } \Om, \\
	s = g, ~\text{ on } \bdys, \quad \vnu \cdot \nabla s &= 0, ~\text{ on } \dOm \setminus \bdys.
\end{split}
\end{equation}

Next, given $(s,\vn) \in \Admiserk(g,\vq)$, we consider the space of tangential variations of $\vn$:
\begin{align}\label{eqn:tangent_variation_space}
	\Vperp (s,\vn) &:= \left\{ \vv \in [L^2(\Om)]^d : \vv \cdot \vn = 0 \text{ a.e. in } \Om, \text{ and } s \vv \in [H^1(\Om)]^d \right\},
\end{align}
which is connected with the constraint $\vn \in \Sp^{d-1}$ in the following sense.  If $\vn \equiv \vn(\vx,t)$ is evolving director field such that $|\vn|^2=1$, then
\begin{equation*}
	0 = \partial_{t} |\vn|^2 = 2 \partial_{t} \vn \cdot \vn \quad \Rightarrow \quad \partial_{t} \vn \in \Vperp (s,\vn).
\end{equation*}
Indeed, introducing a tangential perturbation of $\vn$: $\vp(t) = \vn + t \vv$ where $\vv \in \Vperp (s,\vn)$, we have that $|\vp|^2 = 1 + t^2 |\vv|$, which preserves $|\vn|=1$ up to second order in $t$. We remark that if $s \geq c_0 > 0$ in $\Om$,  then $\vv \in \Vperp (\vn)$ is necessarily in $[H^1(\Om)]^d$.

Let $\ipvt{\cdot}{\cdot} : \Vperp (s,\vn) \times \Vperp (s,\vn) \to \R$ be a bounded bilinear form (inner product) for tangential variations of $\vn$.  For the sake of exposition, assume an $L^2(\Om)$ inner product, e.g. $\ipvt{\vt}{\vv} := \ipOm{\vt}{\vv}$, but other choices can be made.
Similar to \eqref{eqn:Erk_grad_flow_s}, we define a gradient flow for $\vn$:
\begin{equation}\label{eqn:Erk_grad_flow_vn}
\begin{split}
	\ipvt{\partial_{t} \vn(\cdot,t)}{\vv} = -\delta_{\vn} \Eerkone [s,\vn; \vv], \quad \forall \vv \in \Vperp (s,\vn) \cap [\Hbdy{\bdyvn}]^{d},
\end{split}
\end{equation}
where $\partial_{t} \vn = 0$ on $\bdyvn$ preserves the boundary condition for $\vn$, and the first variation is given by
\begin{equation}\label{eqn:Erk_energy_first_variation_vn}
\begin{split}
\delta_{\vn} \Eerkone [s,\vn; \vv] &= \iO s^2 \nabla \vn \cdot \nabla \vv \, d\vx.
\end{split}
\end{equation}
Applying a formal integration by parts to \eqref{eqn:Erk_grad_flow_vn} gives
\begin{equation}\label{eqn:Erk_grad_flow_vn_int_by_parts}
	\iO \partial_t \vn \cdot \vv \, d\vx = - \iO -\nabla \cdot ( s^2 \nabla \vn) \cdot \vv \, d\vx, \quad \forall \vv \in \Vperp (s,\vn) \cap [\Hbdy{\bdyvn}]^{d},
\end{equation}
where we use the implicit Neumann condition $\vnu \cdot \nabla \vn = \vzero$
on $\dOm \setminus \bdyvn$.  Hence, $\vn$ satisfies the (nonlinear) \emph{degenerate} parabolic PDE:
\begin{equation}\label{eqn:Erk_grad_flow_vn_strong_form}
\begin{split}
	\partial_t \vn - \nabla \cdot ( s^2 \nabla \vn) &= \vzero, ~\text{ in } \Om, \\
	\vn = \vq, ~\text{ on } \bdyvn, \quad \vnu \cdot \nabla \vn &= \vzero, ~\text{ on } \dOm \setminus \bdyvn.
\end{split}
\end{equation}

Assuming $(s,\vn)$ evolve according to \eqref{eqn:Erk_grad_flow_s_strong_form}, \eqref{eqn:Erk_grad_flow_vn_strong_form}, we have that
\begin{equation}\label{eqn:Erk_grad_flow_energy_decrease}
\begin{split}
	\partial_{t} \Eerkone [s,\vn] &= \delta_{s} \Eerkone [s,\vn; \partial_{t} s] + \delta_{\vn} \Eerkone [s,\vn; \partial_{t} \vn], \\
	&= - \ips{\partial_{t} s}{\partial_{t} s} - \ipvt{\partial_{t} \vn}{\partial_{t} \vn} \leq 0,
\end{split}
\end{equation}
and therefore the energy is monotonically decreasing.

\subsubsection{Discrete Gradient Flow}\label{sec:Erk_discrete_gradient_flow}
Here we discuss a discrete quasi-gradient flow algorithm, as described in \cite[Section 4.2.2]{Nochetto_SJNA2017}. Let $(s_{h}^{k}, \vn_{h}^{k}) \in \Admiserk^{h}(g_h,\vq_h)$ where $k$ indicates a ``time-step'' index.
To simplify notation, we write
\begin{equation*}
s_i^k := s_h^k(\vx_{i}),
\quad
\vn_i^k := \vn_h^k(\vx_{i}),
\quad
z_i := z_h(\vx_{i}),
\quad
\vv_i := \vv_h(\vx_{i}).
\end{equation*}

Using a fully implicit, backward Euler time discretization for $\partial_{t} s$, and applying the finite element space discretization in Section \ref{sec:Erk_FE_discretization}, we discretize \eqref{eqn:Erk_grad_flow_s} by
\begin{equation}\label{eqn:Erk_grad_flow_s_discrete}
\begin{split}
	\ips{\frac{s_{h}^{k+1} - s_{h}^{k}}{\dt}}{z_{h}} = -\delta_{s} \Eerkone^h [s_{h}^{k+1}, \vn_{h}^{k+1}; z_{h}], \quad \forall z_{h} \in \Sh(\bdys,0),
\end{split}
\end{equation}
where $\dt > 0$ is a finite time step. The discrete variational derivative is given by
\begin{equation}\label{eqn:Erk_energy_first_variation_s_discrete}
\begin{split}
	\delta_{s} \Eerkone^h [s_{h}^{k}, \vn_{h}^{k}; z_{h}] &= \kappa \ipOm{\nabla s_{h}^{k}}{\nabla z_{h}} + \frac{1}{2} \delta_{s} \Eerkring^h [s_{h}^{k}, \vn_{h}^{k}; z_{h}] + \frac{1}{\Bulkcoef} \ipOm{\DWfunc'(s_{h}^{k})}{z_{h}}, \\
	\delta_{s} \Eerkring^h [s_{h}^{k}, \vn_{h}^{k}; z_{h}] &= \sum_{i, j = 1}^N k_{ij} |\dij \vn_{h}^{k}|^2 \left(\frac{s_i^k z_i+ s_j^k z_j}{2}\right).
\end{split}
\end{equation}

Next, we define the discrete version of \eqref{eqn:tangent_variation_space}:
\begin{equation}\label{eqn:discrete_tangent_variation_space}
\begin{split}
	\Vperp_{h} (\vn_{h}) &= \{ \vv_h \in \Uh : \vv_h(\vx_{i}) \cdot \vn_h(\vx_{i}) = 0, \text{ for all nodes } \vx_{i} \in \Nk_h \},
\end{split}
\end{equation}
Using a ``linearized'' backward Euler time discretization for $\partial_{t} \vn$, and the finite element space discretization in Section \ref{sec:Erk_FE_discretization}, we discretize \eqref{eqn:Erk_grad_flow_vn} by
\begin{equation}\label{eqn:Erk_grad_flow_vn_discrete}
\begin{split}
	\ipvt{\frac{\vn_{h}^{k+1} - \vn_{h}^{k}}{\dt}}{\vv_{h}} &= -\delta_{\vn} \Eerkone^{h} [s_{h}^{k+1}, \vn_{h}^{k+1}; \vv_{h}], ~~ \forall \vv_{h} \in \Vperp_{h} (\vn_{h}^{k}) \cap \Uh(\bdyvu,\vzero), \\
	\text{subject to } |\vn_{h}^{k+1} &(\vx_{i})|=1, \text{ for all nodes } \vx_{i} \in \Nk_h,
\end{split}
\end{equation}
where the discrete variational derivative is given by
\begin{equation}\label{eqn:Erk_energy_first_variation_vn_discrete}
\begin{split}
	\delta_{\vn} \Eerkone^h [s_{h}^{k}, \vn_{h}^{k}; \vv_{h}] &= \frac{1}{2} \delta_{\vn} \Eerkring^h [s_{h}^{k}, \vn_{h}^{k}; \vv_{h}], \\
	\delta_{\vn} \Eerkring^h [s_{h}^{k}, \vn_{h}^{k}; \vv_{h}] &= \sum_{i, j = 1}^N k_{ij} \left(\frac{(s_i^k)^2 + (s_j^k)^2 }{2}\right) ( \dij \vn_h^k ) \cdot ( \dij \vv_h ).
\end{split}
\end{equation}

Thus, a \emph{possible} algorithm is the following.  Given $(s_{h}^{k}, \vn_{h}^{k}) \in \Admiserk^{h}(g_h,\vq_h)$, solve \eqref{eqn:Erk_grad_flow_s_discrete}, \eqref{eqn:Erk_grad_flow_vn_discrete} \emph{simultaneously} to obtain $(s_{h}^{k+1}, \vn_{h}^{k+1}) \in \Admiserk^{h}(g_h,\vq_h)$.  Starting from an initial guess $(s_{h}^{0}, \vn_{h}^{0}) \in \Admiserk^{h}(g_h,\vq_h)$, we iterate this until some convergence criteria is achieved.

Unfortunately, this is a fully coupled non-linear system of equations with a non-convex constraint $|\vn_{h}(\vx_i)|=1$.
Therefore, we adopt to split the gradient flow iteration into three sequential steps.  In order to obtain a monotone, energy decreasing scheme, we first employ a convex splitting approach \cite{Wise_SJNA2009,Shen_DCDS2010,Shen_SJSC2010} for the double well potential $\DWfunc$, i.e. we write it as a difference of two convex functions $\DWfuncimp$, $\DWfuncexp$:
\begin{equation}\label{eqn:Erk_DW_pot_convex_split}
	\DWfunc (s) = \DWfuncimp(s) - \DWfuncexp(s).
\end{equation}
With this, and recalling \eqref{eqn:Erk_grad_flow_s_discrete}, \eqref{eqn:Erk_energy_first_variation_s_discrete}, we make the following approximation
\begin{equation}\label{eqn:Erk_DW_variation_convex_split}
	\ipOm{\DWfunc'(s_{h}^{k+1})}{z_{h}} := \ipOm{\DWfuncimp'(s_{h}^{k+1})}{z_{h}} - \ipOm{\DWfuncexp'(s_{h}^{k})}{z_{h}}.
\end{equation}
The following result is from \cite[Lem. 4.1]{Nochetto_SJNA2017}.
\begin{lemma}[convex-concave splitting] \label{lem:Erk_convex_splitting}
For any $s_h^{k}$ and $s_h^{k+1}$ in $\Sh$, \eqref{eqn:Erk_DW_variation_convex_split} implies
\begin{equation}\label{eqn:Erk_convex_splitting}
	\ipOm{\DWfunc(s_{h}^{k+1})}{1} - \ipOm{\DWfunc(s_{h}^{k})}{1} \leq \ipOm{\DWfunc'(s_{h}^{k+1})}{s_h^{k+1} - s_h^k}.
\end{equation}
\end{lemma}

We can now formulate our alternating direction, discrete gradient flow algorithm.  Given $(s_{h}^{0}, \vn_{h}^{0}) \in \Admiserk^{h}(g_h,\vq_h)$, iterate Steps \textbf{1}-\textbf{3} for $k \ge 0$:
\begin{enumerate}
	\item \textbf{Tangential flow for $\vn_{h}$.} First, linearize $\vn_{h}^{k+1}$ by $\widetilde{\vn}_{h}^{k+1} := \vn_{h}^{k} + \vt_{h}^{k}$, for some $\vt_h^k \in \Vperp_{h} (\vn_h^k) \cap \Hbdy{\bdyvn}$ to be determined.  Note that $\widetilde{\vn}_{h}^{k+1} \notin \Nh$.  Next, choose $\ipkvt{\vt_{h}}{\vv_{h}} := \delta_{\vn} \Eerkring^h [s_{h}^{k}, \vt_{h}; \vv_{h}]$ which is an effective (discrete) inner product on $\Vperp_{h} (\vn_h^k)$.
	
Then, assuming time step is unity, we replace \eqref{eqn:Erk_grad_flow_vn_discrete} by: find $\vt_{h}^{k} \in \Vperp_{h} (\vn_h^k) \cap \Hbdy{\bdyvn}$ such that
\begin{equation}\label{eqn:Erk_grad_flow_vn_discrete_ADI}
\begin{split}
\ipkvt{\vt_{h}^{k}}{\vv_{h}} &= -\delta_{\vn} \Eerkring^{h} [s_{h}^{k}, \vn_{h}^{k}; \vv_{h}], ~~ \forall \vv_{h} \in \Vperp_{h} (\vn_{h}^{k}) \cap \Uh(\bdyvu,\vzero).
\end{split}
\end{equation}

	\item \textbf{Projection.} Define $\vn_{h}^{k+1} \in \Nh (\bdyvn,\vq_{h})$ by
\begin{equation}\label{eqn:Erk_grad_flow_projection}
\begin{split}
	\vn_i^{k+1} := \frac{\widetilde{\vn}_i^{k+1}}{|\widetilde{\vn}_i^{k+1}|} \equiv \frac{ \vn_i^k + \vt_i^k } { \left| \vn_i^k + \vt_i^k \right| }, ~ \text{ at all nodes } x_i \in \Nk_h.
\end{split}
\end{equation}
	
	\item \textbf{Gradient flow for $s_{h}$.} Using $(s_h^k, \vn_h^{k+1})$, find $s_h^{k+1}$ in $\Sh (\bdys, g_h)$ such that
\begin{equation}\label{eqn:Erk_grad_flow_s_discrete_ADI}
\begin{split}
	\ipOm{\frac{s_h^{k+1} - s_h^{k}}{\dt}}{z_{h}} &= -\delta_{s} \Eerkone^h [s_{h}^{k+1}, \vn_{h}^{k+1}; z_{h}], \quad \forall z_{h} \in \Sh(\bdys, 0), \\
	&= -\kappa \ipOm{\nabla s_{h}^{k+1}}{\nabla z_{h}} - \frac{1}{2} \delta_{s} \Eerkring^h [s_{h}^{k+1}, \vn_{h}^{k+1}; z_{h}] \\
	&\qquad\quad - \frac{1}{\Bulkcoef} \ipOm{\DWfunc'(s_{h}^{k+1})}{z_{h}}, \quad \forall z_{h} \in \Sh(\bdys, 0).
\end{split}
\end{equation}
\end{enumerate}

The following result, taken from \cite[Thm. 4.2]{Nochetto_SJNA2017}, shows the robustness of this algorithm.
\begin{theorem}[energy decrease]
\label{thm:Erk_monotone_energy_decrease}
Let $\Tk_h$ satisfy \eqref{eqn:weakly-acute}. The iterate
$(s_h^{k+1}, \vn_h^{k+1})$ of the above algorithm exists and satisfies
\begin{equation}\label{eqn:Erk_monotone_energy_decrease}
\Eerkone^h [s_h^{k+1} ,\vn_h^{k+1}] \leq \Eerkone^h [s_h^k ,\vn_h^k ] - \frac{1}{\dt} \int_{\Omega} (s_h^{k+1} - s_h^k)^2 dx
\end{equation}
Equality holds if and only if $(s_h^{k+1}, \vn_h^{k+1}) = (s_h^k,\vn_h^k)$.
\end{theorem}

After summing in $k$ in \eqref{eqn:Erk_monotone_energy_decrease}, the estimate
\[
\Eerkone^h [s_h^{K} ,\vn_h^{K}] \leq \Eerkone^h [s_h^{0} ,\vn_h^{0}] - \frac{1}{\dt} \sum_{k = 0}^{K-1} \| s_h^{k+1} - s_h^k \|_{L^2(\Om)}^2
\]
follows immediately. Therefore, if we set as termination criterion that $ \| s_h^{k+1} - s_h^k \|_{L^2(\Om)} < \eps$ for some prescribed tolerance $\eps > 0$, the algorithm must finish in a finite number of iterations.

\begin{remark}[projection is energy-decreasing] \label{rmk:energy-decrease}
In order to guarantee that Step \textbf{2} above is energy decreasing, we need to use the weak-acuteness assumption on the mesh from Hypothesis \ref{hyp:weakly-acute}, cf. \cite{Alouges_SJNA1997,Bartels_SJNA2006,Nochetto_SJNA2017}.
\end{remark}

\begin{example} \label{ex:Erk_energy_evolution}
We illustrate the energy monotonicity with a computational experiment. We consider the square $\Omega = (0,1)^2$ and set $\kappa = 1$ and  $\Bulkcoef = 1$ in the Ericksen energy. The double-well potential we consider is such that its convex splitting (recall \eqref{eqn:Erk_DW_pot_convex_split}) is
\begin{equation} \label{eqn:Erk_DW_potential} \begin{split}
\DWfunc (s) & = \psi_c(s) - \psi_e(s) \\ 
&  := 63.0 s^2 - (-16 s^4 + 21.33333333333333 s^3 + 57 s^2) .
\end{split} \end{equation}
We point out that $\DWfunc$ has a local minimum at $s = 0$ and a global minimum at $s = s^* := 0.750025$. We impose Dirichlet boundary conditions for both $s$ and $\vn$ on $\Gamma_s = \Gamma_\vn = \dOm$,
\begin{equation}\label{eqn:plus_3_degree_defect}
s = s^*, ~~ \vn(x,y) = (\cos \theta, \sin \theta)^\top, 
~~ \theta(x,y) = 3 \, \atan \left(x-0.3, y-0.6 \right),
\end{equation}
where $\atan : \R^2 \setminus \{0\} \to [-\pi, \pi]$ is the four quadrant inverse tangent function, namely, $\atan(x_0,y_0)$ is the angle between the positive $x$-axis and the ray to the point $(x_0,y_0)$.  This gives a defect of degree $+3$.

We consider the gradient flow algorithm discussed in Section \ref{sec:Erk_discrete_gradient_flow} with time step $\dt = 10^{-1}$, initialized with $(s_h^0, \vn_h^0) \equiv (s^*, (1,0)^\top)$ in the interior nodes, and with stopping criterion $\| s_h^{k+1} - s_h^k \|_{L^2(\Om)} < 10^{-8}$. Figure \ref{fig:Erk_energy_decrease} illustrates the energy monotonicity property of our algorithm, that finishes in $328$ steps. Figure \ref{fig:Erk_evolution_gradient_flow} shows the evolution of the iterates at some steps in the algorithm. We obtain an equilibrium configuration in which three point defects are present in the domain. More precisely, at the final configuration the degree of orientation $s_h$ reaches local minima approximately at $s_h(0.20,0.78) \approx 8.0 \times 10^{-3}$, $s_h(0.31,0.35)\approx 4.9 \times 10^{-3}$ and $s_h(0.63,0.58)\approx 3.3 \times 10^{-3}$.

\begin{figure}[ht]
\includegraphics[width=1\linewidth]{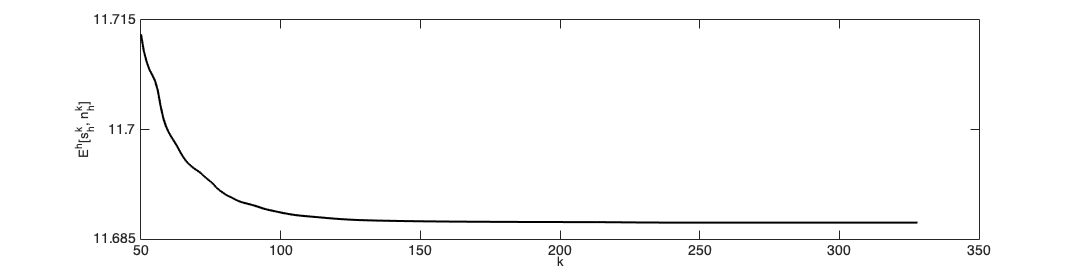}
\caption{Evolution of the discrete Ericksen energy $\Eerkone^h[s_h^k,\vn_h^k]$ in Example \ref{ex:Erk_energy_evolution}.
}
\label{fig:Erk_energy_decrease}
\end{figure}

\begin{figure}[ht]
\includegraphics[width=0.33\linewidth]{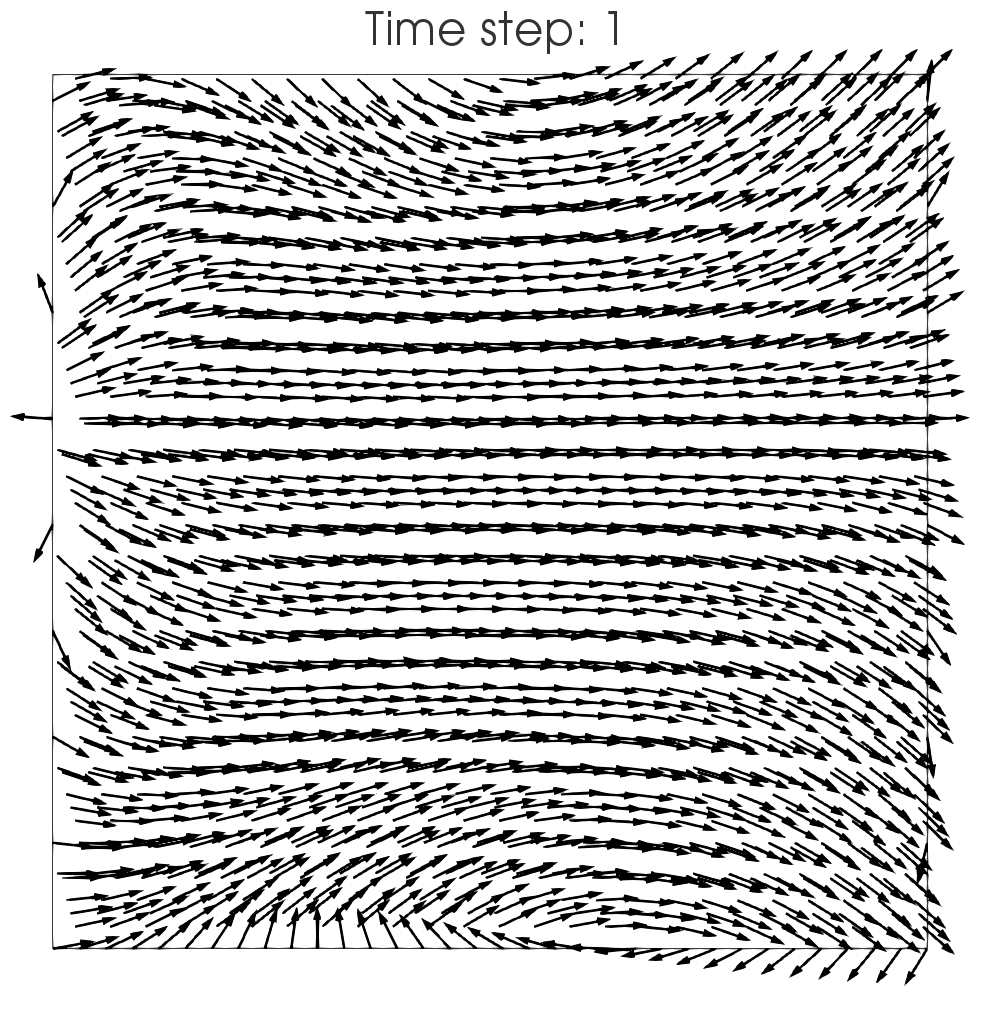} \hspace{-0.3cm}
\includegraphics[width=0.33\linewidth]{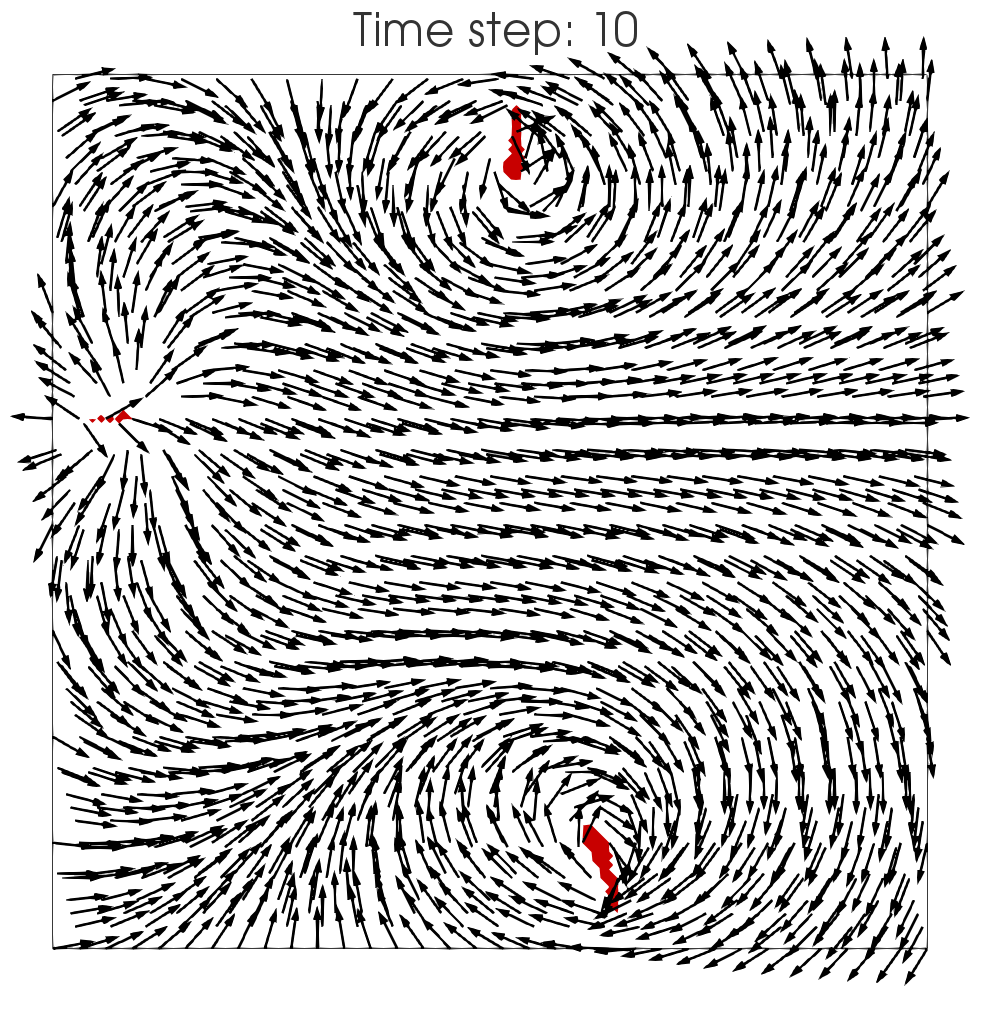} \hspace{-0.3cm} 
\includegraphics[width=0.33\linewidth]{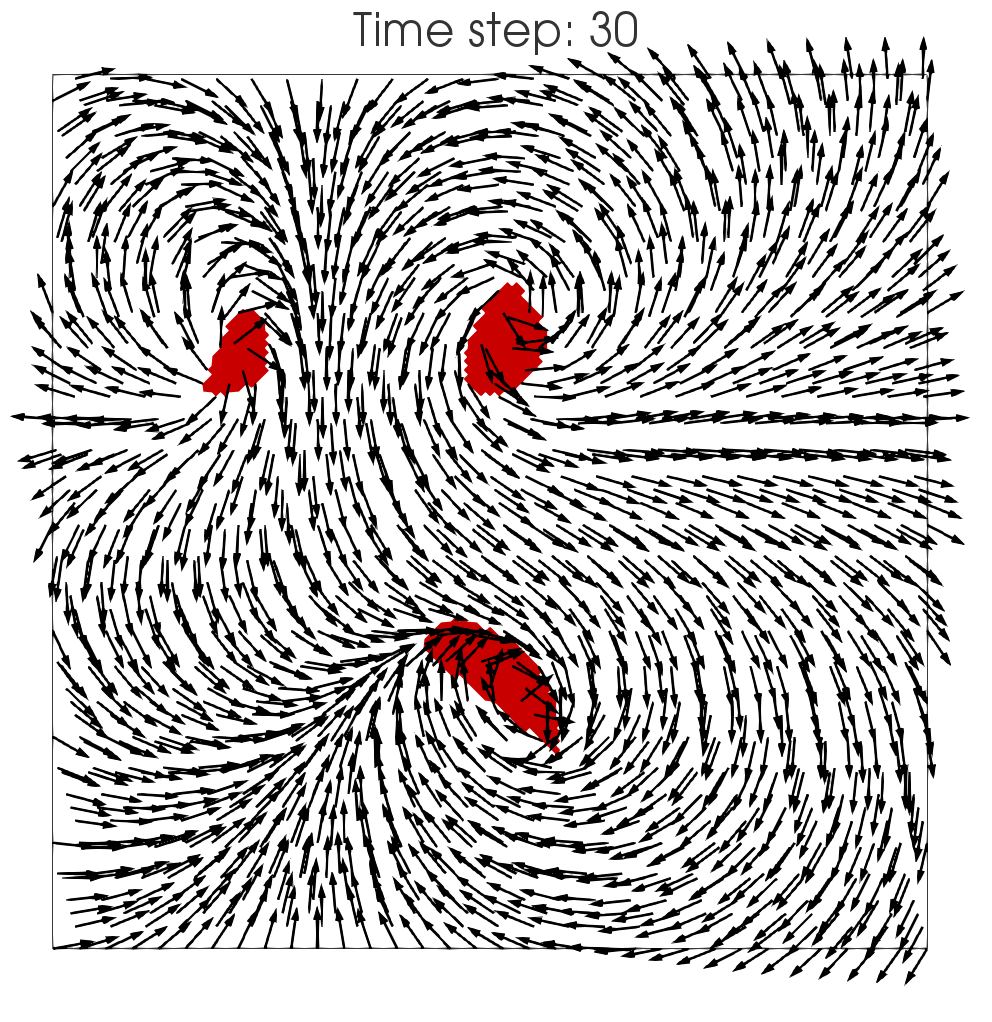}  \\
\includegraphics[width=0.33\linewidth]{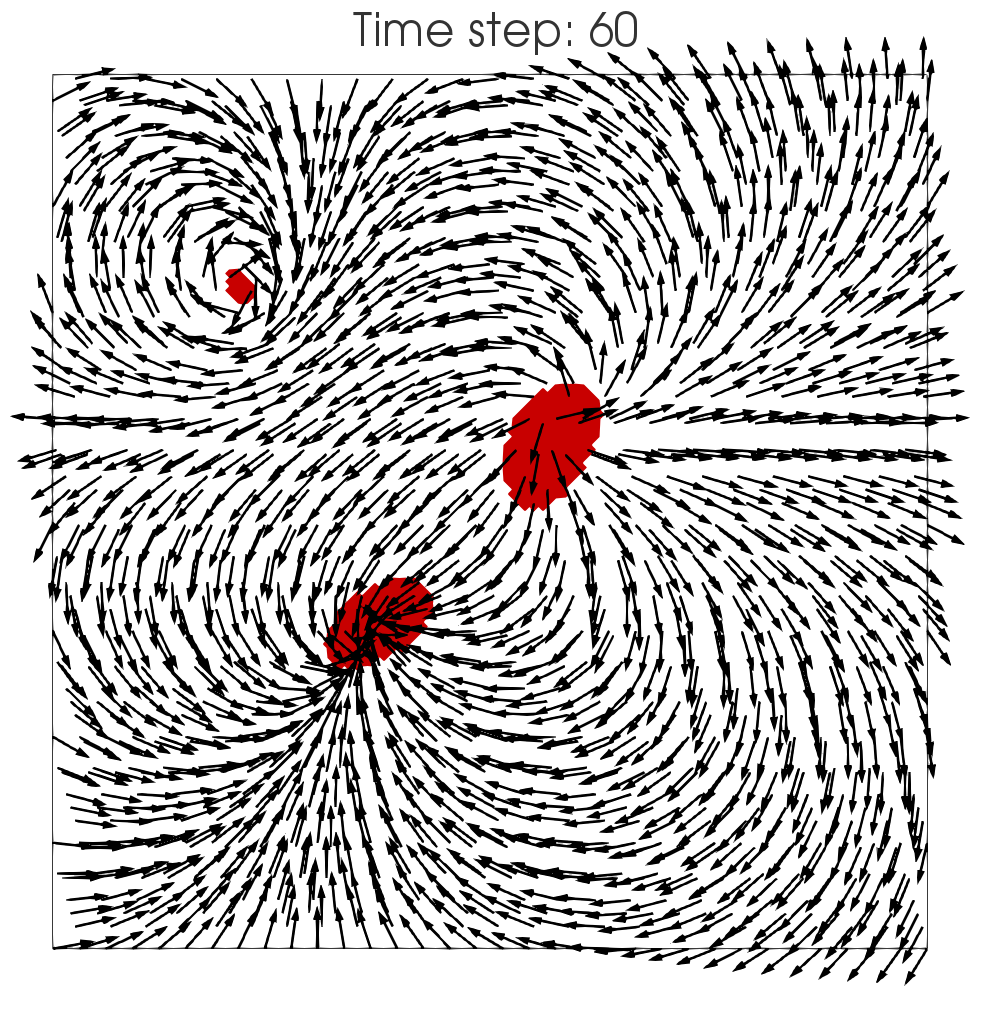} \hspace{-0.3cm}
\includegraphics[width=0.33\linewidth]{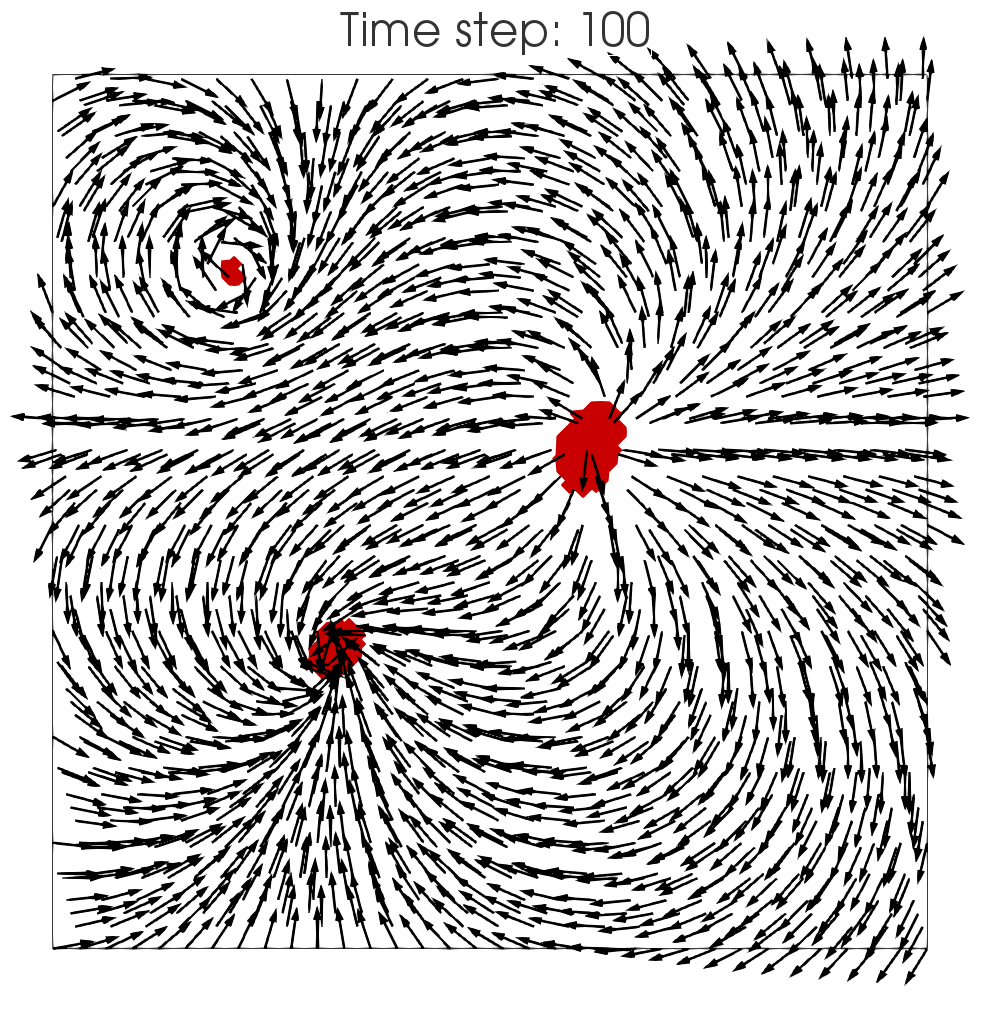} \hspace{-0.3cm}
\includegraphics[width=0.33\linewidth]{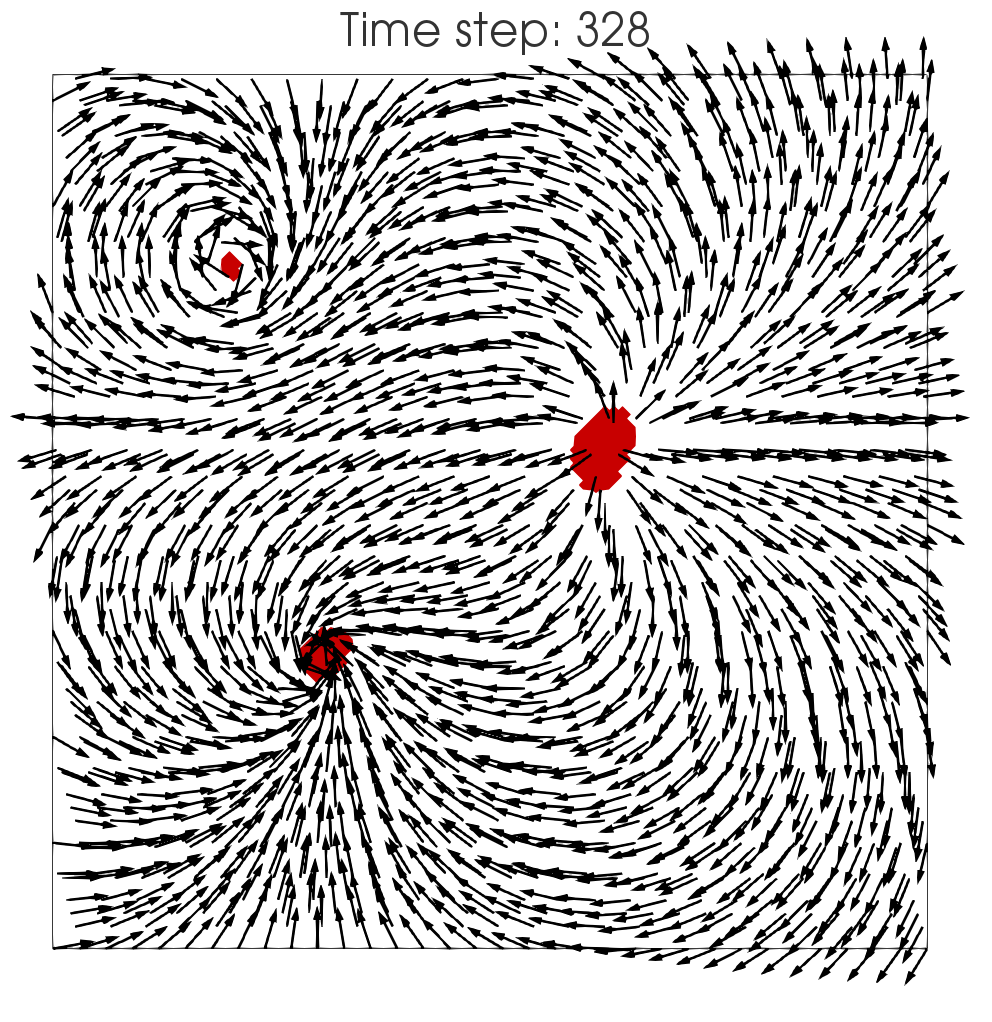}
\caption{Director field $\vn$ in Example \ref{ex:Erk_energy_evolution}. In red, we highlight the regions where $s < 3 \times 10^{-2}$.
}
\label{fig:Erk_evolution_gradient_flow}
\end{figure}

\end{example}
\subsubsection{Theoretical Tools}\label{sec:erk_theoretical_tools}

The singular set $\Sing = \{ s = 0 \}$ plays a critical role in the $\Gamma$-convergence analysis.  The following basic result from \cite[Ch.5, exer. 17]{Evans:book} is used repeatedly when dealing with the singular set.
\begin{lemma}[null gradient on level sets]\label{lem:evans_grad_s_zero_on_s=0}
Let $u \in H^1(\Om)$.  Then, $\nabla u = \vzero$ a.e. on the set $\{ u = c \}$, where $c \in \R$.
\end{lemma}

Since $\DWfunc$ diverges at $s = -1/2$ and $s=1$, it is useful to truncate $s$ away from $s = -1/2, 1$.  The next result, which is a slight modification of \cite[Lem. 3.1]{Nochetto_SJNA2017}, clarifies this.
\begin{lemma}[truncation]\label{lem:Erk_truncate_s}
Assume $(g,\vq)$ satisfies Hypothesis \ref{hyp:basic_bdy_data} (recall $c_0 > 0$).  Let $(s,\vu,\vn) \in \Admiserk(g,\vq)$ and define
\begin{equation}\label{eqn:Erk_truncate_s}
	\trunc{s}_{\rho} := \max \left\{ -\frac{1}{2} + \rho, \min \{ s , 1 - \rho \} \right\},
\end{equation}
for any $\rho \geq 0$, and set $\trunc{\vu}_{\rho} := \trunc{s}_{\rho} \vn$.  Then, $(\trunc{s}_{\rho},\trunc{\vu}_{\rho},\vn) \in \Admiserk(g,\vq)$ for all $\rho \leq c_0$ and
\begin{equation*}
\left\| (s,\vu) - (\trunc{s}_{\rho}, \trunc{\vu}_{\rho}) \right\|_{H^1(\Om)} \to 0, ~\text{ as } \rho \to 0.
\end{equation*}
This is also implies that
\begin{equation}\label{eqn:Erk_truncate_s_monotone_energy}
	\Eerkmain [\trunc{s}_{\rho}, \vn] \leq \Eerkmain [s, \vn], \quad \ipOm{\DWfunc(\trunc{s}_{\rho})}{1} \leq \ipOm{\DWfunc(s)}{1},
\end{equation}
assuming Hypothesis \ref{hyp:bulk_pot} holds as well.

The same assertion holds for any $(s_{h},\vu_{h},\vn_{h}) \in \Admiserk^{h}(g_{h},\vq_{h})$ except the truncation is defined node-wise, i.e.
$(\interp \trunc{s_{h}}_{\rho}, \interp \trunc{\vu_{h}}_{\rho}, \vn_{h})\in \Admiserk^{h}(g_{h},\vq_{h})$ and
\begin{equation}\label{eqn:Erk_truncate_s_monotone_discrete_energy}
\Eerkmain^{h} [\interp \trunc{s_{h}}_{\rho}, \vn_{h}] \leq \Eerkmain^{h} [s_{h}, \vn_{h}], \quad \ipOm{\DWfunc(\interp \trunc{s_{h}}_{\rho})}{1} \leq \ipOm{\DWfunc(s_{h})}{1}.
\end{equation}
\end{lemma}

The following proposition (taken from \cite[Prop. 3.2]{Nochetto_SJNA2017}) is needed to construct a recovery sequence (see Section \ref{sec:Erk_Gamma_convergence}).
\begin{proposition}[Regularization in $\Admiserk(g,\vq)$]\label{prop:Erk_regularization}
Suppose the boundary data satisfies Hypothesis \ref{hyp:basic_bdy_data}.  Let $(s,\vu,\vn) \in \Admiserk(g,\vq)$, with $-\frac{1}{2} + \rho \leq s \leq 1 - \rho$ a.e. in $\Om$ for any $\rho$ such that $0 \leq \rho \leq c_{0}$.  Then, given $\delta > 0$, there exists a triple $(s_\delta, \vu_\delta, \vn_\delta) \in \Admiserk(g,\vq)$, such that $s_\delta \in W^{1,\infty}(\Om)$, $\vu_\delta \in [W^{1,\infty}(\Om)]^{d}$, and
\begin{equation*}
\| (s,\vu) - (s_\delta, \vu_\delta) \|_{H^1(\Om)} \le \delta,
\end{equation*}
\begin{equation*}
-\frac{1}{2} + \rho \le s_\delta (x) \le 1 - \rho, \quad \forall x \in \Om.
\end{equation*}
Thus, there exists $\Van_{\epsilon} \subset \Om$ such that $|\Van_{\epsilon}| < \epsilon$ and $(s_{\delta}, \vu_{\delta})$ converges uniformly on $\Om \setminus \Van_{\epsilon}$.

Moreover, define $\vn_\delta := \vu_\delta / s_\delta$ if $s_\delta \neq 0$, and take $\vn_\delta$ to be any unit vector if $s_\delta = 0$.  Then, $\vn_{\delta} \to \vn$ in $[L^2(\Om \setminus \Sing)]^d$.  Moreover, for each fixed $\epsilon > 0$, $\vn_\delta$ is Lipschitz on $\Om \setminus \{ |s_\delta| \leq \epsilon \}$ with Lipschitz constant proportional to $\epsilon^{-1}$.
\end{proposition}

\subsubsection{Gamma Convergence}\label{sec:Erk_Gamma_convergence}

We briefly review the main results of \cite{Nochetto_SJNA2017} needed to prove $\Gamma$-convergence of $\Eerkone^{h} [s_{h}, \vn_{h}]$ to $\Eerkone [s, \vn]$.  For the existence of a recovery sequence we have \cite[Lem. 3.3]{Nochetto_SJNA2017}.
\begin{lemma}[lim-sup inequality]\label{lem:Erk_limsup}
Let $(s_\epsilon, \vu_\epsilon, \vn_{\epsilon}) \in \Admiserk(g,\vr,\vq) \cap [W^{1,\infty}(\Om)]^{1 + 2d}$ be the functions constructed in Proposition \ref{prop:Erk_regularization}, for any $\epsilon > 0$, and let $(s_{\epsilon,h},\vu_{\epsilon,h},\vn_{\epsilon,h}) \in \Admiserk^{h}(g_h,\vr_h,\vq_{h})$ be their Lagrange interpolants. Then
\begin{equation}
\begin{split}
\Eerkmain[s_\epsilon, \vn_\epsilon] &=
\lim_{h \to 0} \Eerkmain^{h} [s_{\epsilon,h}, \vn_{\epsilon,h}] \\
&= \lim_{h \to 0} \EerkUmain^{h} [s_{\epsilon,h}, \vu_{\epsilon,h}] =
\EerkUmain [s_\epsilon, \vu_\epsilon].
\end{split}
\end{equation}
\end{lemma}

The next result (\cite[Lem. 3.4]{Nochetto_SJNA2017}) is needed for the lim-inf inequality.
\begin{lemma}[weak lower semi-continuity]\label{lem:Erk_weak_lsc}
The energy $\iO L_h(\vw_h, \nabla \vw_h) d\vx$, with
\begin{equation}
	L_h (\vw_h, \nabla \vw_h) := (\kappa - 1) | \nabla  \interp |\vw_h| |^2 + | \nabla  \vw_h |^2,
\end{equation}
is well defined for any $\vw_h \in \Uh$
and is weakly lower semi-continuous in $H^1(\Om)$, i.e. for any
weakly convergent sequence $\vw_h \rightharpoonup \vw$ in $H^1(\Om)$, we have
\begin{equation}
	\liminf_{h \to 0} \iO L_h(\vw_h, \nabla \vw_h) \, d\vx \geq \iO (\kappa - 1) | \nabla |\vw| |^2 + | \nabla  \vw |^2 d\vx.
\end{equation}
\end{lemma}

A basic part of any $\Gamma$-convergence result is an equi-coercivity result (\cite[Lem. 3.5]{Nochetto_SJNA2017}).
\begin{lemma}[coercivity]\label{lem:Erk_coercivity}
For any $(s_h,\vu_h,\vn_h)\in\Admiserk^{h}(g_h,\vr_h,\vq_h)$, we have
\begin{align*}
\Eerkmain^h [s_h, \vn_h] \geq &\; \min\{\kappa, 1\}
\max \left\{\iO |\nabla \vu_h|^2 dx,
\iO |\nabla s_h|^2 dx \right\}
\end{align*}
as well as (recall \eqref{eqn:vu_discrete})
\begin{align*}
\Eerkmain^h [s_h, \vn_h] \geq &\; \min\{\kappa, 1\}
\max\left\{\iO |\nabla \widetilde{\vu}_h|^2 dx, \iO |\nabla I_h |s_h||^2 dx\right\}.
\end{align*}
\end{lemma}

The following result is a modification of \cite[Lem. 3.6]{Nochetto_SJNA2017}, which characterizes the limit functions in our $\Gamma$-convergence result.
\begin{lemma}\label{lem:Erk_char_limit}
Let $(s_h,\vu_h,\vn_h)$ in $\Admiserk^h$ and suppose $(s_h,\vu_h)$ converges weakly to $(s,\vu)$ in $[H^1(\Om)]^{1+d}$.  Then, $(s_h,\vu_h)$ converges to $(s,\vu)$ strongly in $[L^2(\Om)]^{1+d}$, a.e. in $\Om$, where $-1/2 \leq s \leq 1$, $|s|=|\vu|$ a.e. in $\Om$, and there exists a director field $\vn : \Om \to \Sp^{d-1}$, with $\vn \in [L^2(\Om)]^d \cap [L^\infty(\Om)]^d$, such that $\vu = s \vn$ a.e. in $\Om$.  Thus, $(s,\vu,\vn)$ in $\Admiserk$.

Furthermore, $\vn_h$ converges to $\vn$ in $[L^2(\Om \setminus \Sing)]^d$ and a.e. in $\Om \setminus \Sing$, and for each fixed $\epsilon > 0$:
\begin{enumerate}
	\item there exists $\Van_{\epsilon}' \subset \Om$ such that $|\Van_{\epsilon}'| < \epsilon$ and $(s_{h}, \vu_{h})$ converges uniformly to $(s,\vu)$ on $\Om \setminus \Van_{\epsilon}'$;
	
	\item $\vn_h \to \vn$ uniformly on $\Om \setminus (\Sing_\epsilon \cup \Van_{\epsilon}')$, where $\Sing_\epsilon = \{ |s(x)| \leq \epsilon \}$.
\end{enumerate}

Note: the same results hold for $(\widetilde{s}_h,\widetilde{\vu}_h,\vn_h)$ in $\Admiserk^h$ converging to $(\widetilde{s},\widetilde{\vu},\vn)$ in $\Admiserk$, where $\widetilde{s} := |s|$, and $\widetilde{\vu} := \widetilde{s} \vn$ (recall \eqref{eqn:vu_discrete}).
\end{lemma}

Combining the above results, \cite[Thm. 3.7]{Nochetto_SJNA2017} demonstrates $\Gamma$-convergence of our discrete energy to the continuous energy.
\begin{theorem}[convergence of global discrete minimizers]\label{thm:Erk_converge_numerical_soln}
Let $\{ \Tk_h \}$ satisfy \eqref{eqn:weakly-acute}.  If $(s_h,\vu_h,\vn_h)\in \Admiserk(g_h,\vr_h,\vq_h)$ is a sequence of global minimizers of $\Eerkone^{h}[s_h,\vn_h]$ in \eqref{eqn:Erk_total_discrete}, then
every cluster point is a global minimizer of the continuous energy
$\Eerkone [s ,\vn]$ in \eqref{eqn:Ericksen_energy_one_constant}.
\end{theorem}

\section{The Uniaxially Constrained $\vQ$-Model}\label{sec:uniaxial_Q-tensor}
In this section, we address the mathematical formulation of the minimization problem for the one-constant Landau-deGennes energy $\ELdGone$ (cf. \eqref{eqn:Landau-deGennes_energy_one_const}) under the uniaxiality constraint \eqref{eqn:Q_matrix_uniaxial}. For three-dimensional problems, the approach discussed in Section \ref{sec:LdG_numerics} has two drawbacks.

First, a basic argument \cite{Sonnet_book2012} shows that minimizers of $\vQ \mapsto \iO \Bulkfunc (\vQ)$ have the form of a uniaxial nematic \eqref{eqn:Q_matrix_uniaxial}. This is \emph{false} for $\ELdGone$ in \eqref{eqn:Landau-deGennes_energy_one_const} with general boundary conditions. Thus, minimizers of the form \eqref{eqn:Q_tensor_components} \emph{violate} the algebraic form of \eqref{eqn:Q_matrix_uniaxial} and exhibit a {\em biaxial escape} \cite{Palffy-muhoray_LC1994, Sonnet_PRE1995, Lamy_arXiv2013}. This is analogous to the escape to the 3rd dimension in LC director models \cite{Virga_book1994}. This is not desirable if the underlying nematic LC is guaranteed to be uniaxial, which is the case in most thermotropic materials.

In second place, the minimization problem \eqref{eqn:LdG_min_problem_FE_approx} leads to a non-linear system with five coupled variables in 3-D, which is expensive to solve and possibly not robust \cite{Lee_APL2002,Ravnik_LC2009,Zhao_JSC2016,Zhao_JCP2016}. 

These drawbacks motivate us to enforce the uniaxiality constraint \eqref{eqn:Q_matrix_uniaxial} in the Landau-deGennes one-constant energy \eqref{eqn:Landau-deGennes_energy_one_const}. The model we obtain has similarities with the Ericksen model, but it has the advantage of allowing for non-orientable minimizers that exhibit half-integer order defects. We also point out that in 2-D, this approach is equivalent to minimizing \eqref{eqn:Landau-deGennes_energy_one_const} because, according to Remark \ref{rmk:QTensor_2d}, $\vQ$-tensors must be uniaxial.

The approach we pursue is based on the Ericksen model (Section \ref{sec:math_ericksen}). Namely, we shall use $(s,\vn)$ as variables, where $\vn$ is a (possibly non-orientable) vector field, and then recover $\vQ$ by means of \eqref{eqn:Q_matrix_uniaxial}, namely
\[
\vQ = s \left( \vn \otimes \vn -  \frac1d \vI \right).
\] 
Compared to directly minimizing \eqref{eqn:Landau-deGennes_energy_one_const} using the $\vQ$-tensor as a variable, this will allow us to derive an algorithm that can find a minimizer by \emph{solving a sequence of linear systems of smaller dimension}.  See \cite[Prop. 1, pg. 11]{BallOtto_book2015} for a different approach to enforcing uniaxiality.

Finally, we comment that uniaxial models effectively arise in a small elastic constant limit. In \cite{Majumdar_ARMA2010}, Majumdar and Zarnescu studied the one-constant model \eqref{eqn:Landau-deGennes_energy_one_const} with a small bulk coefficient $\Bulkcoef$ (which is equivalent to a small elastic constant). They showed that, under suitable boundary conditions, in the limit $\Bulkcoef \to 0$, Landau-deGennes minimizers converge to minimizers for the Oseen-Frank energy. The analysis in \cite{Majumdar_ARMA2010} is refined in \cite{Nguyen_CVPDE2013}, where the dependence of the difference between the solution to both models with respect to $\Bulkcoef$ is analyzed.

\subsection{Modeling Assumptions}\label{sec:model_assume}

For a uniaxially constrained $\vQ$-tensor as in \eqref{eqn:Q_matrix_uniaxial}, we write $\vNN = \vn \otimes \vn$, which will be treated as a control variable in minimizing \eqref{eqn:Landau-deGennes_energy_one_const}. We introduce the set 
\begin{equation}\label{eqn:matrix_constraint}
\LL^{d-1} = \{ \vA \in \R^{d \times d} : \text{there exists } \vn \in \Sp^{d-1}, \ \vA = \vn \otimes \vn \},
\end{equation}
which can be identified with the real projective space $\vR\vP^{d-1}$ through the map
\[
\vn \otimes \vn  \longmapsto \{ \vn, -\vn \}. 
\]
This illustrates that the uniaxially constrained Landau-deGennes model takes into account the molecular {\em direction} but not the orientation. In contrast to the Oseen-Frank and Ericksen models, the $\vQ$-tensor model allows for half-integer defects.

Because $\nabla \vQ = \nabla s \otimes \left( \vNN - \frac1d \vI \right) + s \nabla \vNN$, we have
\[
 |\nabla \vQ|^2 = |\nabla s|^2 \left| \vNN - \frac1d \vI \right|^2 + s^2 |\nabla \vNN|^2 + 2 s \left[ \nabla s \otimes \left( \vNN - \frac1d \vI \right) \right] \dd \nabla \vNN.
\]
A direct calculation gives $\left| \vNN - \frac1d \vI \right|^2 = \frac{d-1}{d}$ and $\left[ \nabla s \otimes \left( \vNN - \frac1d \vI \right) \right] \dd \nabla \vNN = 0$,
and therefore
\begin{align*}
    |\nabla \vQ|^2 = \frac{d-1}{d} |\nabla s|^2  + s^2 |\nabla \vNN|^2.
\end{align*}
Also, the equalities:
\begin{equation*}
\begin{split}
   \text{for $d=2$:} \quad (1/2) s^2 &= \tr(\vQ^2), \quad 0 = \tr(\vQ^3), \quad (1/4) s^4 = (\tr(\vQ^2))^2, \\
   \text{for $d=3$:} \quad (2/3) s^2 &= \tr(\vQ^2), \quad (2/9) s^3 = \tr(\vQ^3), \quad (4/9) s^4 = (\tr(\vQ^2))^2,
\end{split}
\end{equation*}
follow immediately. Therefore, in the one-constant approximation of the uniaxially constrained $\vQ$-tensor model, the energy
\eqref{eqn:Landau-deGennes_energy_one_const} becomes
\begin{equation}\label{eqn:nematic_Q-tensor_energy_s_ntens}
\begin{split}
& \ELdGone [\vQ] = \Euni[s,\vNN] := \Eunimain[s,\vNN] + \Ebulk[s], \\
& \Eunimain[s,\vNN] := \frac12 \left( \frac{d-1}{d} \iO |\nabla s|^2  \, dx + \Euniring[s,\vNN] \right), \\
& \Euniring[s,\vNN] := \iO s^2 |\nabla \vNN|^2 \, dx, \\
& \Ebulk[s] := \frac{1}{\Bulkcoef} \iO \Bulkfunc (s) \, dx,
\end{split}
\end{equation}
where, with some abuse of notation, we write $\Bulkfunc (s) := \Bulkfunc(\vQ)$.

It is apparent that \eqref{eqn:nematic_Q-tensor_energy_s_ntens} has the \emph{same form} as the Ericksen energy \eqref{eqn:Ericksen_energy_TEMP_one_const}, with the only difference that $\vNN$ replaces $\vn$.
Thus, we introduce a change of variable analogous to the one in the Ericksen case; we set $\vU = s \vNN$ and rewrite
\begin{equation}\label{eqn:Q_tensor_energy_sU}
\Eunimain[s,\vNN] = \EuniUmain[s,\vU] := \frac12 \left( -\frac1d \iO  |\nabla s|^2 \, dx + \iO |\nabla \vU|^2 \, dx \right).
\end{equation}

From the discussion in Section \ref{sec:order_param}, we recall that the degree of orientation needs to satisfy $s \in [-\frac{1}{d-1}, 1]$.
In the same spirit as before, we define the admissible class as 
\begin{equation}\label{eqn:LdG_admissibleclass}
\begin{split}
\Admisuni := \{ (s, \vNN)\in H^1(\Om) \times [L^\infty(\Om)]^{d \times d}: & (s, \vU,\vNN)  \text{ satisfies \eqref{eqn:LdG_structural_conditions}}, \\
	& \text{with } \vu \in [H^1(\Om)]^d \},
\end{split}
\end{equation}
with the structural condition
\begin{equation}\label{eqn:LdG_structural_conditions}
\begin{split}
-\frac{1}{d-1} \leq s \leq 1, \qquad \vU = s \vNN, \qquad \vNN \in \LL^{d-1} \text{ a.e.~in } \Om.
\end{split}
\end{equation}

In the same fashion as we did with the Ericksen model, we shall write $(s,\vU,\vNN)$ in $\Admisuni$, to denote $(s,\vNN)$ in $\Admisuni$, $\vU$ in $[H^{1}(\Om)]^{d\times d}$, and $(s,\vU,\vNN)$ satisfies \eqref{eqn:LdG_structural_conditions}. 
In order to enforce boundary conditions on $(s,\vU)$, possibly on different parts of the boundary, we assume the following condition (cf Hypothesis \ref{hyp:basic_bdy_data}).
  
\begin{hypothesis}\label{hyp:uniaxial_basic_bdy_data}
There exist functions $g \in W^{1,\infty}(\R^d)$, $\vR \in [W^{1,\infty}(\R^d)]^{d\times d}$, $\vM \in [L^{\infty}(\R^d)]^{d\times d}$, such that $(g,\vR,\vM)$ satisfies \eqref{eqn:LdG_structural_conditions} on $\R^d$, i.e. $\vR = g \vM$ and $\vM \in \LL^{d-1}$ a.e. in $\R^d$.  Furthermore, we assume that $g$ satisfies \eqref{eqn:bdy_g_away_from_limits}, that is, there is a fixed $c_0 > 0$ (small) such that
\begin{equation*}
	c_0 \leq g \leq 1 - c_0.
\end{equation*}
The latter implies that $\vM$ is of class $H^1$ in a neighborhood of $\dOm$ and satisfies $g^{-1} \vR = \vM \in \LL^{d-1}$ on $\dOm$.

Moreover, let $\bdys$, $\bdyvU$, $\bdyvNN$ be open subsets of $\dOm$ on which to enforce Dirichlet conditions for $s$, $\vU$, $\vNN$ (respectively), and assume that $\bdyvU = \bdyvNN \subset \bdys$.
\end{hypothesis}

With these boundary conditions, we have the following restricted admissible class,
\begin{align}\label{eqn:LdG_restricted_admissible_class}
  \Admisuni(g,\vM) := \left\{ (s, \vNN) \in \Admisuni : \ s|_{\Gamma_s} = g, \quad \vNN|_{\Gamma_\vNN} = \vM \right\},
\end{align}
and Hypothesis \ref{hyp:uniaxial_basic_bdy_data} guarantees that setting boundary conditions for $(s,\vNN)$ is meaningful.  

Finally, we require the double-well potential to effectively confine the degree of orientation variable $s$ to a meaningful range.

\begin{hypothesis}[Landau-deGennes potential] \label{hyp:LdG_bulk_pot}
The coefficients $A,B,C$ in 
\eqref{eqn:Landau-deGennes_bulk_potential} are such that
\begin{equation} \label{eqn:psi_nondegenerate}
\begin{split}
\Bulkfunc(s) &\ge \Bulkfunc(1-\delta_0) \ \mbox{ for } s \ge 1- \delta_0, \\
\Bulkfunc(s) &\ge \Bulkfunc\left(-\frac{1}{d-1}+\delta_0\right) \ \mbox{ for } s \le -\frac{1}{d-1} + \delta_0.
\end{split}
\end{equation}
Moreover, we modify $\Bulkfunc$ near the bounds $s = -1/(d-1)$ and $s = 1$ so that $\Bulkfunc(\cdot)$ diverges (recall \eqref{eqn:bulk_potential_properties}).
\end{hypothesis}

\subsection{Discretization}\label{sec:uniaxial_FE_discretization}

We discretize $\Om$ in the same fashion as in Section \ref{sec:Erk_FE_discretization}. We assume $\Om \subset \R^d$ is partitioned by a conforming simplicial shape-regular triangulation $\Tk_h = \{ T_i \}$, with no geometric error caused by domain approximation. Moreover, we maintain the weak-acuteness mesh assumption (cf. Hypothesis \ref{hyp:weakly-acute}).

Next, we consider continuous linear Lagrange finite element spaces on $\Om$. That is, the space for $s_h$ is $\Sh$ as in \eqref{eqn:Erk_discrete_spaces}, while the spaces for the tensor variables $\vU_h$ and $\vNN_h$ are
\begin{equation}\label{eqn:LdG_discrete_spaces}
\begin{split}
	\UUh &:= \{ \vU_h \in [H^1(\Om)]^{d\times d} : \vU_h |_{T} \in \Pk_{1} (T), \forall T \in \Tk_h \}, \\
\THh &:= \{ \vNN_h \in \UUh : \vNN_h(x_i) \in \LL^{d-1}, \forall x_{i} \in \Nk_h  \},
\end{split}
\end{equation}
where $\THh$ imposes the rank-one, unit norm constraint only at the vertices of the mesh.
Dirichlet boundary conditions are included via the following discrete spaces:
\begin{equation*}\label{eqn:LdG_discrete_spaces_BC}
\begin{split}
	\Sh (\bdys, g_h) &:= \{ s_h \in \Sh : s_h |_{\bdys} = g_h \}, \\
	\UUh (\Gamma_\vU, \vR_h) &:= \{ \vU_h \in \UUh : \vU_h |_{\Gamma_\vU} = \vR_h \}, \\
	\THh (\Gamma_\vNN,\vM_h) &:= \{ \vNN_h \in \THh : \vNN_h |_{\Gamma_\vNN} = \vM_h \},
\end{split}
\end{equation*}
where $g_h := I_h g$, $\vR_h := I_h \vR$, and $\vM_h := I_h \vM$ are the discrete Dirichlet data.  This leads to the following discrete admissible class with boundary conditions:
\begin{equation}\label{eqn:LdG_admissible_class_BC_discrete}
\begin{split}
	\Admisuni^{h}(g_h,\vM_h) := \big\{ & (s_h,\vNN_h) \in \Sh (\bdys, g_h) \times \THh (\Gamma_\vNN,\vM_h) : \\
	&(s_h,\vU_h,\vNN_n) \text{ satisfies \eqref{eqn:LdG_struct_condition_discrete}}, \text{ with } \vU_h \in \UUh (\Gamma_\vU,\vR_h) \big\},
\end{split}
\end{equation}
where
\begin{equation}\label{eqn:LdG_struct_condition_discrete}
\vU_h = I_h (s_h \vNN_h), \quad - \frac{1}{d-1} \leq s_h \leq 1 \text{ in } \Om, \quad \text{and } \vNN_h(x_{i}) \in \LL^{d-1}, \forall x_{i} \in \Nk_{h},
\end{equation}
is called the \emph{discrete structural condition} of $\Admisuni^{h}$.  If we write $(s_h,\vU_h,\vNN_h) \in \Admisuni^{h}$, then this is equivalent to $(s_h,\vNN_h) \in \Admisuni^h$, $\vU_h \in \UUh$, and $(s_h,\vU_h,\vNN_h)$ satisfies \eqref{eqn:Erk_struct_condition_discrete}. In view of Hypothesis \ref{hyp:uniaxial_basic_bdy_data}, we can also impose the Dirichlet condition
$\vNN_h = I_h[g_h^{-1} \vR_h]$ on $\dOm$.

The discrete version of $\Eunimain[s, \vNN]$ is derived similarly to the Ericksen case. We set
\begin{equation}\label{eqn:delta_ij_UniQ}
\dij s_h := s_h(x_i) - s_h(x_j), \quad \dij \vNN_h := \vNN_h(x_i) - \vNN_h(x_j),
\end{equation}
and define the main part of the discrete energy to be
\begin{equation}\label{eqn:LdG_discrete_energy_main}
\begin{split}
\Eunimain^h[s_h, \vNN_h] := & \frac{d-1}{4d} \sum_{i, j = 1}^n k_{ij} \left( \dij s_h \right)^2 \\
& + \frac{1}{4} \sum_{i, j = 1}^n k_{ij} \left(\frac{s_h(x_i)^2 + s_h(x_j)^2}{2}\right) |\dij \vNN_h|^2.
\end{split}
\end{equation}
Above, the first term corresponds to
\[
\frac12 \sum_{i, j = 1}^n k_{ij} \left( \dij s_h \right)^2 =  \iO |\nabla s_h|^2 dx,
\]
while the second term is a first order approximation of $\frac12 \iO s^2|\nabla \vNN|^2 dx$. For convenience, we shall denote
\begin{equation} \label{eqn:LdG_def_Euniringh}
\Euniring^h[s_h,\vNN_h] := \frac12 \sum_{i, j = 1}^n k_{ij} \left(\frac{s_h(x_i)^2 + s_h(x_j)^2}{2}\right) |\dij \vNN_h|^2. 
\end{equation}

The bulk energy is discretized in the same way as before,
\begin{equation}\label{eqn:LdG_discrete_energy_bulk}
\Ebulk^h [s_h] := \frac{1}{\Bulkcoef} \iO \Bulkfunc (s_h) dx.
\end{equation}
With the notation introduced above, the formulation of the discrete problem reads as follows. Find $(s_h, \vNN_h) \in \Sh (\bdys,g_h) \times \THh (\Gamma_\vNN,\vM_h)$ such that the following energy is minimized:
\begin{equation}\label{eqn:LdG_discrete_energy}
\Euni^h[s_h, \vNN_h] := \Eunimain^h [s_h, \vNN_h] + \Ebulk^h [s_h].
\end{equation}

Because the discrete spaces consist of piecewise linear functions, the structural condition $\vU_h = s_h \vNN_h$ is only satisfied at the mesh nodes (cf. \eqref{eqn:LdG_struct_condition_discrete}). Therefore, there is a variational crime  that we need to account for. Similarly to Lemma \ref{lem:energydecreasing}, the discrete Landau-deGennes energy possesses an energy inequality property \cite[Lem. 1]{BorthNochettoWalker_sub2019}. For our analysis, we introduce the functions
\begin{equation} \label{eqn:def_of_tilde}
\widetilde{s}_h = I_h(|s_h|), \qquad \widetilde{\vU}_h = I_h (|s_h| \vNN_h),
\end{equation}
and remark that $(\widetilde{s}_h, \widetilde{\vU}_h, \vNN_h)$ satisfies \eqref{eqn:LdG_struct_condition_discrete}.

\begin{lemma}[energy inequality] \label{lem:energy_inequality}
Let the mesh $\Tk_h$ satisfy \eqref{eqn:weakly-acute}. Then, for all $(s_h,\vU_h,\vNN_h) \in \Admisuni^h(g_h,\vR_h,\vM_h)$, the discrete energy satisfies
\begin{equation}\label{eqn:energyequality}
  \Eunimain^h[s_h, \vNN_h] - \EuniUmain^h[s_h,\vU_h] = \consistLdG_h,
\end{equation}
as well as
\begin{equation} \label{eqn:energyequality_tilde}
  \Eunimain^h[s_h, \vNN_h] - \EuniUmain^h[\widetilde{s}_h,\widetilde{\vU}_h]
  \ge \consistULdG_h,
\end{equation}
where
\[
\EuniUmain^h[s_h,\vU_h] := \frac12 \left( -\frac{1}{d} \iO |\nabla s_h|^2 dx + \iO |\nabla \vU_h|^2 dx \right),
\]
and
\begin{equation}\label{eqn:LdG_residual}
  \consistLdG_h := \frac{1}{8} \sum_{i, j = 1}^n k_{ij} \big(\dij s_h
  \big)^2 \big| \dij \vNN_h \big|^2 \ge 0,
  \qquad
  \consistULdG_h := \frac{1}{8} \sum_{i, j = 1}^n k_{ij}
  \big(\dij \widetilde{s}_h \big)^2 \big| \dij \vNN_h \big|^2 \ge 0.
\end{equation}
\end{lemma}

\subsection{Gradient flow}\label{sec:LdG_gradient_flow}

\subsubsection{Continuous gradient flow}\label{sec:LdG_continuous_gradient_flow}

We discuss a formal gradient flow to find local minimizers of $\Euni [s,\vNN]$ in \eqref{eqn:nematic_Q-tensor_energy_s_ntens}. More precisely, we revisit \eqref{eqn:LdG_L2_grad_flow}, and impose the uniaxial constraint \eqref{eqn:Q_matrix_uniaxial}. Because our control variables are $s$ and $\vNN$, we shall evolve these two quantities separately, although the resulting gradient flows are coupled. For the sake of exposition, we consider an $L^2$ gradient flow for both $s$ and $\vNN$, although other choices can be made.

We define
\begin{equation}\label{eqn:uniaxial_grad_flow_s}
\begin{split}
	\ips{\partial_{t} s(\cdot,t)}{z} = -\delta_{s} \Euni [s,\vNN; z], \quad \forall z \in H^1_{\bdys} (\Om),
\end{split}
\end{equation}
where $\Hbdy{\bdys} := \{ z \in H^1(\Om) : z|_{\bdys} = 0 \}$ preserves the boundary condition for $s$ and, as in Section \ref{sec:Erk_contin_gradient_flow}, we shall consider $\ips{\cdot}{\cdot}$ as the standard $L^2$-inner product. 
Upon applying a formal integration by parts to this equality and using the Neumann condition $\vnu \cdot \nabla s = 0$ on $\dOm \setminus \bdys$, it follows that $s$ satisfies
\begin{equation}\label{eqn:uniaxial_grad_flow_s_strong_form}
\begin{split}
	\partial_t s - \frac{d-1}{d} \Delta s + |\nabla \vNN|^2 s + \frac{1}{\Bulkcoef} \Bulkfunc'(s) &= 0, ~\text{ in } \Om, \\
	s = g, ~\text{ on } \bdys, \quad \vnu \cdot \nabla s &= 0, ~\text{ on } \dOm \setminus \bdys.
\end{split}
\end{equation}

We now consider the evolution of the $\vNN$ variable. For that purpose, we need a characterization of the tangent space $T_\vNN \LL^{d-1}$ at $\vNN \in \LL^{d-1}$. Following \cite{Bartels_bookch2014}, given $\vNN = \vn \otimes \vn$, we consider a smooth curve $\vgamma : (-\delta, \delta) \to \Sp^{d-1}$ such that $\vgamma(0) = \vn$, and set $\vGamma : (-\delta, \delta) \to \LL^{d-1}$, $\vGamma(t) = \vgamma(t) \otimes \vgamma(t)$. Then, setting $\vv := \frac{d \vgamma}{dt} (0) \in T_\vn \Sp^{d-1}$, we obtain $\vV \in T_\vNN \LL^{d-1}$ by
\[
\vV = \frac{d \vGamma}{dt} (0) = \left. \left( \frac{d}{dt} \vgamma(t) \otimes \vgamma(t) \right) \right|_{t =0} = \vn \otimes \vv + \vv \otimes \vn .
\]
Thus,  at $\vNN = \vn \otimes \vn$, there is a bijection between $T_\vNN \LL^{d-1}$ and $T_\vn \Sp^{d-1}$.

This motivates us to define the space of tangential variations of $\vNN = \vn \otimes \vn$ as
\begin{equation} \label{eqn:uniaxial_tangent_variation_space}
\begin{split}
	\Vperp (s,\vNN) := \big\{ \vV \in [L^2(\Om)]^{d\times d} : \ & \vV = \vn \otimes \vv + \vv \otimes \vn, \ \vv \cdot \vn = 0 \text{ a.e. in } \Om, \\
	& s \vV \in [H^1(\Om)]^{d\times d} \big\}.
\end{split}
\end{equation}
Clearly, the restriction $\vv \cdot \vn = 0$ a.e. in $\Om$ is sufficient to guarantee orthogonality, because
\[
(\vv \otimes \vn) \colon (\vn \otimes \vn) = (\vv \cdot \vn) (\vn \cdot \vn) = 0 \quad \mbox{if } \vv \cdot \vn = 0.
\]
Moreover, if $s \geq c_0 > 0$ in $\Om$, then $\vV \in \Vperp (s,\vNN)$ must belong to $[H^1(\Om)]^d$.

Additionally, considering a tangential perturbation of $\vNN$, $\vT = \vNN + \dt \vV$ with $\vV \in \Vperp (s,\vNN)$ preserves the constraint $|\vNN| = 1$ up to second order, since $|\vT|^2 = 1 + \dt^2 |\vV|^2$. However, our discrete gradient flow algorithm shall exploit the identification $T_\vNN \LL^{d-1} \simeq T_\vn \Sp^{d-1}$ to consider vector-valued perturbations (instead of tensor-valued). Namely, if we take tangential variations in $\vn$ and set $\vW := (\vn + \dt \vv) \otimes (\vn + \dt \vv)$ with $\vv \in T_\vn \Sp^{d-1}$, then this yields a tangential variation of $\vNN$ up to second order,
\begin{equation} \label{eqn:uniaxial_tangential_inconsistency}
\vW = \vNN + \dt \vV + \dt^2 \vv \otimes \vv, \quad \vV = \vn \otimes \vv + \vv \otimes \vn \in \Vperp(s,\vNN).
\end{equation}

We set a gradient flow for $\vNN$ as follows:
\begin{equation}\label{eqn:uniaxial_grad_flow_vNN}
\begin{split}
\ipvT{\partial_{t} \vNN(\cdot,t)}{\vV}
	= -\delta_{\vNN} \Euni [s,\vNN; \vV], \quad \forall \vV \in \Vperp (s,\vNN) \cap [\Hbdy{\bdyvNN}]^{d\times d}.
\end{split}
\end{equation}
Assume that $\ipvT{\cdot}{\cdot}$ above is the inner product in $L^2(\Om)$. 
After an integration by parts, it follows that $\vNN$ satisfies
\begin{equation}\label{eqn:uniaxial_grad_flow_vNN_strong_form}
\begin{split}
	\partial_t \vNN - \nabla \cdot ( s^2 \nabla \vNN) &= \vzero, ~\text{ in } \Om, \\
	\vn = \vM, ~\text{ on } \bdyvNN, \quad \vnu \cdot \nabla \vNN &= \vzero, ~\text{ on } \dOm \setminus \bdyvNN.
\end{split}
\end{equation}

Assuming $(s,\vn)$ evolve according to \eqref{eqn:uniaxial_grad_flow_s_strong_form}, \eqref{eqn:uniaxial_grad_flow_vNN_strong_form}, it follows that
\begin{equation}\label{eqn:uniaxial_grad_flow_energy_decrease}
\begin{split}
	\partial_{t} \Euni [s,\vNN] &= \delta_{s} \Euni [s,\vNN; \partial_{t} s] + \delta_{\vNN} \Euni [s,\vNN; \partial_{t} \vNN], \\
	&= - \| \partial_{t} s \|_{L^2(\Om)} - \| \partial_{t} \vNN \|_{L^2(\Om)} \leq 0,
\end{split}
\end{equation}
and therefore the energy is monotonically decreasing.

\subsubsection{Discrete gradient flow}\label{sec:LdG_discrete_gradient_flow}
Given $k \ge 0$, let $(s_{h}^{k}, \vNN_{h}^{k}) \in \Admisuni^{h}(g_h,\vM_h)$ and we write
\begin{equation*}
s_i^k := s_h^k(\vx_{i}),
\quad
\vNN_i^k := \vNN_h^k(\vx_{i}),
\quad
\vn_i^k := \vn_h^k(\vx_{i}),
\quad
z_i := z_h(\vx_{i}),
\quad
\vv_i := \vv_h(\vx_{i}).
\end{equation*}

We consider the finite element discretization discussed in Section \ref{sec:uniaxial_FE_discretization} and use a fully implicit, backward Euler time discretization for $\partial_{t} s$, to  discretize \eqref{eqn:uniaxial_grad_flow_s} by
\begin{equation}\label{eqn:uniaxial_grad_flow_s_discrete}
\begin{split}
	\ips{\frac{s_{h}^{k+1} - s_{h}^{k}}{\dt}} {z_{h}} = -\delta_{s} \Euni^h [s_{h}^{k+1}, \vNN_{h}^{k+1}; z_{h}], \quad \forall z_{h} \in \Sh(\bdys, 0),
\end{split}
\end{equation}
where $\dt > 0$ is a finite time step and $\Sh(\bdys, 0)$ is defined according to \eqref{eqn:LdG_discrete_spaces_BC}. The discrete variational derivative is given by
\begin{equation}\label{eqn:uniaxial_energy_first_variation_s_discrete}
\begin{split}
	\delta_{s} \Euni^h [s_{h}^{k}, \vNN_{h}^{k}; z_{h}] &= \frac{d-1}{d} (\nabla s_{h}^{k}, \nabla z_{h}) + \frac{1}{2} \delta_{s} \Euniring^h [s_{h}^{k}, \vNN_{h}^{k}; z_{h}] + \frac{1}{\Bulkcoef} (\Bulkfunc'(s_{h}^{k}), z_{h}), \\
	\delta_{s} \Euniring^h [s_{h}^{k}, \vNN_{h}^{k}; z_{h}] &= \sum_{i, j = 1}^N k_{ij} |\dij \vNN_{h}^{k}|^2 \left(\frac{s_i^k z_i+ s_j^k z_j}{2}\right).
\end{split}
\end{equation}

The discrete version of \eqref{eqn:uniaxial_tangent_variation_space}, where $\vNN_h = \vn_h \otimes \vn_h$ at the mesh nodes, is defined by
\begin{equation}\label{eqn:discrete_tangent_variation_space_UniQ}
\begin{split}
	\Vperp_{h} (\vNN_{h}) = \{ \vV_h \in \UUh : \ & \vV_h(\vx_{i}) = \vn_h(\vx_{i}) \otimes \vv_h(\vx_{i}) + \vv_h(\vx_{i}) \otimes \vn_h(\vx_{i}), \\
	 & \vv_h(\vx_{i}) \cdot \vn_h(\vx_{i}) = 0, \text{ for all nodes } \vx_{i} \in \Nk_h \},
\end{split}
\end{equation}

Thus, a natural way to discretize \eqref{eqn:uniaxial_grad_flow_vNN} would be
\begin{equation} \label{eqn:uniaxial_grad_flow_vNN_discrete_temp}
\begin{split}
\ipvT{\frac{\vNN_{h}^{k+1} - \vNN_{h}^{k}}{\dt}}{\vV_{h}} &= -\delta_{\vNN} \Euni^{h} [s_{h}^{k+1}, \vNN_{h}^{k+1}; \vV_{h}], ~~ \forall \vV_{h} \in \Vperp_{h} (\vNN_{h}^{k}) \cap \UUh(\Gamma_\vU,\vzero), \\
	\text{subject to } \vNN_{h}^{k+1} &(\vx_{i}) \in \LL^{d-1}, \text{ for all nodes } \vx_{i} \in \Nk_h,
\end{split}
\end{equation}
where the discrete variational derivatives are given by
\begin{equation}\label{eqn:uniaxial_energy_first_variation_vNN_discrete}
\begin{split}
	\delta_{\vNN} \Euni^h [s_{h}^{k}, \vNN_{h}^{k}; \vV_{h}] &= \frac{1}{2} \delta_{\vNN} \Euniring^h [s_{h}^{k}, \vNN_{h}^{k}; \vV_{h}], \\
	\delta_{\vNN} \Euniring^h [s_{h}^{k}, \vNN_{h}^{k}; \vV_{h}] &= \sum_{i, j = 1}^N k_{ij} \left(\frac{(s_i^k)^2 + (s_j^k)^2 }{2}\right) ( \dij \vNN_h^k ) \colon ( \dij \vV_h ).
\end{split}
\end{equation}

Therefore, we could consider the following algorithm:  given $(s_{h}^{k}, \vNN_{h}^{k}) \in \Admisuni^{h}(g_h,\vM_h)$, solve \eqref{eqn:uniaxial_grad_flow_s_discrete}, \eqref{eqn:uniaxial_grad_flow_vNN_discrete_temp} \emph{simultaneously} to obtain $(s_{h}^{k+1}, \vNN_{h}^{k+1}) \in \Admisuni^{h}(g_h,\vM_h)$. 
Starting from an initial guess $(s_{h}^{0}, \vNN_{h}^{0}) \in \Admisuni^{h}(g_h,\vM_h)$, we iterate this until some convergence criteria is achieved. This algorithm yields a fully coupled non-linear system of equations with the constraint $\vNN_{h}^{k+1}(\vx_i) \in \LL^{d-1}$.

There are two simplifications to be made. First, in the same fashion as we did for the Ericksen model in Section \ref{sec:Erk_discrete_gradient_flow}, we split the gradient flow iteration into three steps. Namely, we evolve $\vNN_h$, resulting in a tangential update that does not necessarily belong to $\LL^{d-1}$ at the nodes; after this, we need to project this update; and finally, evolve $s_h$ with a gradient flow step.

However, the second step in this algorithm is problematic: projecting an arbitrary tensor onto $\LL^{d-1}$ is more challenging than the simple unit length normalization step in the algorithm from Section \ref{sec:Erk_discrete_gradient_flow}. Therefore, instead of looking for tensor variations of $\vNN$, we shall exploit the identification between the tangent spaces $T_\vNN \LL^{d-1}$ and $T_\vn \Sp^{d-1}$ to obtain a {\em vectorial update}. 

We point out that if $\vNN_{h}^{k+1}$ was a tangential update, so that 
\[
\vNN_i^{k+1} - \vNN_i^{k} = \vn_i^k \otimes \vt_i + \vt_i \otimes \vn_i^k
\]
for some $\vt_i$ such that $\vn_i^k\cdot \vt_i = 0$, then we could replace the Frobenius inner product by a vectorial one:
\[
(\vNN_i^{k+1} - \vNN_i^{k}) \colon \vV_i = 2 \vt_i^k \cdot \vv_i, \quad \forall \vV_i = \vn_i^k \otimes \vv_i + \vv_i \otimes \vn_i^k .
\]
Therefore, instead of \eqref{eqn:uniaxial_grad_flow_vNN_discrete_temp} we consider, for $\vNN^k_i = \vn_i^k \otimes \vn_i^k$,
\begin{equation}\label{eqn:uniaxial_grad_flow_vNN_discrete}
\begin{split}
	\frac{1}{\dt}\ipvt{\vt_h^k}{\vv_{h}} = -\delta_{\vNN} & \Euni^{h} [s_{h}^{k+1}, \vNN_{h}^{k} + \vn_h^k \otimes \vt_h^k + \vt_h^k \otimes \vn_h^k; \vV_{h}], \\ & \forall \vV_{h} \in \Vperp_{h} (\vn_{h}^{k}) \cap \Uh(\bdyvu, \vzero).
\end{split}
\end{equation}
Upon taking the update $\widetilde{\vn}_i^{k+1} = \vn_i^k + \vt_i^k$, we can recover the constraint $\vNN^{k+1}_i \in \LL^{d-1}$ by considering 
\[
\vNN^{k+1}_i = \frac{\widetilde{\vn}_i^{k+1}}{|\widetilde{\vn}_i^{k+1}|} \otimes \frac{\widetilde{\vn}_i^{k+1}}{|\widetilde{\vn}_i^{k+1}|}.
\]

Because of the second-order inconsistency committed when updating $\vNN$ with a non-tangential variation (recall \eqref{eqn:uniaxial_tangential_inconsistency}), we need a careful selection of the $\ipvt{\cdot}{\cdot}$-form. Moreover, near the discrete singular set, namely wherever $s_h^k$ is small, it is critical to allow for relatively large variations $\vt_h^k$ in order to accelerate the algorithm. 
Given a function $\omega \in L^\infty (\Om)$ with $\omega \ge 0$, we define the weighted $H^1$-space
\begin{equation*}
  \| \vv \|_{H^1_{\omega}(\Om)} := \left( \int_\Om |\vv(x)|^2 \, dx +
  \int_\Om |\nabla\vv(x)|^2 \, \omega(x) \, dx \right)^{1/2},
\end{equation*}
and we write by 
$(\cdot,\cdot)_{H^1_{\omega}(\Om)}$ its inner product. In the algorithm below, we shall consider the weight $\omega=(s_h^k)^2$.

Finally, we point out that the double well potential can be treated in the same way as for the Ericksen model. Indeed, by using a convex-concave splitting $\Bulkfunc = \psi_c - \psi_e$
and considering the approximation
\begin{equation}\label{eqn:uniaxial_DW_variation_convex_split}
	\ipOm{\Bulkfunc'(s_{h}^{k+1})}{z_{h}} := \ipOm{\DWfuncimp'(s_{h}^{k+1})}{z_{h}} - \ipOm{\DWfuncexp'(s_{h}^{k})}{z_{h}},
\end{equation}
we obtain an unconditionally stable evolution for $\Ebulk^h[s_h]$ (cf. Lemma \ref{lem:Erk_convex_splitting}).

Our discrete quasi-gradient flow algorithm is as follows.
Given $(s_h^0, \vNN_h^0) \in \Admisuni^h(g_h, \vM_h)$, with $\vNN_h^0 = \vn_h^0 \otimes \vn_h^0$, and a time step $\dt > 0$, iterate Steps 1--3 for $k \ge 0$:

\begin{enumerate}
\item \textbf{(Weighted) tangent flow step for $\vNN_h$:} find $\vt^k_h \in \Vperp_{h} (\vn_h^k) \cap H^1_{\Gamma_\vNN}(\Om)$ and $\vT^k_h = \vn^k_h \otimes \vt^k_h + \vt^k_h \otimes \vn^k_h$, such that
\begin{equation} \label{eqn:uniaxial_flow_theta}
\begin{split}
	\frac{1}{\dt} (\vt^k_h , \vv_h)_{H^1_{(s^k_h)^2}(\Om)}& = - \delta_{\vNN} \Euniring^h [s^k_h, \vNN^k_h + \vT^k_h; \vV_h], \\ & \forall \vV_h = \vn^k_h \otimes \vv_h + \vv_h \otimes \vn^k_h, \ \vv_h \in  \Vh \cap H^1_{\Gamma_\vNN}(\Om).
\end{split}
\end{equation}

\item \textbf{Projection:} update $\vNN_h^{k+1} \in \THh (\Gamma_\vNN, \vM_h)$ by
\begin{equation} \label{eqn:uniaxial_vNN_update}
\vNN_h^{k+1}(x_i) := \frac{\vn_h^k(x_i) + \vt_h^k(x_i)}{|\vn_h^k(x_i) + \vt_h^k(x_i)|} \otimes \frac{\vn_h^k(x_i) + \vt_h^k(x_i)}{|\vn_h^k(x_i) + \vt_h^k(x_i)|}, \quad \forall x_i \in \Nk_h.
\end{equation}

\item \textbf{Gradient flow step for $s_h$:} find $s_h^{k+1} \in \Sh(\bdys,g_{h})$ such that, for all $z_{h} \in \Sh(\bdys,0)$,
\begin{equation} \label{eqn:uniaxial_flow_s}
\begin{split}
	\frac{1}{\dt} \ipOm{s_h^{k+1} - s_h^k}{z_h}  &= -\delta_{s} \Euni^h[s_h^{k+1}, \vNN_h^{k+1} ; z_h ],  \\
	&= -\frac{d-1}{d} \ipOm{\nabla s_{h}^{k+1}}{\nabla z_{h}} - \frac{1}{2} \delta_{s} \Euniring^h [s_{h}^{k+1}, \vNN_{h}^{k+1}; z_{h}] \\
	&\qquad\quad - \frac{1}{\Bulkcoef} \ipOm{\Bulkfunc'(s_{h}^{k+1})}{z_{h}}.
\end{split}
\end{equation}
\end{enumerate}

Under a mild time-step restriction, this algorithm is energy-decreasing \cite[Thm. 2]{BorthNochettoWalker_sub2019}.

\begin{theorem}[energy decrease] \label{thm:energy_decrease}
Assume the family of meshes is weakly acute (cf. Hypothesis \ref{hyp:weakly-acute}) and $\dt < C h^{d/2}$. Then, it holds that 
\[
\Euni^h[s_h^N, \vNN_h^N] + \frac{1}{\dt} \left( \sum_{k=0}^{N-1} \| \vt_h^k \|^2_{H^1_{(s_h^k)^2}(\Omega)}  +  \| s_h^{k+1} - s_h^k \|_{L^2(\Om)}^2 \right) \le \Euni^h[s_h^0, \vNN_h^0] \quad \forall N \ge 1. 
\]
Thus, the discrete energy is monotonically decreasing.
\end{theorem}

\begin{remark}[CFL condition]
The use of the weighted $H^1_{(s_h^k)^2}(\Omega)$-norm in Step 1 is needed to bound the second-order consistency error \eqref{eqn:uniaxial_tangent_variation_space}.
This, in turn, leads to the stability constraint $\dt \le C h^{d/2}$ because of the use of an inverse estimate between $L^\infty(\Omega)$ and $L^2(\Omega)$ \cite{BorthNochettoWalker_sub2019}. However, if $(s_h^k)^2$ is bounded away from zero, then a milder CFL condition can be obtained, namely $\dt\le C h^{d/2-1}|\log h|^{-1}$ \cite{Bartels_bookch2014}.
\end{remark}

\subsection{Gamma Convergence}\label{sec:LdG_Gamma_conv}
The roadmap to prove $\Gamma$-convergence of the discrete energy minimization problems to the continuous one is the same as for the Ericksen model, and makes use of the general philosophy  \cite{Braides_book2002},
\begin{center}
equi-coerciveness + $\Gamma$-convergence $\Rightarrow$ convergence of minimum problems.
\end{center}

As a first step, we remark that truncating the double-well potential decreases energy.

\begin{lemma}[truncation]\label{lem:LdG_truncate_s}
Assume $(g,\vM)$ satisfies Hypothesis \ref{hyp:uniaxial_basic_bdy_data}.  Let $(s,\vU,\vNN) \in \Admiserk(g,\vM)$ and, given $\rho  \ge 0$, consider $\trunc{s}_{\rho}$ as in \eqref{eqn:Erk_truncate_s}, namely: 
 define
\begin{equation*}
	\trunc{s}_{\rho} := \max \left\{ -\frac{1}{2} + \rho, \min \{ s , 1 - \rho \} \right\}.
\end{equation*}
Define $\trunc{\vU}_{\rho} := \trunc{s}_{\rho} \vNN$.  Then, $(\trunc{s}_{\rho},\trunc{\vU}_{\rho},\vn) \in \Admiserk(g,\vR)$ for all $\rho \leq c_0$ and
\begin{equation*}
\left\| (s,\vU) - (\trunc{s}_{\rho}, \trunc{\vU}_{\rho}) \right\|_{H^1(\Om)} \to 0, ~\text{ as } \rho \to 0.
\end{equation*}
This is also implies that
\begin{equation*}
	\Eunimain [\trunc{s}_{\rho}, \vNN] \leq \Eunimain [s, \vNN], \quad \Ebulk[\trunc{s}_{\rho}] \leq \Ebulk[s],
\end{equation*}
where we also assume Hypothesis \ref{hyp:LdG_bulk_pot}.

The same assertion holds for any $(s_{h},\vU_{h},\vNN_{h}) \in \Admisuni^{h}(g_{h},\vM_{h})$ if the truncation is defined node-wise. Namely, if
$(\interp \trunc{s_{h}}_{\rho}, \interp \trunc{\vU_{h}}_{\rho}, \vNN_{h})\in \Admisuni^{h}(g_{h},\vM_{h})$ then
\begin{equation*}
\Eunimain^{h} [\interp \trunc{s_{h}}_{\rho}, \vNN_{h}] \leq \Eunimain^{h} [s_{h}, \vNN_{h}], \quad \Ebulk^h[\interp \trunc{s_{h}}_{\rho}] \leq \Ebulk^h[s_{h}].
\end{equation*}
\end{lemma}

Because our discrete admissible class is defined by enforcing the structural conditions nodewise, we use Lagrange interpolation to construct a recovery sequence (i.e., to prove the lim-sup property needed for $\Gamma$-convergence). However, the natural space for $(s,\vU)$ is $[H^1(\Omega)]^{1+d\times d}$ (cf. \eqref{eqn:LdG_admissibleclass}), and thus this construction cannot be done a priori: the Lagrange interpolant of an admissible pair $(s,\vU)$ may not be defined at all if $d \ge 2$. This motivates the following result, which is a counterpart of Proposition \ref{prop:Erk_regularization}. Essentially, it guarantees that Lipschitz continuous functions are $H^1$-dense in the admissible class.

\begin{proposition}[Regularization in $\Admisuni(g,\vM)$]\label{prop:LdG_regularization}
Suppose the boundary data satisfies Hypothesis \ref{hyp:uniaxial_basic_bdy_data}.  Let $(s,\vU,\vNN) \in \Admisuni(g,\vM)$, with $-\frac{1}{d-1} + \rho \leq s \leq 1 - \rho$ a.e. in $\Om$ for any $\rho$ such that $0 \leq \rho \leq c_{0}$.  Then, given $\epsilon > 0$, there exists a triple $(s_\epsilon, \vU_\epsilon, \vNN_\epsilon) \in \Admisuni(g,\vM)$, such that $s_\epsilon \in W^{1,\infty}(\Om)$, $\vU_\epsilon \in [W^{1,\infty}(\Om)]^{d\times d}$, and
\begin{equation*}
\| (s,\vU) - (s_\epsilon, \vU_\epsilon) \|_{H^1(\Om)} \le \epsilon,
\end{equation*}
\begin{equation*}
-\frac{1}{d-1} + \rho \le s_\epsilon (x) \le 1 - \rho, \quad \forall x \in \Om.
\end{equation*}
Thus, there exists $\Van_{\delta} \subset \Om$ such that $|\Van_{\delta}| < \epsilon$ and $(s_\epsilon, \vU_\epsilon)$ converges uniformly on $\Om \setminus \Van_{\delta}$.

Moreover, define $\vNN_\epsilon := \vU_\epsilon / s_\epsilon$ if $s_\epsilon \neq 0$, and take $\vNN_\epsilon$ to be any tensor in $\LL^{d-1}$ if $s_\epsilon = 0$.  Then, $\vNN_{\epsilon} \to \vNN$ in $[L^2(\Om \setminus \Sing)]^d$.  Moreover, for each fixed $\delta > 0$, $\vNN_\epsilon$ is Lipschitz on $\Om \setminus \{ |s_\epsilon| \leq \delta \}$ with Lipschitz constant proportional to $\delta^{-1}$.
\end{proposition}

The proof of the proposition above is more delicate than for the Ericksen case. Indeed, smoothening the tensor field $\vNN$ involves convolution and thus breaks its uniaxial structure. Therefore, uniaxiality needs to be rebuilt into the regularized field. 
We recall that, for instance in three dimensions, the eigenvalues of the uniaxial $\vQ$ tensor $\vQ = s (\vNN - \frac{1}{3} \vI)$ are $\{ 2s/3, -s/3, -s/3 \}$. Heuristically, convolution with a localized kernel should not affect much the eigenframe of $\vNN (x)$ if $s$ is uniformly positive in a neighborhood of $x$. In such a case, one can simply extract the leading eigenspace to construct a uniaxial field. However, if $s$ is not uniformly positive the argument does not carry. To deal with this, the idea in \cite[Prop. 7]{BorthNochettoWalker_sub2019} is to regularize the positive semidefinite field $|s|\vNN$ within a scale $\delta$, to rebuild the uniaxiality and, for a coarser scale $\sigma \ge \delta$, to recover the sign of $s$ by using a suitably regularized sign function.

Once we know that Lipschitz continuous functions are dense among the admissible pairs $(s, \vU)$, we can build a recovery sequence by using Lagrange (nodal) interpolation.

\begin{lemma}[lim-sup inequality]\label{lem:LdG_limsup}
Let $(s_\epsilon, \vU_\epsilon, \vNN_{\epsilon}) \in \Admisuni(g,\vR,\vM)$ be the functions constructed in Proposition \ref{prop:LdG_regularization}, for any $\epsilon > 0$, and let $(s_{\epsilon,h},\vU_{\epsilon,h},\vNN_{\epsilon,h}) \in \Admisuni^{h}(g_h,\vR_h,\vM_{h})$ be their Lagrange interpolants. Then
\begin{equation*}
\begin{split}
\Eunimain[s_\epsilon, \vNN_\epsilon] &=
\lim_{h \to 0} \Eunimain^{h} [s_{\epsilon,h}, \vNN_{\epsilon,h}] \\
&= \lim_{h \to 0} \EuniUmain^{h} [s_{\epsilon,h}, \vU_{\epsilon,h}] =
\EuniUmain [s_\epsilon, \vU_\epsilon].
\end{split}
\end{equation*}
\end{lemma}

Weak lower semi-continuity follows by the same arguments as in the Ericksen case.

\begin{lemma}[weak lower semi-continuity]\label{lem:LdG_weak_lsc}
The energy $\iO L_h(\vW_h, \nabla \vW_h) d\vx$, with
\begin{equation*}
	L_h (\vW_h, \nabla \vW_h) := -\frac1d | \nabla  \interp |\vW_h| |^2 + | \nabla  \vW_h |^2,
\end{equation*}
is well defined for any $\vW_h \in \UUh$
and is weakly lower semi-continuous in $H^1(\Om)$, i.e. for any
weakly convergent sequence $\vW_h \rightharpoonup \vW$ in $H^1(\Om)$, we have
\begin{equation}
	\liminf_{h \to 0} \iO L_h(\vW_h, \nabla \vW_h) \, d\vx \geq \iO -\frac1d | \nabla |\vW| |^2 + | \nabla  \vW |^2 d\vx.
\end{equation}
\end{lemma}
\begin{proof}
Indeed, ``flattening'' the matrix $\vW \in \R^{d \times d}$ to a vector $\vw \in \R^{d^2}$, we can use the same proof from Lemma \ref{lem:Erk_weak_lsc} to prove the result because the norm of the gradient of the flattened matrix equals the Fr\"obenius norm of $\nabla \vU_h$; recall that $\kappa = (d-1)/d$.
\end{proof}

The next result shows that the discrete energy controls the $H^1$ norms of both $s_h$ and $\vU_h$. This gives us the compactness needed to prove convergence of discrete minimizers towards minimizers of the uniaxially constrained Landau-deGennes energy \eqref{eqn:nematic_Q-tensor_energy_s_ntens}.

\begin{lemma}[coercivity]\label{lem:LdG_coercivity}
For any $(s_h,\vU_h,\vNN_h)\in\Admisuni^{h}(g_h,\vR_h,\vM_h)$, we have
\begin{align*}
\Eunimain^h [s_h, \vNN_h] \geq &\; \frac{d-1}{d}
\max \left\{\iO |\nabla \vU_h|^2 dx,
\iO |\nabla s_h|^2 dx \right\}
\end{align*}
as well as (recall \eqref{eqn:vu_discrete})
\begin{align*}
\Eunimain^h [s_h, \vn_h] \geq &\; \frac{d-1}{d}
\max\left\{\iO |\nabla \widetilde{\vU}_h|^2 dx, \iO |\nabla I_h |s_h||^2 dx\right\}.
\end{align*}
\end{lemma}

Next, we prove that the limit functions satisfy the Landau-deGennes admissibility condition \eqref{eqn:LdG_restricted_admissible_class} (cf. \cite[Lem. 9]{BorthNochettoWalker_sub2019}).

\begin{lemma}\label{lem:LdG_char_limit}
Let $(s_h,\vU_h,\vNN_h)$ in $\Admisuni^h$ and suppose $(s_h,\vU_h)$ converges weakly to $(s,\vU)$ in $[H^1(\Om)]^{1+d\times d}$.  Then, $(s_h,\vU_h)$ converges to $(s,\vU)$ strongly in $[L^2(\Om)]^{1+d\times d}$, a.e. in $\Om$, where $-1/(d-1) \leq s \leq 1$, $|s|=|\vU|$ a.e. in $\Om$, and there exists a field $\vNN : \Om \to \LL^{d-1}$, so that $\vNN \in [L^\infty(\Om)]^{d\times d}$, such that $\vU = s \vNN$ a.e. in $\Om$.  Thus, $(s,\vU,\vNN)$ in $\Admisuni$.

Furthermore, $\vNN_h$ converges to $\vNN$ in $[L^2(\Om \setminus \Sing)]^{d\times d}$ and a.e. in $\Om \setminus \Sing$, $\vNN$ admits a Lebesgue gradient on $\Om \setminus \Sing$, that satisfies the identity $|\nabla \vU|^2 = |\nabla s|^2 + s^2|\nabla \vNN|^2$ a.e. in $\Om\setminus\Sing$, and for each fixed $\epsilon > 0$:
\begin{enumerate}
	\item there exists $\Van_{\epsilon}' \subset \Om$ such that $|\Van_{\epsilon}'| < \epsilon$ and $(s_{h}, \vU_{h})$ converges uniformly to $(s,\vU)$ on $\Om \setminus \Van_{\epsilon}'$;
	
	\item $\vNN_h \to \vNN$ uniformly on $\Om \setminus (\Sing_\epsilon \cup \Van_{\epsilon}')$, where $\Sing_\epsilon = \{ |s(x)| \leq \epsilon \}$.
\end{enumerate}

Note: the same results hold for $(\widetilde{s}_h,\widetilde{\vU}_h,\vNN_h)$ in $\Admiserk^h$ converging to $(\widetilde{s},\widetilde{\vU},\vNN)$ in $\Admisuni$, where $\widetilde{s} := |s|$, and $\widetilde{\vU} := \widetilde{s} \vNN$ (recall \eqref{eqn:def_of_tilde}).
\end{lemma}

Collecting Lemmas \ref{lem:LdG_limsup}--\ref{lem:LdG_char_limit}, a standard argument yields the $\Gamma$-convergence of our discrete energy to the continuous energy.
\begin{theorem}[convergence of global discrete minimizers]\label{thm:LdG_converge_numerical_soln}
Let $\{ \Tk_h \}$ satisfy \eqref{eqn:weakly-acute}.  If $(s_h,\vU_h,\vNN_h)\in \Admisuni(g_h,\vR_h,\vM_h)$ is a sequence of global minimizers of $\Euni^{h}[s_h,\vNN_h]$ in \eqref{eqn:LdG_discrete_energy}, then
every cluster point is a global minimizer of the continuous energy
$\Euni [s ,\vNN]$ in \eqref{eqn:nematic_Q-tensor_energy_s_ntens}.
\end{theorem}

\subsection{Numerical Experiment}\label{sec:half_degree_line_defect}

We simulate a curved line defect in the unit cube $(0,1)^3$ that exhibits a $+1/2$ degree ``point'' defect in each horizontal plane of the cube; hence, the line field is non-orientable.  We first simulate the uniaxially constrained model, then the standard LdG model.

\subsubsection{Uniaxially Constrained Model}\label{sec:uni_half_degree_line_defect}

The double-well potential with a convex splitting is given by
\begin{equation} \label{eqn:uni_double-well_splitting}
\begin{split}
	\Bulkfunc (s) & = \psi_c(s) - \psi_e(s) \\ 
	&  := (36.7709 s^2 + 1) - (-7.39101 s^4 + 4.51673 s^3 + 39.27161 s^2),
\end{split}
\end{equation}
with $\Bulkcoef = 1/16$, and note that $\Bulkfunc$ has a local maximum at $s = 0$ and a global minimum at $s = s^* := 0.700005531$ with $\Bulkfunc(s^*) = 0$.

The boundary conditions for $\vNN$ were constructed in the following way.  Let $\theta_{0}(x,y)$ define a $+1/2$ degree defect in the plane, located at $(0.3,0.3)$ by
\begin{equation}\label{eqn:plus_1/2_degree_defect}
\quad \theta_{0}(x,y)= \frac12 \mbox{atan2} \left(\frac{y-0.3}{x-0.3}\right),
\end{equation}
where $\mbox{atan2}$ is the four-quadrant inverse tangent function (analogous to \eqref{eqn:plus_3_degree_defect}).  Likewise, let $\theta_{1}(x,y)$ define a $+1/2$ degree defect in the plane, located at $(0.7,0.7)$.  Next, define the Dirichlet boundary $\Gamma_s = \Gamma_\vNN = \dOm \setminus \Gamma_{o}$, where $\Gamma_{o} := \overline{\Om} \cap (\{ z=0 \} \cup \{ z=1 \})$.  Then, the Dirichlet conditions are
\begin{equation}\label{eqn:uni_plus_1/2_degree_defect_BCs}
\begin{split}
	s = s^*, \quad \vn(x,y) &= (\cos \theta, \sin \theta,0), \quad \vNN = \vn \otimes \vn,  \\
	\theta(x,y,z) &= (1 - z) \theta_{0} (x,y) + z \theta_{1} (x,y) + \pi z,
\end{split}
\end{equation}
with vanishing Neumann condition on $\Gamma_{o}$.  Basically, the boundary conditions consist of rotating a planar $+1/2$ degree point defect as a function of $z$.  The solution is computed with the gradient flow approach in Section \ref{sec:LdG_discrete_gradient_flow} and time step $\dt = 10^{-3}$, and initialized with
\[
s = s^*, \quad \vn = (\cos \alpha, \sin \alpha, 0), \quad \vNN = \vn \otimes \vn, \quad \alpha(x,y,z) = \theta_{2} (x,y) + \pi z,
\]
where $\theta_{2} (x,y)$ corresponds to a $+1/2$ degree defect centered at $(0.5,0.5)$; this configuration has an initial energy of $\Euni^h[s_h,\vNN_h] = 10.013214$.

Figure \ref{fig:uni_half_line_defect_view_A} shows three dimensional views of the minimizing configuration, where as Figure \ref{fig:uni_half_line_defect_view_B} shows four horizontal slices of the solution.  A non-orientable line defect is observed, with final energy $\Euni^h[s_h,\vNN_h] = 5.2042593769$ and $\min(s_{h}) \approx 2.145 \times 10^{-2}$.
\begin{center}
\begin{figure}[ht]
\includegraphics[width=0.44\linewidth]{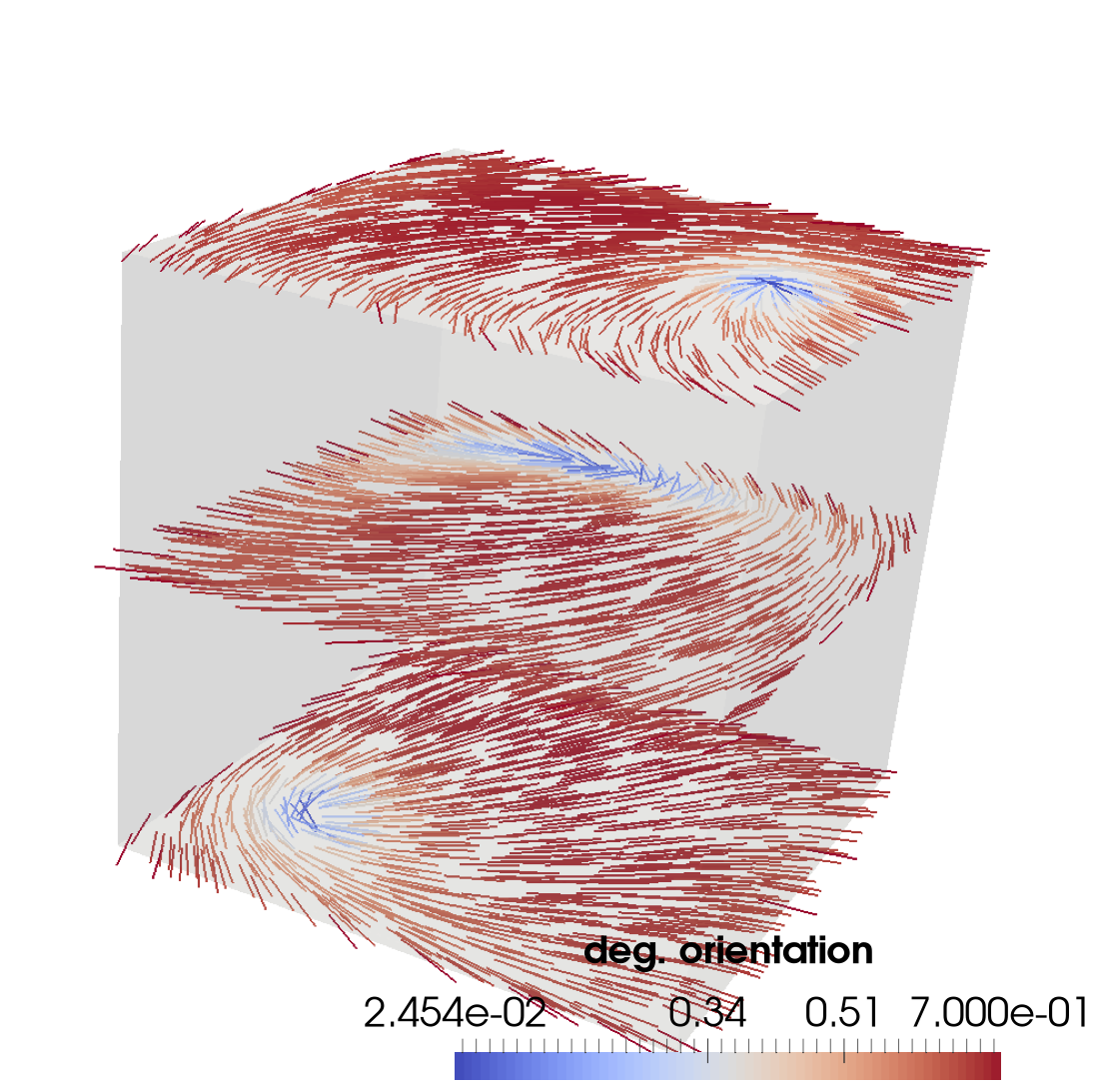} \hspace{0.1in}
\includegraphics[width=0.44\linewidth]{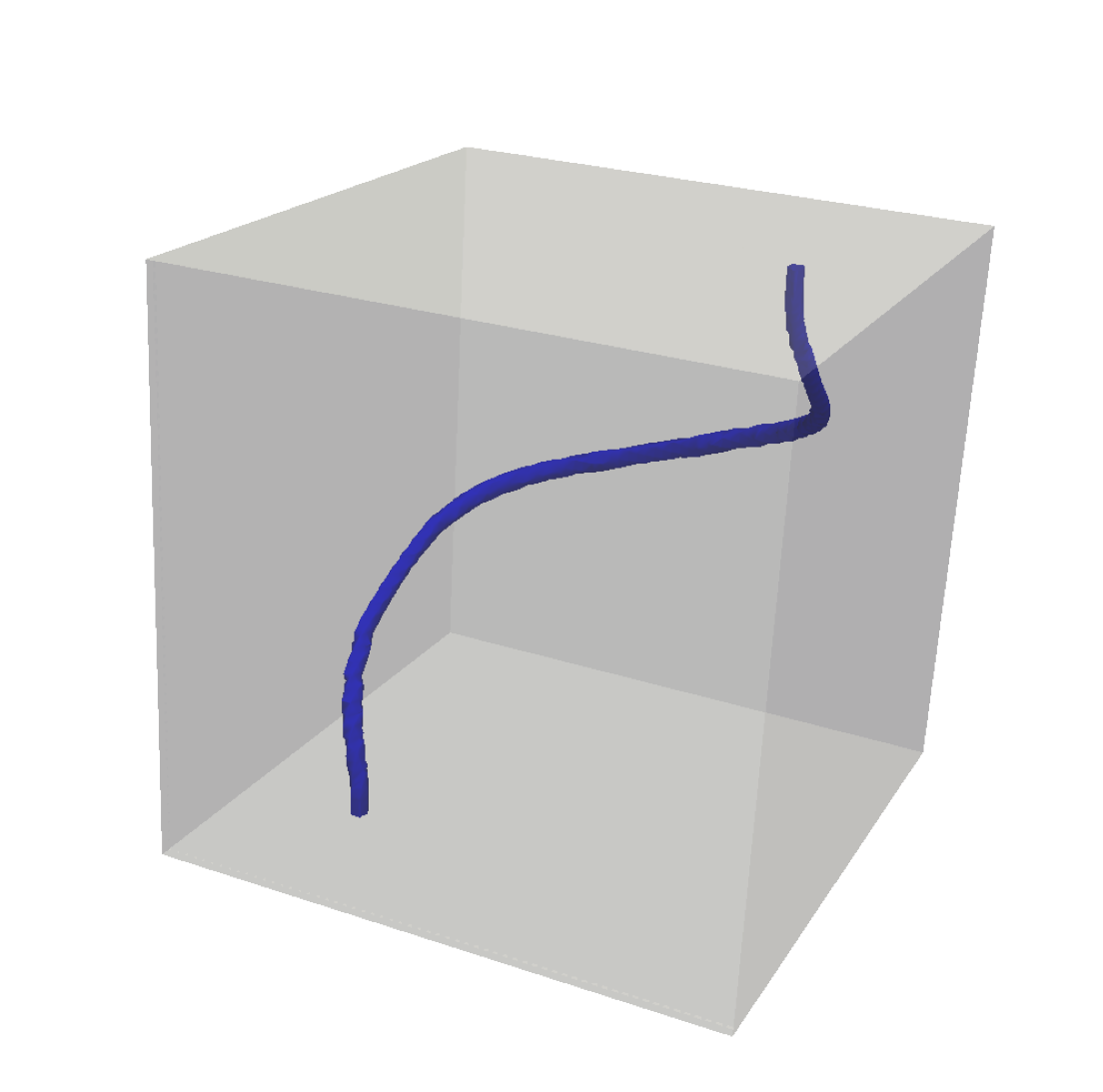}
\caption{A $+1/2$ degree line defect in a 3-D cube domain (Section \ref{sec:uni_half_degree_line_defect}). Left: line field $\vNN$ is shown at levels $z=0.0$, $0.5$, $1.0$ (colored by $s$).  Right: The $s=0.05$ iso-surface is shown that contains the line defect.  In each horizontal plane, the line field exhibits a $+1/2$ degree \emph{point} defect in 2-D.  The twisting of the line defect is due to the choice of boundary conditions.
}
\label{fig:uni_half_line_defect_view_A}
\end{figure}
\end{center}

\begin{center}
\begin{figure}[ht]
\includegraphics[width=0.4\linewidth]{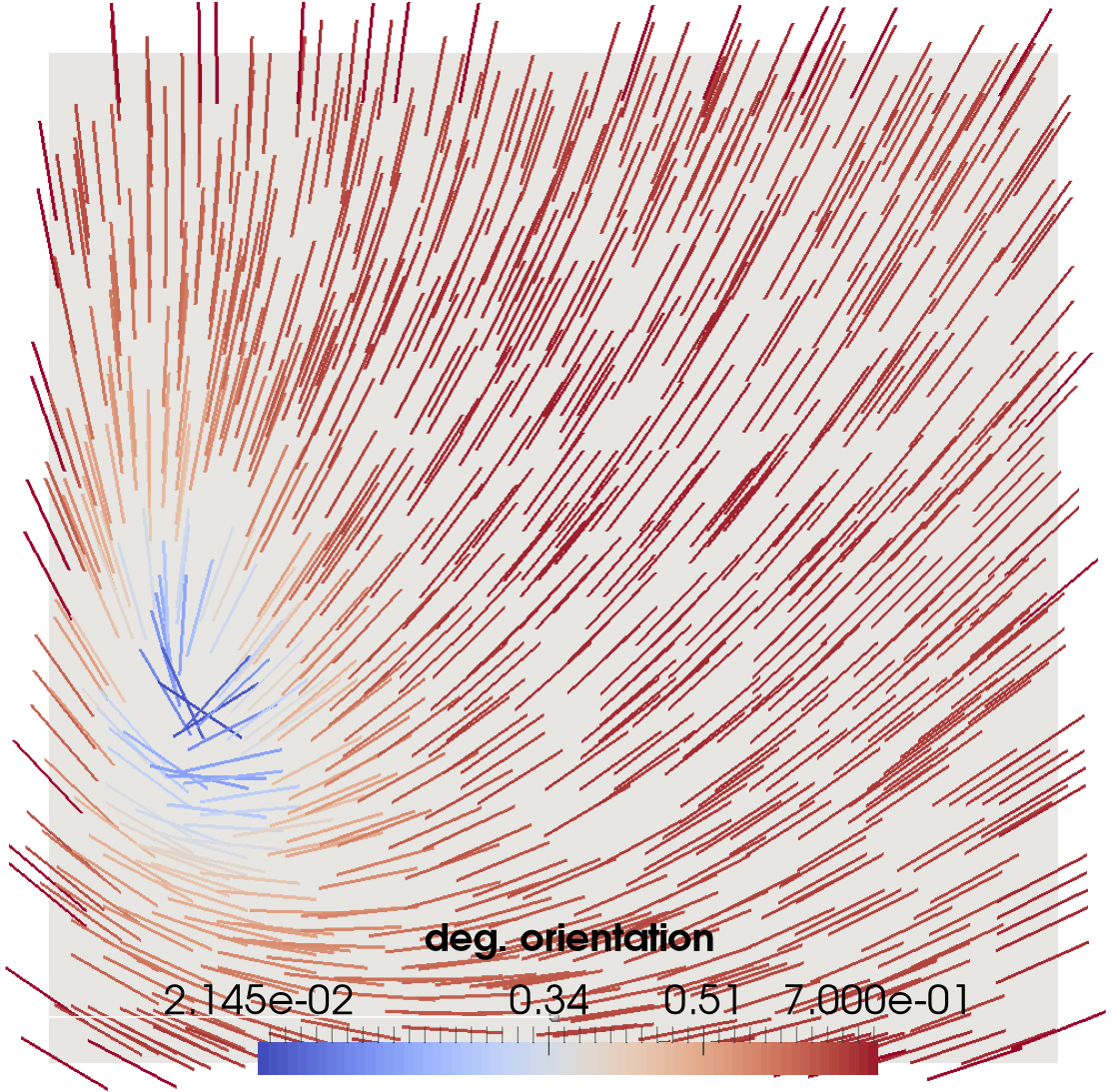} 
\includegraphics[width=0.395\linewidth]{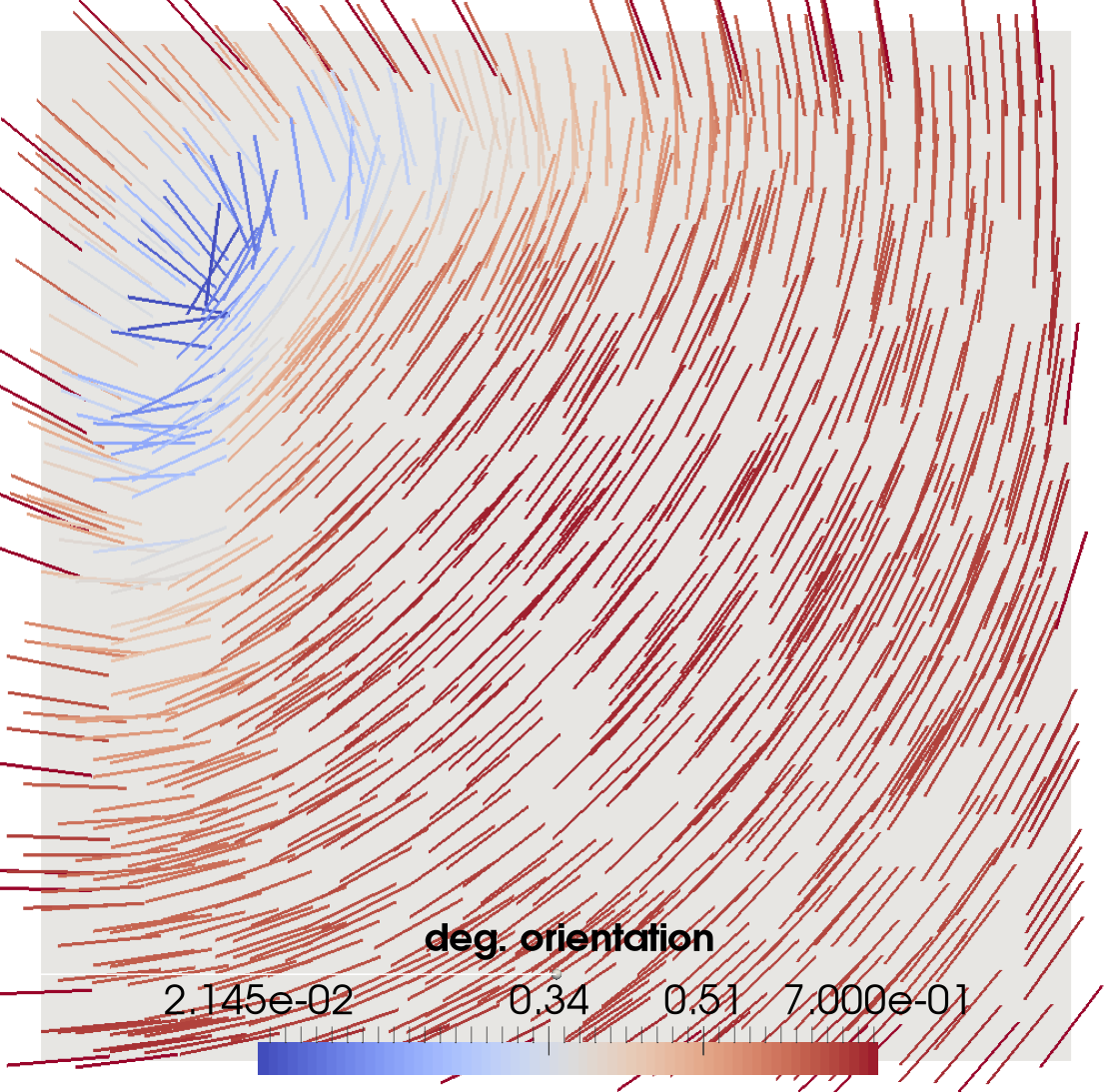} \vspace{0.08in}

\includegraphics[width=0.38\linewidth]{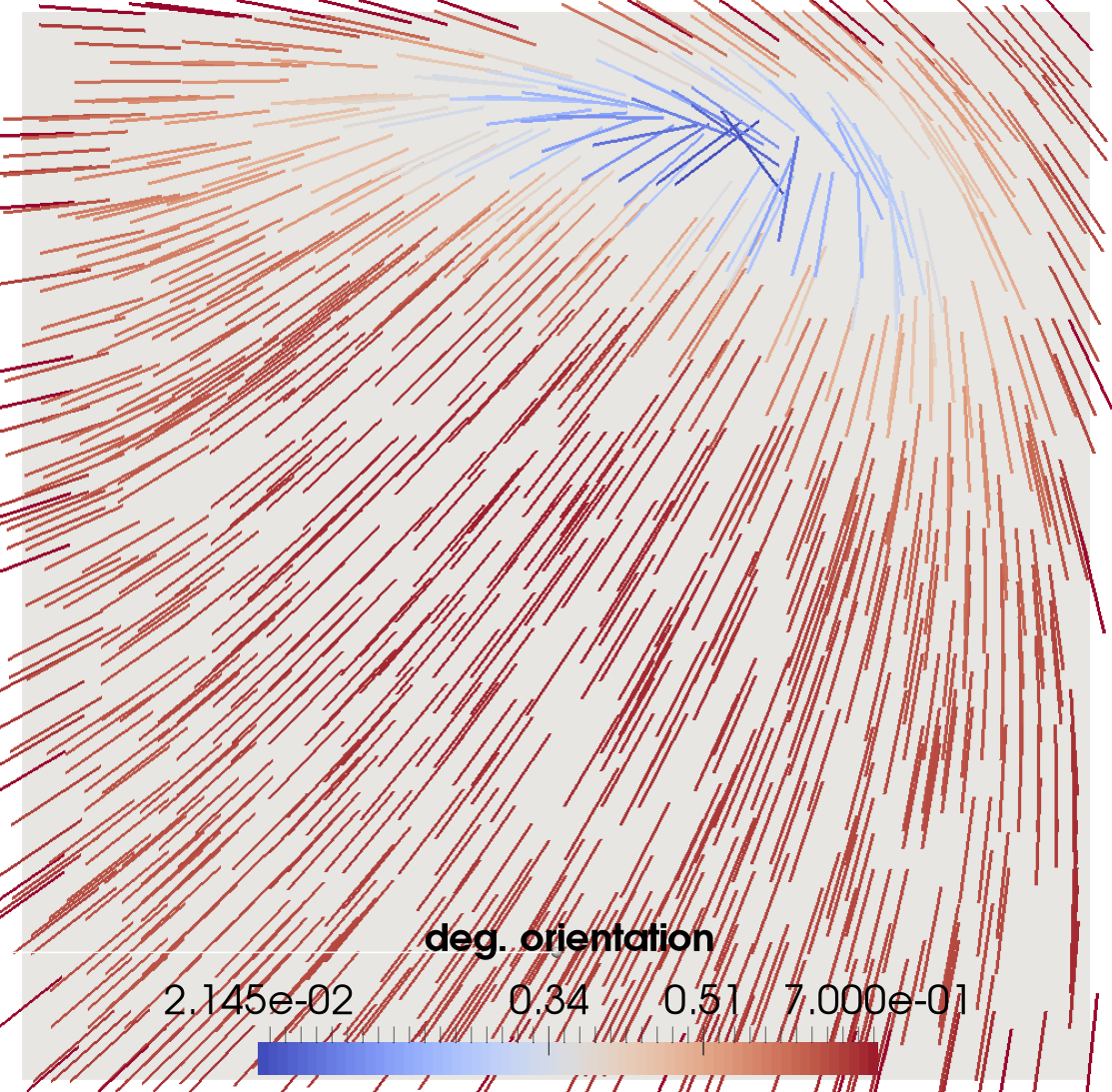} \hspace{0.05in}
\includegraphics[width=0.38\linewidth]{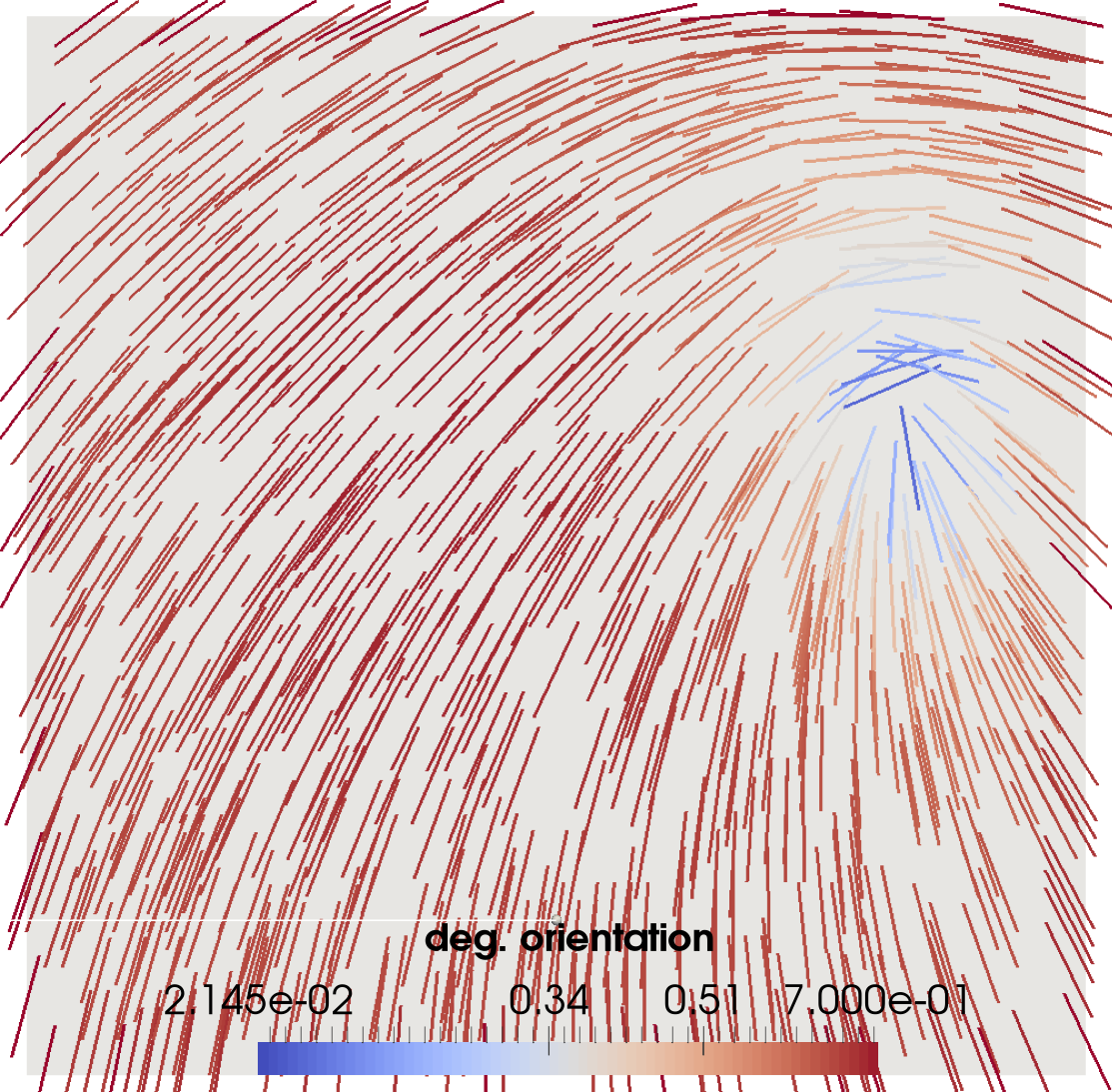} 
\caption{Horizontal slices of the $+1/2$ degree line defect in a 3-D cube domain shown in Figure \ref{fig:uni_half_line_defect_view_A} (Section \ref{sec:uni_half_degree_line_defect}). Top: left is $z=0.2$, right is $z=0.4$.  Bottom: left is $z=0.6$, right is $z=0.8$.  The location of the point defect in each plane rotates with the boundary conditions.
}
\label{fig:uni_half_line_defect_view_B}
\end{figure}
\end{center}

\subsubsection{The Standard LdG Model}\label{sec:LdG_half_degree_line_defect}

Next, we simulate the model in Section \ref{sec:LdG_theory}.  We use the boundary conditions in \eqref{eqn:uni_plus_1/2_degree_defect_BCs} and the double-well potential in \eqref{eqn:uni_double-well_splitting}.  In terms of the standard LdG model, the Dirichlet boundary conditions on $\Gmdir := \Gamma_s \equiv \Gamma_\vNN$ are
\begin{equation}\label{eqn:std_LdG_plus_1/2_degree_defect_BCs}
\begin{split}
\vQ &= s^* \left( \vNN - \frac{1}{3} \vI \right) \mbox{ on } \Gmdir,
\end{split}
\end{equation}
where $\vNN$ is taken from \eqref{eqn:uni_plus_1/2_degree_defect_BCs}, with vanishing Neumann condition on $\Gamma_{o}$; this is consistent with the boundary conditions in \eqref{eqn:uni_plus_1/2_degree_defect_BCs}. Moreover, the double-well potential is given by \eqref{eqn:Landau-deGennes_bulk_potential}, \eqref{eqn:LdG_bulk_convex_split}, where
\begin{equation}\label{eqn:std_LdG_DW_coefs_plus_1/2_degree_defect}
\begin{split}
\BulkK = 1.0, \quad \BulkA &= -7.502104, \quad \BulkB = 60.975813, \\
\BulkC &= 66.519069, \quad \Bulkstab = 552.230967,
\end{split}
\end{equation}
which is consistent with the double well potential \eqref{eqn:uni_double-well_splitting}.  The minimizer is computed using the gradient flow approach in Section \ref{sec:LdG_numerics}, with time step $\dt = 0.01$, and initialized with the minimizer from the uniaxial model.  All other parameters are the same.

For visualizing the solution, we shall use the {\em biaxiality parameter} \cite[eqn. (25)]{Majumdar_ARMA2010}, given by
\begin{equation}\label{eqn:biaxiality_param}
\beta(\vQ) = 1 - 6 \frac{\left( \tr (\vQ^3) \right)^2}{\left( \tr (\vQ^2) \right)^3},
\end{equation}
where $0 \leq \beta(\vQ) \leq 1$ and has the properties:
\begin{enumerate}
	\item $\beta(\vQ) = 0$ if and only if $\vQ$ is uniaxial, i.e. $\vQ$ has the form \eqref{eqn:Q_matrix_uniaxial};
	\item $\beta(\vQ) = 1$ if and only if $s_1 / s_2 = 2$, where $s_1$, $s_2$ appear in the biaxial form \eqref{eqn:Q_matrix_biaxial}.
\end{enumerate}
In other words, $\beta(\vQ)$ provides a simple measure of uniaxiality versus biaxiality.

Figure \ref{fig:LdG_half_line_defect_view_A} shows three dimensional views of the minimizing configuration, whereas Figure \ref{fig:LdG_half_line_defect_view_B} shows the biaxiality and point-wise $l^2$ error $|(\vQ_{\mathrm{uni}} - \vQ_{\mathrm{LdG}})(\vx)|$ between the uniaxial (uni) and standard LdG solutions.  A non-orientable line defect is observed, with final energy $\Euni^h[s_h,\vNN_h] = 4.5533587$ and achieves a maximum biaxiality of $1.0$.

\begin{center}
\begin{figure}[ht]
\includegraphics[width=0.44\linewidth]{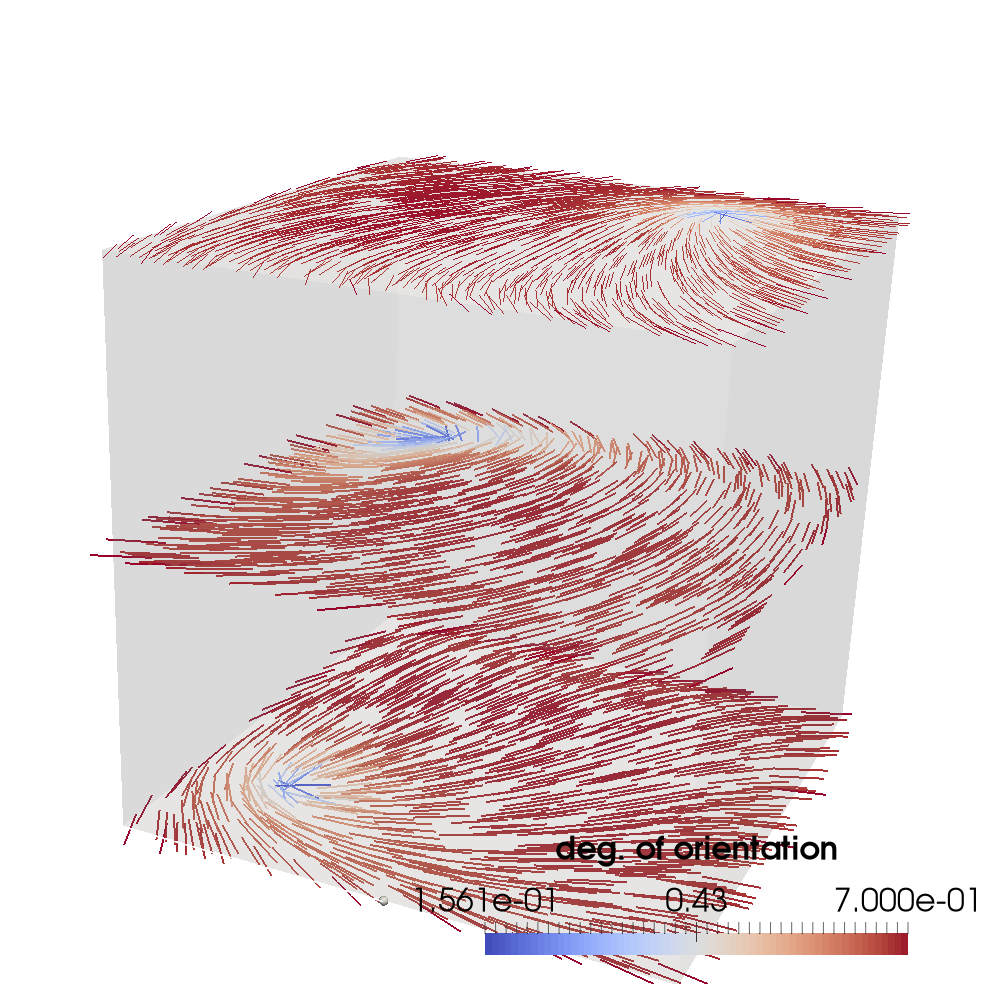} \hspace{0.1in}
\includegraphics[width=0.44\linewidth]{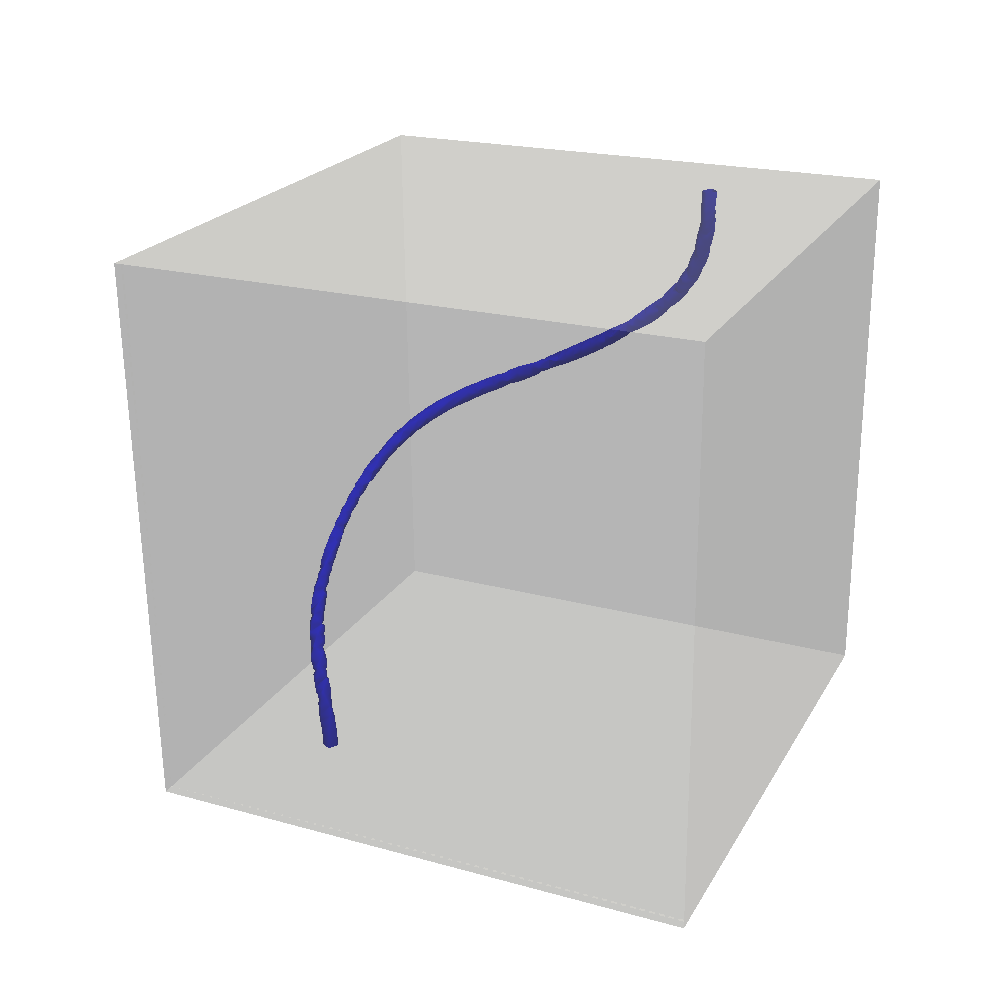}
\caption{A $+1/2$ degree line defect in a 3-D cube domain (Section \ref{sec:LdG_half_degree_line_defect}). Left: line field (taken as the dominant eigenvector of $\vQ$) is shown at levels $z=0.0$, $0.5$, $1.0$ (colored by the effective $s := (3/2) \lambda$, where $\lambda$ is the dominant eigenvalue).  Right: The $s=0.22$ iso-surface is shown that contains the line defect.  In each horizontal plane, the line field exhibits a $+1/2$ degree \emph{point} defect in 2-D.  The solution looks qualitatively a little different from Figure \ref{fig:uni_half_line_defect_view_A}.
}
\label{fig:LdG_half_line_defect_view_A}
\end{figure}
\end{center}

\begin{center}
\begin{figure}[ht]
\includegraphics[width=0.44\linewidth]{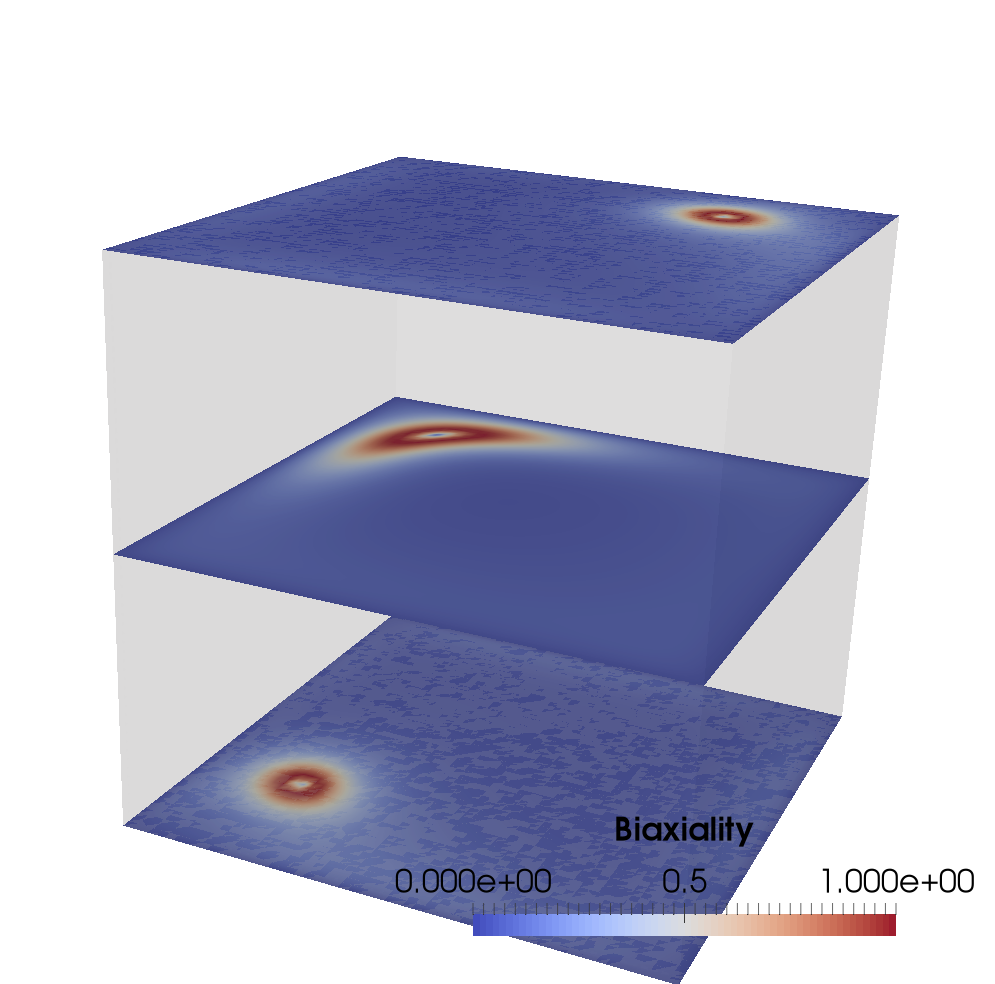} \hspace{0.1in}
\includegraphics[width=0.44\linewidth]{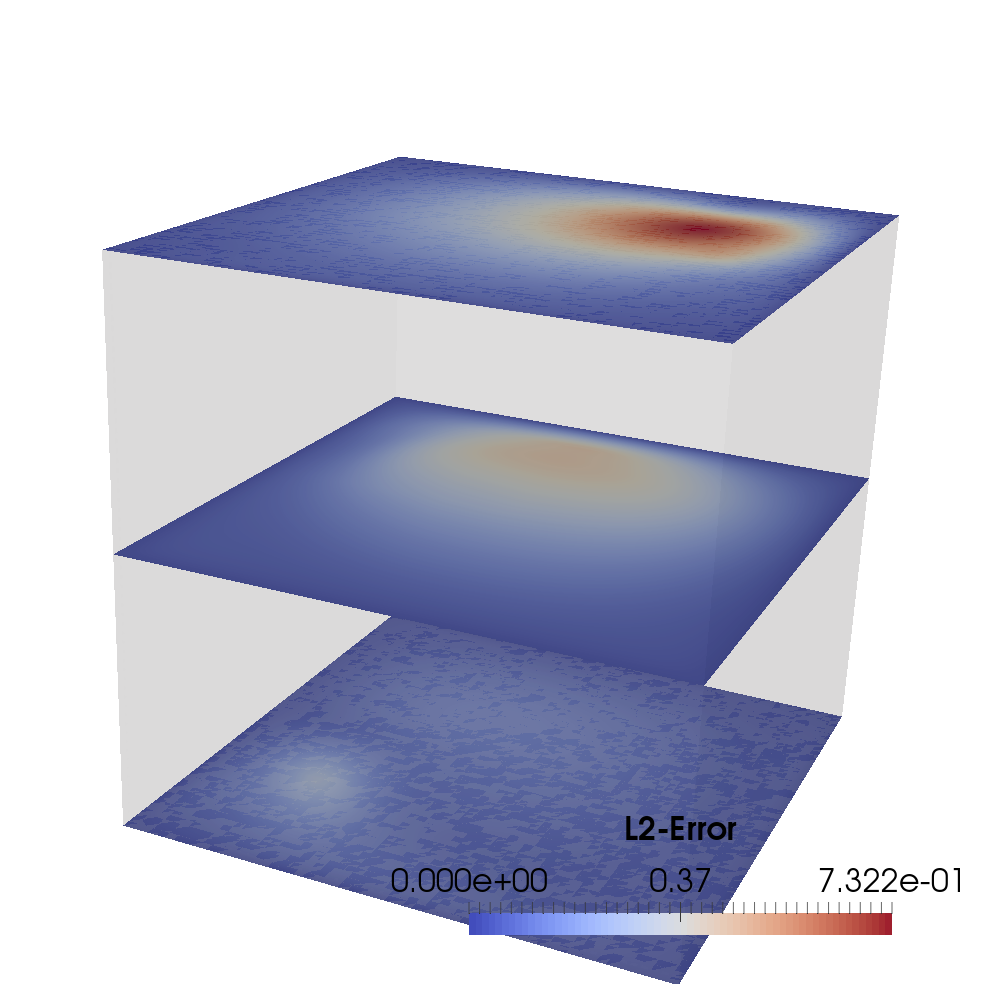}
\caption{View of the biaxiality and $l^2$ error of the solution in Figure \ref{fig:LdG_half_line_defect_view_A} (Section \ref{sec:LdG_half_degree_line_defect}). 
Left: clearly, there is a high degree of biaxiality near the defect.  Right: plots of the $l^2$ error $|(\vQ_{\mathrm{uni}} - \vQ_{\mathrm{LdG}})(\vx)|$ are shown.  The error is larger near the defect and, interestingly, it increases as a function of $z$.
}
\label{fig:LdG_half_line_defect_view_B}
\end{figure}
\end{center}

\section{Colloidal effects}\label{sec:colloids}
The presence of a colloidal particle in suspension in a LC material modifies the topology of the domain. This, in turn, can induce interesting equilibrium states with non-trivial defect configurations.  A famous example is the so-called {\em Saturn ring defect} \cite{Alama_PRE2016,Gu_PRL2000}, which is a circular ring of defect surrounding a spherical hole inside the LC domain (see Figure \ref{fig:Diagram_Saturn_Ring}).  Figure \ref{fig:Diagram_LC_Sphere_Inclusion_Defect} shows more detail on the director configuration for the Saturn ring defect.  The boundary conditions on the spherical inclusion are $\vn = \vnu$ (the unit normal of the spherical hole) on $\Gm_{i}$ and $\vn = (0,0,1)\tp$ on $\Gm_{o}$.  Note that the disclination ring can have an alternate configuration (see right plot in Figure \ref{fig:Diagram_LC_Sphere_Inclusion_Defect}), which depends on the size of the particle \cite{Wang_PRL2016,Fukuda_JPCM2004,Ravnik_LC2009}.  Either way, we emphasize that the presence of the hole can \emph{force} a defect in the LC.

\begin{figure}[ht]
\includegraphics[width=0.4\linewidth]{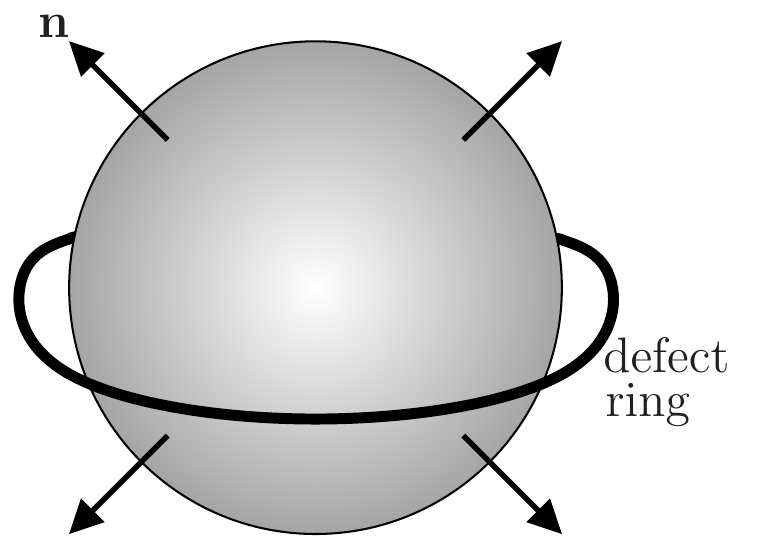} \hspace{-0.5cm}
\caption{Saturn ring defect in a director field model. A spherical colloidal particle is shown with normal anchoring conditions on its boundary (i.e. the director field $\vn$ is normal to the sphere). The singular set $\Sing$ (where $s = 0$) is marked by the thick curve and occurs depending on the outer boundary conditions (away from the sphere) imposed on $\vn$. 
}
\label{fig:Diagram_Saturn_Ring}
\end{figure}

\begin{figure}[ht]
\includegraphics[width=0.47\linewidth]{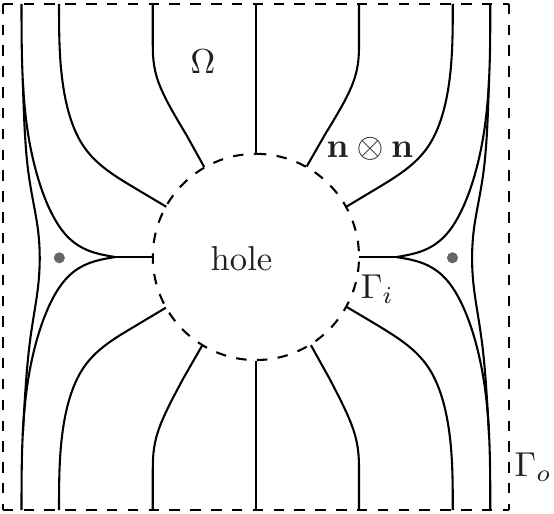} \hspace{0.3cm}
\includegraphics[width=0.47\linewidth]{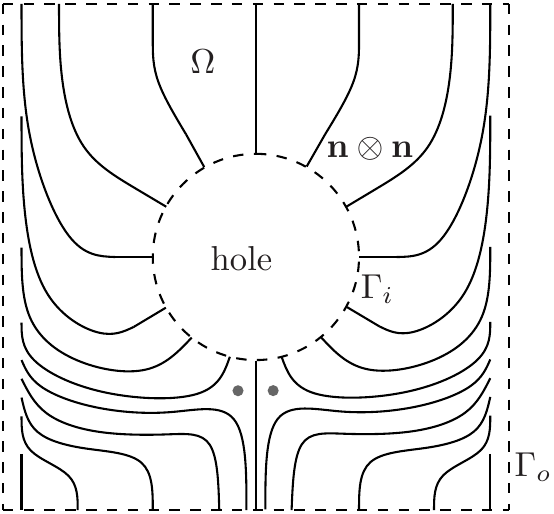}
\caption{Illustration of Saturn ring defect pattern.  Left: a two dimensional vertical slice of the domain shown in Figure \ref{fig:Diagram_Saturn_Ring}; thick lines show the line field $\vn \otimes \vn$. The defect region is marked by the two grayed circles.  Right: another possible defect configuration \cite{Wang_PRL2016,Fukuda_JPCM2004,Ravnik_LC2009}.  The ring has a much smaller radius and is below the spherical inclusion.
}
\label{fig:Diagram_LC_Sphere_Inclusion_Defect}
\end{figure}

This section discusses the capabilities of the Ericksen and uniaxially constrained Landau-deGennes models, and the corresponding numerical methods described in sections \ref{sec:Erk_FE_discretization} and \ref{sec:uniaxial_FE_discretization}, to capture defects in the presence of colloids. We shall model colloids as spherical inclusions inside the LC domain. 

\subsection{Conforming non-obtuse mesh}\label{sec:conforming_mesh}

Given an arbitrary domain, it may be quite difficult to generate a conforming, non-obtuse, tetrahedral mesh. As far as we know, the question of whether it is possible to generate a non-obtuse tetrahedral mesh of a general three dimensional domain remains open.

Here we report on numerical results over a certain non-obtuse mesh of a cylindrical domain with a hole cut out. We refer to \cite[Sec. 5.1.1]{Nochetto_JCP2018} for details about the mesh construction. For the simulations in this section, the domain $\Om$ is a ``prism'' type of cylindrical domain with square cross-section $[-0.25 \sqrt{2}, 0.75 \sqrt{2}]^2$, is centered about the $z=0$ plane, and has height $6$. It contains a spherical inclusion, with boundary $\Gm_{i}$, centered at $(\sqrt{2}/4,\sqrt{2}/4,0)$ with radius $0.283/\sqrt{2}$.

For the Ericksen model, one could in principle consider the strong anchoring conditions
\begin{equation} \label{eqn:Erk_inclusion_BCs1}
\vn = \vnu \mbox{ on } \Gm_i, \quad \vn = (0,0,1)^\top \mbox{ on } \Gm_o = \dOm \setminus \Gm_i, \quad s = s^* \mbox{ on } \dOm, 
\end{equation}
where $\vnu$ is the outer unit normal of the spherical inclusion and $s^*$ is the global minimum of the double well potential \eqref{eqn:Erk_DW_potential}. These boundary conditions \emph{do not} lead to a ring-like defect, but rather to disperse/point defects, depending on the value of $\kappa$ in $\Eerkmain$ (cf. \cite[Sec. 5.1.2]{Nochetto_JCP2018}).

Instead of \eqref{eqn:Erk_inclusion_BCs1}, we consider the following boundary conditions:
\begin{equation}\label{eqn:Erk_inclusion_BCs2}  \begin{split} 
& \vn = \vnu \mbox{ on } \Gm_i, \quad s = s^* \mbox{ on } \dOm, \\ 
& \vn \mbox{ interpolates between } (0,0,-1)^\top \mbox{ and } (0,0,1)^\top \mbox{ on }\Gm_o. 
\end{split} \end{equation}

Figure \ref{fig:Erk_Saturn_ring} shows the outcome of a numerical simulation with $\kappa = 1$, and a gradient flow with initial conditions $s = s^*$, 
\[ \begin{split}
\vn(x,y,z) =  (0,0,-1)^\top \mbox{ if } z < 0, \\
\vn(x,y,z) =  (0,0,1)^\top \mbox{ if } z \ge 0.
\end{split} \]
\begin{figure}[ht]
\includegraphics[width=0.47\linewidth]{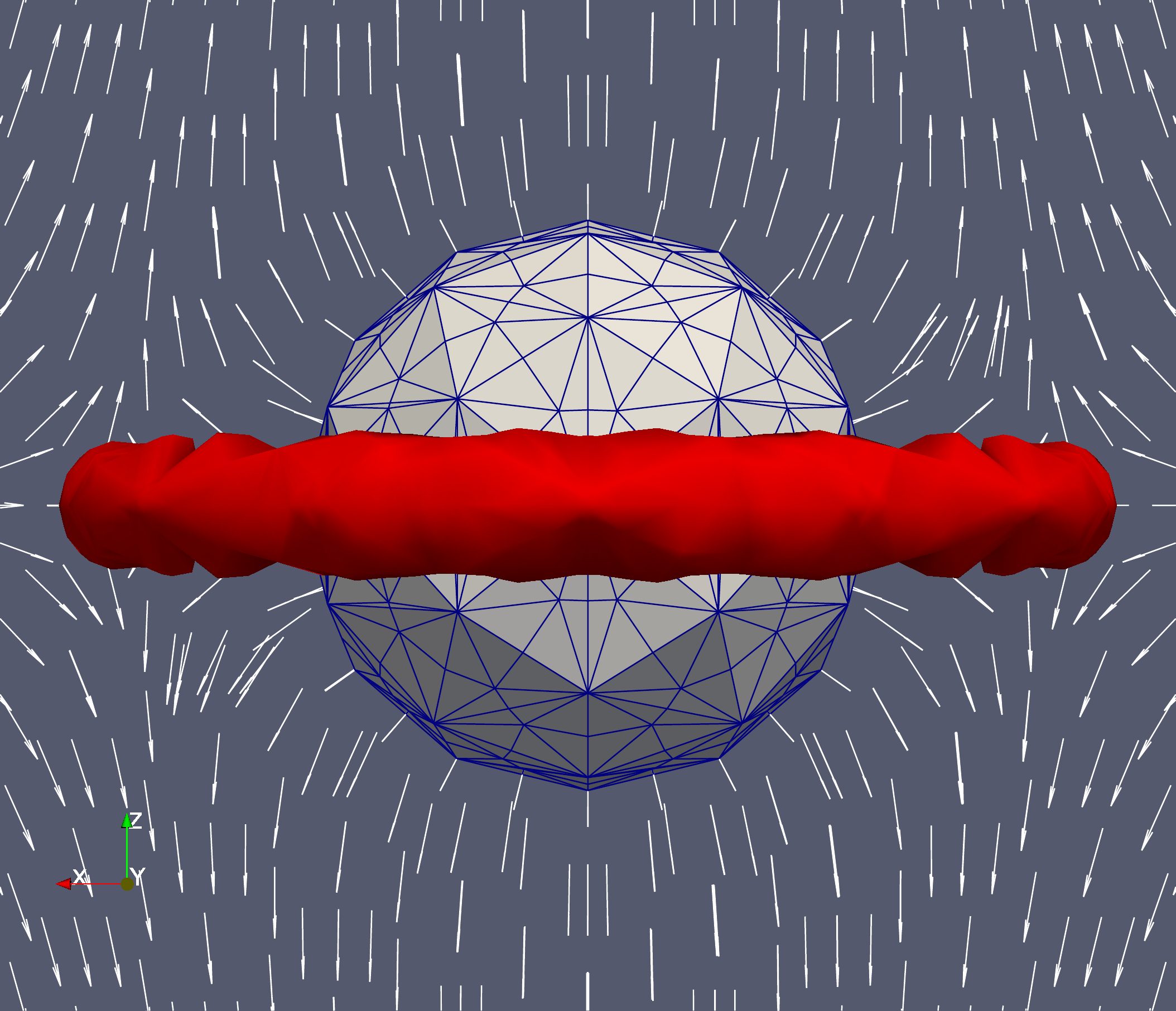} \hspace{0.3cm}
\includegraphics[width=0.47\linewidth]{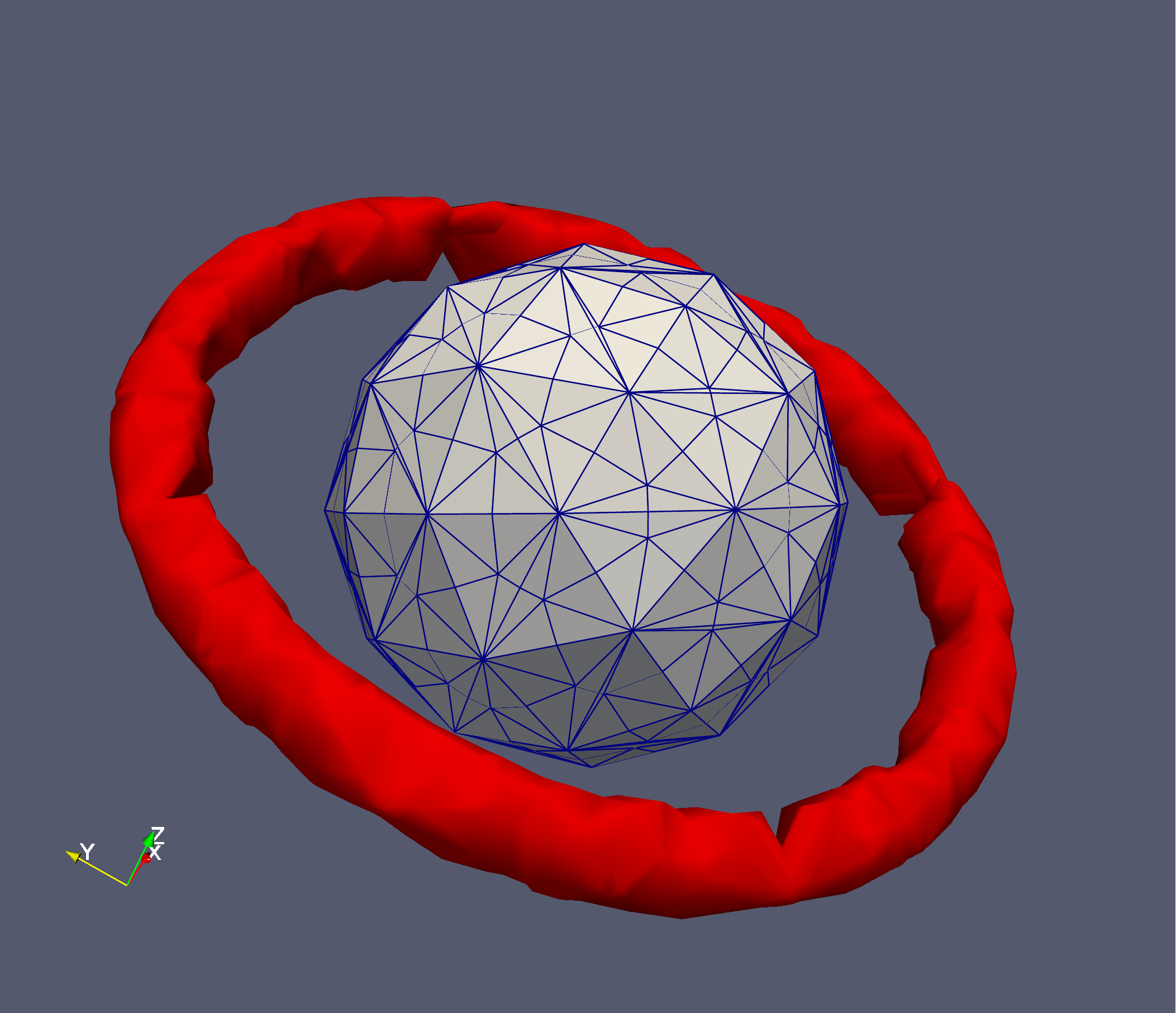}
\caption{Simulation results for the Ericksen model with boundary conditions \eqref{eqn:Erk_inclusion_BCs2}. The surface mesh of the internal hole is shown and
the defect region is indicated by the $s = 0.12$ iso-surface (plotted in red). The director field $\vn$ is depicted with white arrows.}
\label{fig:Erk_Saturn_ring}
\end{figure}
Importantly, the structure of the director field does not coincide with the one expected from the Landau-deGennes model \cite{Alama_PRE2016}. Here, at every vertical slice, the defect in the director field has degree $-1$, while in \cite{Alama_PRE2016} the degree of the defect is $-1/2$. The Ericksen model imposes an orientability constraint that is \emph{not} part of the physical problem.

The uniaxially-constrained Landau-deGennes model is capable of capturing such a non-orientable configuration. We impose the strong anchoring conditions
\begin{equation} \label{eqn:uni_LdG_inclusion_BCs}
\vn = \vnu \mbox{ on } \Gm_i, \quad \vn = (0,0,1)^\top \mbox{ on } \Gm_o, \quad \vNN = \vn \otimes \vn \mbox{ on } \dOm, \quad
s = s^* \mbox{ on } \dOm, 
\end{equation}
where now $s^* = 0.7$ is the global minimum of the double-well potential
\begin{equation}\label{eqn:DW_uni_LdG_Saturn-ring_conform_mesh}
\begin{split}
 \Bulkfunc (s) &= \psi_c(s) - \psi_e(s) \\ 
&:= (36.770913 s^2 + 1) \\
&\qquad  - (-7.3910077 s^4 + 4.5167269 s^3 + 39.271614 s^2).
\end{split}
\end{equation}
We take a time step $\dt = 10^{-3}$ for the gradient flow, which is initiated with $s=s^*$ and $\vn = (0,0,1)^\top$. Figure \ref{fig:LdG_Saturn_ring} displays the final configuration of $(s,\vNN)$. A cross-section shows the non-orientability of the resulting line field. 

\begin{figure}[ht]
\includegraphics[width=0.47\linewidth]{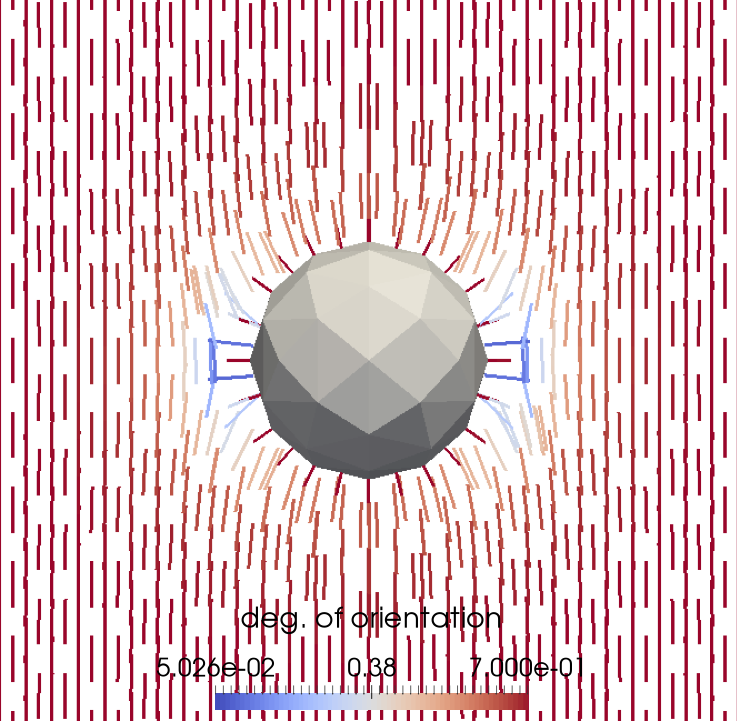} \hspace{0.3cm}
\includegraphics[width=0.47\linewidth]{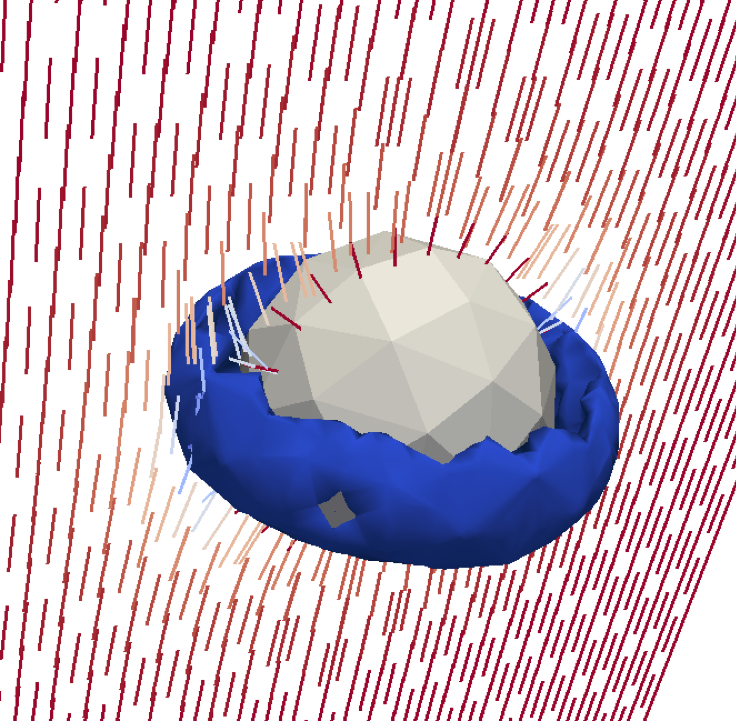}
\caption{Simulation results for the uniaxially-constrained Landau-deGennes model with boundary conditions \eqref{eqn:uni_LdG_inclusion_BCs}. In the left panel, the line field is plotted with color scale based on $s$, and the $-1/2$ degree defect are visible on the left and right sides of the spherical inclusion. The right panel shows the $s=0.25$ isosurface in blue.}
\label{fig:LdG_Saturn_ring}
\end{figure}

\subsection{Immersed boundary method}\label{sec:immersed_boundary}
Mesh weak-acuteness imposes a hard geometric constraint on the meshes, and can be extremely difficult to satisfy in implementations in three dimensions. As an alternative to it, \cite[Section 5.2]{Nochetto_JCP2018} proposes an immersed boundary approach to deal with general colloid shapes. This approach consists in representing the LC domain by using a phase field function and to incorporate a penalty term into the energies to weakly enforce boundary conditions on the colloid's boundary.

\subsubsection{Colloid representation}\label{sec:colloid_representation}
Assume the colloid is given by an open set $\Oc \subset\subset \Om$, and let $\Ochat \subset \R^d$ be a reference shape such that there is an affine parametrization $\vF : \R^d \to \R^d$,
\[
\hat{\vx} = \vF (\vx) = \vZ \vx + \vb , \quad \Ochat = \vF(\Oc).
\]
Above, $\vZ$ is a rotation matrix and $\vb$ a translation vector. We also use the signed distance functions to $\partial\Ochat$ and $\partial \Oc$, that we denote by $d$ and $\hat{d}$ respectively, and are related by
\[
d(\vx) = d(\vF(\hat{\vx})) = \hat{d}(\hat{\vx}), \quad \forall \vx \in \R^d.
\]
Applying the chain rule, we also deduce the identity
\[
\nabla_{\vx} d (\vx) = \nabla_{\hat{\vx}} \hat{d}(\hat{\vx}) \vZ.
\]

Next, we introduce a phase field function to approximate the colloidal domain. Given $\epsilon > 0$, that will represent the thickness of the transition, 
we consider 
\[
\phaseref_\epsilon : \R \to (-1,1), \quad \phaseref_\epsilon (t) = \frac12 \left( \frac2\pi \arctan \left(-\frac{t}{\epsilon} \right) + 1 \right).
\]
Using this reference phase field function, we define
\[
\phase_\epsilon (\vx) = \phaseref_\epsilon (d(\vx)) = \phaseref_\epsilon (\hat{d}(\hat{\vx})),
\]
that yields
\[
|\nabla_\vx \phase_\epsilon (\vx) |^2 = \left( \frac{1}{\pi\epsilon} \right)^2 \frac{1}{\left(1 + \left( \frac{\hat{d}(\vF(\vx))}{\eps} \right)^2\right)^2} \, |\nabla_{\hat\vx} \hat{d} (\vF(\vx))|^2.
\]

In order to motivate the penalty term that will account for the colloidal inclusion, we note a relation between bulk and surface integrals. Given $f \in C(\overline\Om)$, let
\begin{equation} \label{eqn:phase_field_bulk}
J_\epsilon (f) :=  \epsilon \frac{|\Sp^{d-1}|}{2} \iO f(\vx) |\nabla \phase_\epsilon (\vx) |^2 d\vx .
\end{equation}
Then, in the limit $\epsilon \to 0$, $J_\epsilon$ recovers the surface integral of $f$,
\begin{equation}\label{eqn:phase_field_bdry}
\lim_{\epsilon \to 0} J_\epsilon (f) = \int_{\partial\Oc} f(\vx) dS(\vx).
\end{equation}

\subsubsection{Weak anchoring}\label{sec:weak_anchoring}

Boundary conditions can either be imposed by a Dirichlet condition (\emph{strong anchoring}) or by an energetic penalization term (\emph{weak anchoring}).  Indeed, in some physical situations, weak anchoring is a better reflection of the physics \cite{deGennes_book1995,Virga_book1994}.  We take advantage of this for modeling colloids \cite{Araki_PRL2006,Conradi_SM2009,Copar_Mat2014,Copar_PNAS2015}.

Specifically, we incorporate penalization terms $\Eerkanch$ and $\Eunianch$ into either $\Eerkone$ or $\Euni$, and corresponding terms in the discrete energies.  In the $\vQ$-tensor model, a standard approach is to add the energy term
\begin{equation}\label{eqn:LdG_weak_normal_anchoring}
\begin{split}
	J_{\vnu}(\vQ) = \frac{\anchorN}{2} \int_{\partial\Oc} |\vQ - \vQ^{\vnu}|^2 dS(\vx),
\end{split}
\end{equation}
where $\vQ^{\vnu}$ is the preferred state for $\vQ$ on the boundary of the colloid, which is imposed by an energetic penalization with $\anchorN$ as the weighting term.  For example, $\vQ^{\vnu}$ may have the form \cite{Mottram_arXiv2014}
\begin{equation}\label{eqn:LdG_uniaxial_normal_BC}
\begin{split}
	\vQ^{\vnu} = s^* \left( \vnu \otimes \vnu - \frac1d \vI \right),
\end{split}
\end{equation}
which is a uniaxial tensor, where $\vnu$ is the unit vector normal to $\partial\Oc$; this is called a uniaxial, homeotropic (normal anchoring) condition.

Another popular weak anchoring condition is called \emph{planar degenerate anchoring}, whose purpose is to enforce a uniaxial state at the boundary with the director orthogonal to $\vnu$ \cite{Fournier_EPL2005,Ravnik_LC2009,Changizrezaei_PRE2017}.  Let
\begin{equation}\label{eqn:LdG_uniaxial_planar_degenerate_BC}
\begin{split}
	\widetilde{\vQ} := \vQ + \frac{s}{d} \vI, \qquad \widetilde{\vQ}^{\perp} := \left[ \vI - \vnu \otimes \vnu \right] \widetilde{\vQ} \left[ \vI - \vnu \otimes \vnu \right];
\end{split}
\end{equation}
we point out that, with our notation, $\widetilde{\vQ} = s \vNN = \vU$.  Thus, we
include the following energy term
\begin{equation}\label{eqn:LdG_weak_planar_degenerate_anchoring}
\begin{split}
	J_{\perp}(\vQ) = \frac{\anchorPone}{2} \int_{\partial\Oc} |\widetilde{\vQ} - \widetilde{\vQ}^{\perp}|^2 dS(\vx) + \frac{\anchorPtwo}{2} \int_{\partial\Oc} \left( |\widetilde{\vQ}|^2 - (s^*)^2 \right)^2 dS(\vx),
\end{split}
\end{equation}
where the quartic term is necessary in the standard LdG model to fully enforce a uniaxial state \cite[eqn. (4)]{Fournier_EPL2005} when $d=3$.

\begin{remark}
For some LC materials, in certain specialized experimental conditions, some biaxiality can be observed near the boundary despite using a uniaxial boundary condition \cite{Sluckin_PRL1985,Fournier_EPL2005}.
\end{remark}

Let us now consider the effect of imposing the uniaxial constraint $\vQ = s ( \vNN - \frac1d \vI)$ on the weak anchoring energies.  Starting with normal anchoring \eqref{eqn:LdG_weak_normal_anchoring}, \eqref{eqn:LdG_uniaxial_normal_BC}, we expand $|\vQ - \vQ^{\vnu}|^2$, exploiting that $\vQ$, $\vQ^{\vnu}$ are uniaxial (cf. \eqref{eqn:Q_matrix_uniaxial}), symmetric, and that $|\vNN| = 1$, $|\vnu| = 1$, and so obtain
\begin{equation}\label{eqn:uniaxial_normal_anchoring_expansion}
	| \vQ - \vQ^{\vnu} |^2 = 2ss^* \left( |\vNN|^2 |\vnu|^2 - \vNN \vnu \cdot \vNN \vnu \right) + \frac{d-1}{d} (s - s^*)^2 |\vNN|^2 .
\end{equation}
Since $\vnu = \nabla \phase_\epsilon / | \nabla \phase_\epsilon|$ on $\partial\Oc$, we combine the identity above with \eqref{eqn:phase_field_bulk} and \eqref{eqn:phase_field_bdry} to introduce the continuous weak normal anchoring energy for the uniaxially constrained Landau-deGennes model:
\begin{equation}\label{eqn:uniaxial_weak_normal_anchoring}
\begin{split}
	J_{\vnu} [s,\vNN] := & \frac{\anchorN}{2} | \Sp^{d-1}| \epsilon \iO 2 s s^* \left( |\vNN|^2 |\nabla \phase_\epsilon|^2 - \vNN \nabla \phase_\epsilon \cdot \vNN \nabla \phase_\epsilon \right) \\ 
& + \frac{\anchorN}{2} | \Sp^{d-1}| \epsilon \iO |\nabla \phase_\epsilon|^2 \frac{d-1}{d} (s - s^*)^2 |\vNN|^2.
\end{split}
\end{equation}

To better see the structure of \eqref{eqn:uniaxial_normal_anchoring_expansion}, we write $\vNN = \vn \otimes \vn$, use that $|\vNN|^2 = |\vn|^2$, and get
\begin{equation}\label{eqn:erk_normal_anchoring_expansion}
\begin{split}
	| \vQ - \vQ^{\vnu} |^2 &= 2 s s^* \left( |\vn|^2 |\vnu|^2 - (\vn \cdot \vnu)^2 \right) + \frac{d-1}{d} (s - s^*)^2 |\vn|^2 \\
	&= \vn\tp \left[ 2 s s^* \left( \vI - \vnu \otimes \vnu \right) + \frac{d-1}{d} (s - s^*)^2 \vI \right] \vn \\
	&= \vn\tp \Bigg{[} \left( 2 s s^* + \frac{d-1}{d} (s - s^*)^2 \right) \left( \vI - \vnu \otimes \vnu \right) \\
	&\qquad\qquad\qquad + \frac{d-1}{d} (s - s^*)^2 \left( \vnu \otimes \vnu \right) \Bigg{]} \vn =: \vn\tp H^{\vnu} \vn.
\end{split}
\end{equation}
It follows immediately from this identity that the matrix $H^{\vnu}$ is uniformly positive semi-definite.
Therefore, for the Ericksen model, this motivates to consider the weak normal anchoring energy
\begin{equation}\label{eqn:erk_normal_weak_anchoring}
\begin{split}
	J_{\vnu} [s,\vn] := & \frac{\anchorN}{2} | \Sp^{d-1}| \epsilon \iO 2 s s^* \left( |\vn|^2 |\nabla \phase_\epsilon|^2 - (\vn \cdot \nabla \phase_\epsilon)^2 \right) \\ 
& + \frac{\anchorN}{2} | \Sp^{d-1}| \epsilon \iO |\nabla \phase_\epsilon|^2 \frac{d-1}{d} (s - s^*)^2 |\vn|^2.
\end{split}
\end{equation}

Next, we proceed similarly for the weak planar degenerate anchoring \eqref{eqn:LdG_uniaxial_planar_degenerate_BC}, \eqref{eqn:LdG_weak_planar_degenerate_anchoring}.  Expanding, and using that $\vNN^2 = \vNN$, we get
\begin{equation}\label{eqn:uniaxial_planar_deg_anch_expansion_one}
\begin{split}
|\widetilde{\vQ} - \widetilde{\vQ}^{\perp}|^2 &= s^2 \left| \vNN - [\vI - \vnu \vnu\tp] [\vNN - (\vNN \vnu) \vnu\tp] \right|^2 \\
	&= s^2 \left[ 2 |\vNN \vnu|^2 - (\vnu\tp \vNN \vnu)^2  \right] \\
	&= s^2 (\vNN \dd \vnu \otimes \vnu) \left[ 2 - (\vNN \dd \vnu \otimes \vnu) \right],
\end{split}
\end{equation}
and
\begin{equation}\label{eqn:uniaxial_planar_deg_anch_expansion_two}
\begin{split}
\left( |\widetilde{\vQ}|^2 - (s^*)^2 \right)^2 &= \left( s^2 |\vNN|^2 - (s^*)^2 \right)^2 = (s - s^*)^2 (s + s^*)^2,
\end{split}
\end{equation}
which yields a slightly complicated energy functional for imposing planar anchoring with a desired degree of orientation, $s^*$.  At this point, it is worthwhile to revisit the modeling assumptions made in posing \eqref{eqn:LdG_weak_planar_degenerate_anchoring}.  The main motivation for choosing \eqref{eqn:LdG_weak_planar_degenerate_anchoring} is to enforce planar degenerate anchoring \emph{with a uniaxiality constraint}.  However, our approach enforces uniaxiality in a more explicit way, so other energy penalization terms may be used to achieve planar anchoring.

Indeed, $(\vNN \dd \vnu \otimes \vnu)^2 \ll (\vNN \dd \vnu \otimes \vnu)$ when $|\vNN \dd \vnu \otimes \vnu|$ is small, e.g. when planar anchoring is achieved.  Hence, it is reasonable to make the following approximation
\begin{equation}\label{eqn:uniaxial_planar_deg_anch_expansion_one_approx}
\begin{split}
|\widetilde{\vQ} - \widetilde{\vQ}^{\perp}|^2 &\approx 2 s^2 (\vNN \dd \vnu \otimes \vnu).
\end{split}
\end{equation}
Moreover, we can replace \eqref{eqn:uniaxial_planar_deg_anch_expansion_two} by $(2 s^*)^2 (s - s^*)^2$ as a simpler way to enforce the degree of orientation on the surface.  Therefore, combining with the phase-field approach, we assume the following continuous weak planar degenerate anchoring energy for the uniaxially constrained Landau-deGennes model:
\begin{equation}\label{eqn:uniaxial_weak_planar_deg_anchoring}
\begin{split}
J_{\perp} [s,\vNN] := & \frac{\anchorPone}{2} | \Sp^{d-1}| \epsilon \iO 2 s^2 \left( \nabla \phase_\epsilon \cdot \vNN \nabla \phase_\epsilon \right) \\ 
& + \frac{\anchorPtwo}{2} | \Sp^{d-1} | \epsilon \iO |\nabla \phase_\epsilon|^2 (2 s^*)^2 (s - s^*)^2 |\vNN|^2.
\end{split}
\end{equation}
Furthermore, writing $\vNN = \vn \otimes \vn$, we have
\begin{equation}\label{eqn:erk_weak_planar_deg_anchoring}
\begin{split}
J_{\perp} [s,\vn] := & \frac{\anchorPone}{2} | \Sp^{d-1}| \epsilon \iO 2 s^2 \left( \vn \cdot \nabla \phase_\epsilon \right)^2 \\ 
& + \frac{\anchorPtwo}{2} | \Sp^{d-1} | \epsilon \iO |\nabla \phase_\epsilon|^2 (2 s^*)^2 (s - s^*)^2 |\vn|^2.
\end{split}
\end{equation}

Then, we can define the anchoring energies for the uniaxially constrained Landau-deGennes and the Ericksen models respectively by $\Eunianch[s,\vNN] := J_{\vnu} [s,\vNN] + J_{\perp} [s,\vNN]$ and $\Eerkanch[s,\vn] := J_{\vnu} [s,\vn] + J_{\perp} [s,\vn]$. In case $\anchorPone = \anchorPtwo = 0$ (resp. $\anchorN = 0$), this yields weak normal (resp. weak planar degenerate) anchoring; otherwise, it gives rise to a weak oblique anchoring.

Clearly, if $\vNN = \vn \otimes \vn$ in $\Om$, then $\Eunianch[s,\vNN] \equiv \Eerkanch[s,\vn]$. Note that the energies \eqref{eqn:erk_normal_weak_anchoring} and \eqref{eqn:erk_weak_planar_deg_anchoring} are insensitive to changes in the sign of $\vn$.  With this, we seek to minimize the total energies
\begin{equation*}
\begin{split}
\Eerkone[s,\vn] := \Eerkmain [s,\vn] + \Eerkbulk[s] + \Eerkanch [s, \vn], \\
\Euni[s,\vNN] := \Eunimain [s, \vNN] + \Ebulk[s] + \Eunianch[s,\vNN],
\end{split}
\end{equation*}
under suitable boundary conditions.

Next, we give a discrete counterpart of $\Eerkanch[s,\vn]$.  For convenience, we define the following discrete bilinear forms:
\begin{equation}\label{eqn:discrete_inner_prod_anchor_alt}
\begin{split}
\ipanchvn (\vn_h, \vv_h) &:= \anchorN \iO \interp \Big{\{} 2 s_{h} s^* \left( (\vn_h \cdot \vv_h) |\nabla \phase_\epsilon|^2 - (\nabla \phase_\epsilon \cdot \vn_h) (\nabla \phase_\epsilon \cdot \vv_h) \right) \\
&\qquad\qquad\qquad + |\nabla \phase_\epsilon|^2 \frac{d-1}{d} (s_{h} - s^*)^2 (\vn_h \cdot \vv_h) \Big{\}} \\
&\qquad\quad + \anchorPone \iO \interp \left\{ 2 s_{h}^2 \left( \vn_{h} \cdot \nabla \phase_\epsilon \right)\left( \vv_{h} \cdot \nabla \phase_\epsilon \right) \right\} \\
&\qquad\qquad + \anchorPtwo \iO \interp \left\{ |\nabla \phase_\epsilon|^2 4 (s^*)^2 (s_{h} - s^*)^2 (\vn_h \cdot \vv_h) \right\}, \\
\ipanchs (s_h, z_h) &:= \anchorN \iO \interp \left\{ s_{h} z_{h} |\nabla \phase_\epsilon|^2 \frac{d-1}{d} |\vn_{h}|^2 \right\} \\
&\qquad\qquad + \anchorPone \iO \interp \left\{ s_{h} z_{h} 2 \left( \vn_{h} \cdot \nabla \phase_\epsilon \right)^2 \right\} \\
&\qquad\qquad\quad + \anchorPtwo \iO \interp \left\{ s_{h} z_{h} |\nabla \phase_\epsilon|^2 4 (s^*)^2 |\vn_{h}|^2 \right\}, \\
\linanchsA (z_{h}) &:= \anchorN \iO \interp \left\{ 2 z_{h} s^* \left( |\vn_{h}|^2 |\nabla \phase_\epsilon|^2 - (\vn_{h} \cdot \nabla \phase_\epsilon)^2 \right) \right\} \\ 
\linanchsB (z_{h}) &:= \anchorN \iO \interp \left\{ z_{h} s^* |\nabla \phase_\epsilon|^2 \frac{d-1}{d} |\vn_{h}|^2 \right\} \\
&\qquad + \anchorPtwo \iO \interp \left\{ z_{h} s^* |\nabla \phase_\epsilon|^2 4 (s^*)^2 |\vn_{h}|^2 \right\},
\end{split}
\end{equation}
where $\interp$ is the Lagrange interpolant. These expressions correspond
to using the so-called mass lumping quadrature which, for all $f\in C^0(\overline{\Om})$, reads
\begin{equation}\label{mass-lumping}
\int_\Om I_h f = \sum_{T\in\Tk_h} \int_T I_h f
= \sum_{T\in\Tk_h} \frac{|T|}{d+1} \sum_{i=1}^{d+1} f(x_T^i),
\end{equation}
where $\{x_T^i\}_{i=1}^{d+1}$ are the vertices of $T$. This quadrature
rule is exact for piecewise linear polynomials and has the advantage
that the finite element realization of \eqref{eqn:discrete_inner_prod_anchor_alt}
is a \emph{diagonal} matrix, which induces the following monotonicity result (proved in \cite[Lem. 6]{Nochetto_JCP2018}).
\begin{lemma}[monotone property for lumped mass matrix]\label{lem:monotone_lumped_mass}
Let $m_h : \Uh \times \Uh \rightarrow \R$ be a bilinear form defined by
\begin{equation*}
m_h(\vn_h, \vv_h) := \iO I_h \left[ \vn_h \cdot H(x) \vv_h \right] dx,
\end{equation*}
where $H$ is a continuous $d \times d$ symmetric positive semi-definite matrix.  If $|\vn_h(x_i)| \geq 1$ at all nodes $x_i$ in $\Nk_h$, then
\begin{equation*}
m_h(\vn_h, \vn_h) \geq m_h \left( \frac{\vn_h}{|\vn_h|}, \frac{\vn_h}{|\vn_h|} \right).
\end{equation*}
\end{lemma}

To apply Lemma \ref{lem:monotone_lumped_mass} to the first bilinear
form in \eqref{eqn:discrete_inner_prod_anchor_alt} we observe that $H = H^{\vnu} + H^{\perp}_{1} + H^{\perp}_{2}$, where $H^{\vnu}$ is given in \eqref{eqn:erk_normal_anchoring_expansion}, and
\begin{equation*}
\begin{split}
H^{\perp}_{1} &= 2 s_{h}^2 \nabla \phase_\epsilon \otimes \nabla \phase_\epsilon, \quad H^{\perp}_{2} = |\nabla \phase_\epsilon|^2 4 (s^*)^2 (s_{h} - s^*)^2 \vI.
\end{split}
\end{equation*}
Since $H^{\vnu}$, $H^{\perp}_{1}$, $H^{\perp}_{2}$ are all positive semi-definite, $H$ is symmetric positive semi-definite, thus
\begin{equation}\label{monotone_anchoring}
\ipanchvn (\vn_h, \vn_h) \geq \ipanchvn \left( \frac{\vn_h}{|\vn_h|},  \frac{\vn_h}{|\vn_h|} \right).
\end{equation}

Therefore, we take the discrete weak anchoring energy to be
\begin{equation}\label{eqn:colloid_weak_anchoring_energy_discrete}
\begin{split}
\Eerkanch^{h}[s_h,\vn_h] &:= |\Sp^{d-1}| \epsilon \frac{\ipanchvn (\vn_h, \vn_h)}{2},
\end{split}
\end{equation}
and the discrete total energy is then given by
\begin{equation*}
\begin{split}
\Eerkone^{h}[s_h,\vn_h] := \Eerkmain^{h} [s_h,\vn_h] + \Eerkbulk^{h}[s_{h}] + \Eerkanch^{h} [s_h,\vn_h], \\
\Euni^{h}[s_{h},\vNN_{h}] := \Eunimain^{h} [s_{h}, \vNN_{h}] + \Ebulk^{h}[s_{h}] + \Eunianch^{h}[s_{h},\vNN_{h}],
\end{split}
\end{equation*}
again noting that $\Eunianch^{h}[s_{h},\vNN_{h}] \equiv \Eerkanch^{h}[s_{h},\vn_{h}]$.

Because $\ipanchvn(\vn_h, \vn_h) = \ipanchs (s_h, s_h) + \linanchsA (s_h) + \linanchsB (s^* - 2 s_h)$, a straightforward calculation yields
\begin{equation} \label{eqn:Erk_variation_anchoring}\begin{split}
& \delta_{s_h} \Eerkanch^{h}[s_{h},\vn_{h}; z_h] = |\Sp^{d-1}| \epsilon \left( \ipanchs (s_h, z_h) + \frac{\linanchsA (z_h)}{2} - \linanchsB (z_h) \right), \quad z_h \in  \Sh(\bdys,0) \\
& \delta_{\vn_h} \Eerkanch^{h}[s_{h},\vn_{h}; \vv_h] = |\Sp^{d-1}| \epsilon \ipanchvn (\vn_h, \vv_h), \quad \vv_h \in \Vperp_{h} (\vn_{h})\cap \Uh(\bdyvu,\vzero),
\end{split} \end{equation}
where $\Sh(\bdys,0)$ and $\Uh(\bdyvu,\vzero)$ are defined in \eqref{eqn:Erk_discrete_spaces_BC}, and $\Vperp_{h} (\vn_{h})$ is given by \eqref{eqn:discrete_tangent_variation_space}. Thus, for the computation of discrete minimizers, the first variation formulas \eqref{eqn:Erk_variation_anchoring} must be incorporated into the algorithm described in Section \ref{sec:Erk_contin_gradient_flow}.

\begin{remark}\label{rem:Gm_conv_weak_anchoring}
Since $\Eunianch^{h}[s_{h},\vNN_{h}] \equiv \Eerkanch^{h}[s_{h},\vn_{h}]$ (because $\vNN_{h} = \interp \vn_{h} \otimes \vn_{h}$), proving $\Gm$-convergence for the discrete energy with weak anchoring $\Eerkanch^{h} [s_{h}, \vn_{h}]$ is exactly the same as in \cite[Sec. 8]{Nochetto_JCP2018}.
\end{remark}

\subsubsection{Computational Colloid Example}\label{sec:numerical_colloid}

We simulate a Saturn-ring defect by using the phase field approach described in Section \ref{sec:weak_anchoring}. More precisely, we consider the double-well potential \eqref{eqn:DW_uni_LdG_Saturn-ring_conform_mesh} with $\Bulkcoef = 1/16$, and represent a spherical colloidal inclusion centered at $(0.5, 0.5, 0.5)$ with radius $0.2$ by means of a phase field function with $\epsilon = 6\times 10^{-2}$. The domain is $\Om = (0,1)^3$, and we set homogeneous Neumann conditions on $\Gamma_{o} := \overline{\Om} \cap (\{ z=0 \} \cup \{ z=1 \})$, and the Dirichlet boundary conditions
\begin{equation*}
	s = s^*, \quad \vn(x,y) = (0,0,1), \quad \vNN = \vn \otimes \vn
\end{equation*}
on $\Gamma_s = \Gamma_\vNN = \dOm \setminus \Gamma_{o}$. 

Figure \ref{fig:computational_colloid} shows the result of the gradient flow algorithm described in Section \ref{sec:weak_anchoring} with time-step $\dt = 10^{-2}$ and initialized with
\[
s = s^*, \quad \vn(x,y) = (0,0,1), \quad \vNN = \vn \otimes \vn.
\]

The double-well potential and boundary conditions on $\dOm$ are essentially the same as in the experiment described in Section \ref{sec:conforming_mesh} for the uniaxially-constrained Landau-deGennes model; therefore, it is no surprise that the results are similar to those illustrated in Figure \ref{fig:LdG_Saturn_ring}.

\begin{figure}[ht]
\includegraphics[width=0.49\linewidth]{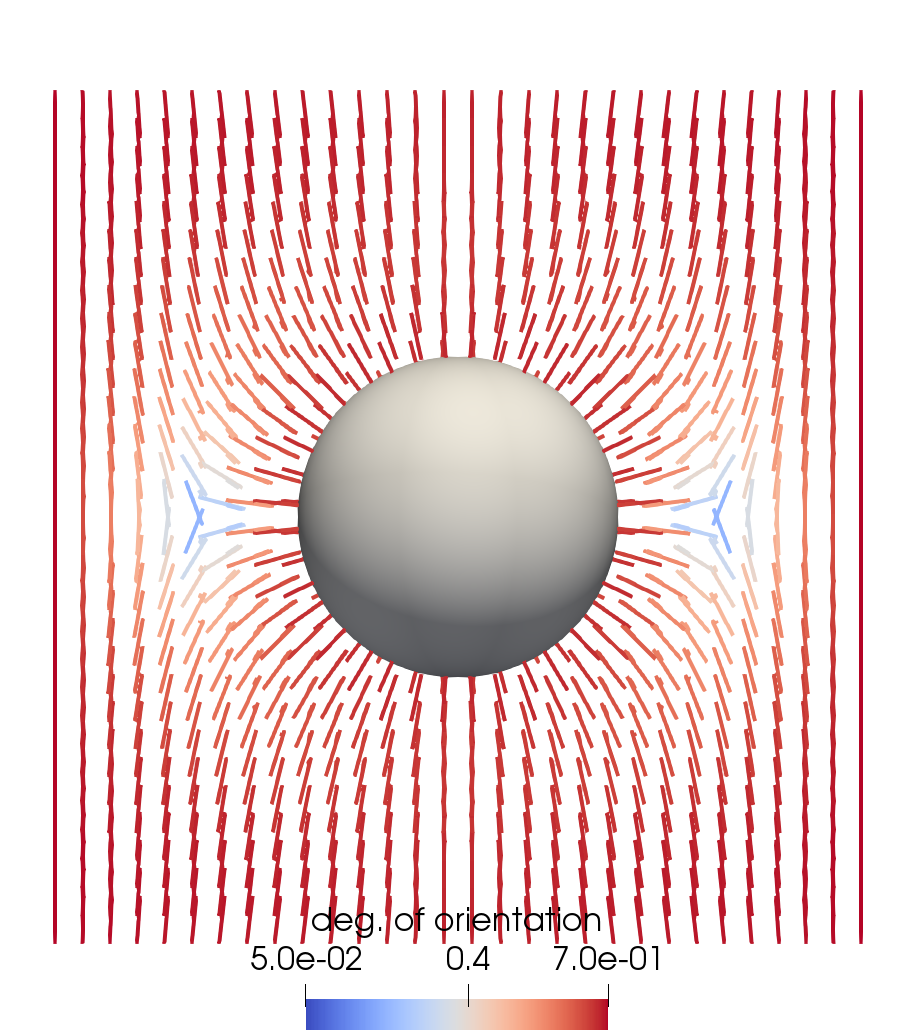} 
\includegraphics[width=0.49\linewidth]{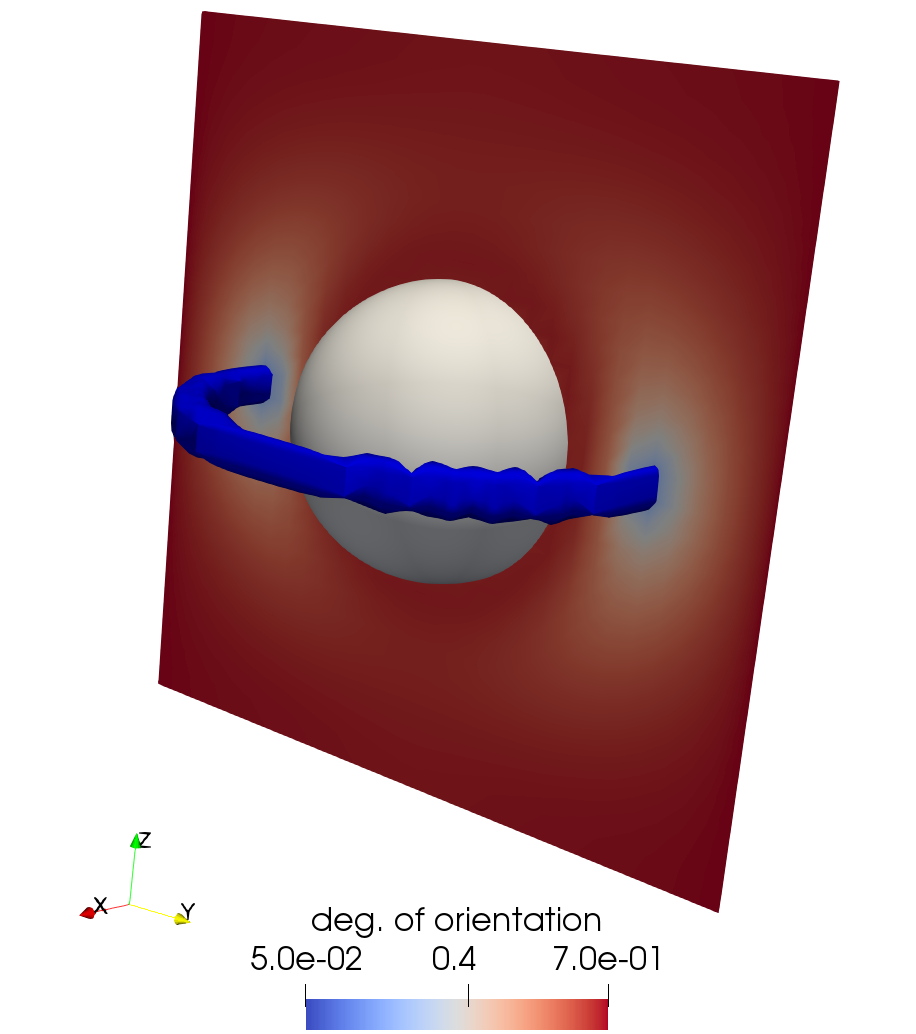}
\caption{Simulation results for the uniaxially-constrained Landau-deGennes model under the setting described in Section \ref{sec:numerical_colloid}. The left panel displays the line field $\vNN$ on the plane $\{ x = 0.5 \}$, with color scale based on the degree of orientation $s$. Two defects of order $-1/2$ are visible on the sides of the colloid. The right panel shows the degree of orientation on the same plane, and the isosurface $s=0.25$ is depicted in blue.}
\label{fig:computational_colloid}
\end{figure}

\section{Electric fields}\label{sec:electric_fields}
The LC models can be augmented by considering external forces acting on them. Here we discuss the incorporation of an electric field into the Ericksen (resp. Landau-deGennes) model. This is achieved by adding another term to the energies $\Eerkone$ (resp. $\Euni$). 

\subsection{Modified energies}\label{sec:electric_energy}

We now consider energies of the form
\begin{equation}\label{eqn:energy_with_ext_field}
\begin{split}
\Eerkone[s,\vn] = \Eerkmain [s,\vn] + \Eerkbulk[s] + \Eerkext [s, \vn] \\
\Euni[s,\vNN] = \Eunimain [s, \vNN] + \Ebulk[s] + \Euniext[s,\vNN],
\end{split}
\end{equation}
where we represent the external field energies in either the Ericksen and Landau-deGennes models by $\Eerkext[s, \vn]$ and $\Euniext[s, \vNN]$, respectively. Given an electric field $\vE$, we consider \cite{Biscari_CMT2007, deGennes_book1995}
\begin{align}
\Eerkext[s,\vn] & := - \frac{\fieldcoef}{2} \left( \ebar \iO (1- s\ga) |\vE|^2 + \ea \iO s (\vE \cdot \vn)^2 \right), \label{eqn:Erk_electric_energy} \\
\Euniext[s,\vNN] & := - \frac{\fieldcoef}{2} \left( \ebar \iO (1- s\ga) |\vE|^2 + \ea \iO s \, \vNN\vE \cdot \vE \right) .\label{eqn:uniaxial_electric_energy}
\end{align}
Above, the constant $\fieldcoef$ is a weighting parameter. If we let $\epar$, $\eperp$ be the dielectric permittivities in the directions parallel and orthogonal to the LC molecules, then $\ebar = (\epar + (d-1) \eperp)/d$ is the average dielectric permittivity and $\ea = \epar - \eperp$ is the dielectric anisotropy. Finally, $\ga = \ea/(d\ebar)$ is a dimensionless ratio; whenever $0 \le \eperp \le \epar$, it must be $0 \le \ga \le 1$.  Note that the definition of the dielectric constants here account for the dimension $d$.

From \eqref{eqn:Erk_electric_energy} and \eqref{eqn:uniaxial_electric_energy}, it is evident that, independently of $s$ and the electric constants, if $\vNN = \vn \otimes \vn$ then $\Eerkext[s,\vn] \equiv \Euniext[s,\vNN]$. Thus, our treatment of both energies follows the same pattern.

We point out that, although the second integrals in \eqref{eqn:Erk_electric_energy} and \eqref{eqn:uniaxial_electric_energy} are bounded, they may be negative.  Hence, some care is required in discretizing the electric energy in order to preserve our energy decreasing minimization scheme. First, define a discrete bilinear form analogous to \eqref{eqn:discrete_inner_prod_anchor_alt}:
\begin{equation}\label{eqn:discrete_inner_prod_elec}
\ipelec(s_h, \vn_h, \vv_h) = \iO \interp \left[ |\ea| |\vE|^2 (\vn_h \cdot \vv_h) - \ea s_h (\vE \cdot \vn_h) (\vE \cdot \vv_h) \right].
\end{equation}
To apply Lemma \ref{lem:monotone_lumped_mass}, we see that
the matrix $H$ reads
\[
H = |\epsilon_a| |\vE|^2 \vI - \epsilon_a s_h \vE\otimes\vE,
\]
and is therefore symmetric and positive semi-definite since $|s_h|\le
1$. Consequently, whenever $|\vn_h| \ge 1$,
\begin{equation}\label{monotone_electric}
\ipelec(s_h, \vn_h, \vn_h) \geq \ipelec \left( s_h, \frac{\vn_h}{|\vn_h|}, \frac{\vn_h}{|\vn_h|} \right).
\end{equation}

We now define the discrete counterpart of \eqref{eqn:energy_with_ext_field} to be
\begin{equation}\label{eqn:energy_with_ext_field_discrete}
\begin{split}
\Eerkone^{h} [s_h, \vn_h] &:= \Eerkmain^{h} [s_h, \vn_h] + \Eerkbulk^{h} [s_{h}] \\
	&\qquad + \Eerkanch^{h}[s_h,\vn_h] + \Eerkext^{h} [s_h, \vn_h ],
\end{split}
\end{equation}
where the discrete electric energy is similar to \eqref{eqn:Erk_electric_energy}
and is given by
\begin{equation}\label{eqn:electric_energy_discrete}
\begin{split}
	\Eerkext^{h}[s_h,\vn_h] = \frac{\fieldcoef}{2} \left( -\ebar \iO (1 - s_h \ga) |\vE|^2 + \ipelec(s_h, \vn_h, \vn_h) - |\ea| \iO |\vE|^2 \right).
\end{split}
\end{equation}

Observe that \eqref{eqn:electric_energy_discrete} is an approximation of
\begin{equation}\label{eqn:electric_energy_discrete_approx}
\begin{split}
\Eerkext^{h}[s_h,\vn_h] &= \frac{\fieldcoef}{2} \Big{(} -\ebar \iO (1 - s_h \ga) |\vE|^2 - \ea \iO s_h (\vE \cdot \vn_h)^2 \\
	&\qquad\qquad + |\ea| \iO |\vE|^2 (|\vn_h|^2 - 1) \Big{)},
\end{split}
\end{equation}
where the ``extra'' term is non-positive and \emph{consistent}
(i.e. it vanishes as $h \rightarrow 0$ provided the singular set
$\Sing$ has zero Lebesgue measure).  Moreover, $\iO |\vE|^2
|\vn|^2$ is constant at the continuous level, whence the extra term does not fundamentally change the energy.  However, it is needed to ensure the projection step in the algorithm decreases the (discrete) energy, which is guaranteed by \eqref{monotone_electric}.

We take first order variations of $\Eerkext^{h}$ in the directions $z_h \in \Sh(\bdys,0)$ and $\vv_h \in \Vperp_{h} (\vn_{h})\cap \Uh(\bdyvu,\vzero)$, to obtain
\[\begin{split}
& \delta_{s_h} \Eerkext^{h}[s_h, \vn_h; z_h] = \frac{\fieldcoef}{2} \left( \ebar \iO z_h \ga |\vE|^2 - \ea \iO \interp \left[ z_h (\vE \cdot \vn_h)^2 \right] \right), \\
& \delta_{\vn_h} \Eerkext^{h}[s_h, \vn_h; \vv_h] = \fieldcoef \, \ipelec(s_h, \vn_h, \vv_h).
\end{split}\]

\begin{remark}\label{rem:Gm_conv_elec_field}
Since $\Euniext^{h}[s_{h},\vNN_{h}] \equiv \Eerkext^{h}[s_h,\vn_h]$ (because $\vNN_{h} = \interp \vn_{h} \otimes \vn_{h}$), proving $\Gm$-convergence for the discrete energy with the electric field contribution $\Eerkext^{h}[s_h,\vn_h]$ is exactly the same as in \cite[Sec. 8]{Nochetto_JCP2018}.
\end{remark}

\subsection{Computational Electric Field Example}\label{sec:numerical_elec_field}

We illustrate the effect of an electric field on the same configuration as in Section \ref{sec:numerical_colloid}. Namely, with the same colloidal inclusion and boundary conditions as there, we incorporate the effect of a constant electric field $\vE = (0,1,0)$. We set the parameter $\fieldcoef = 160.0$, and the material constants $\epar = 7/3$, $\eperp = 1/3$, that yield $\ebar = 1$, $\ea = 2$, $\ga = 2/3$ in \eqref{eqn:electric_energy_discrete}.

The results of our simulation, with the same gradient flow setting as in Section \ref{sec:numerical_colloid}, are shown in Figures \ref{fig:computational_colloid_electric_A} and \ref{fig:computational_colloid_electric_B}. The presence of a strong electric force creates two noticeable effects.  Clearly, the electric energy \eqref{eqn:uniaxial_electric_energy} is minimized whenever the field $\vNN$ is aligned with $\vE$; thus, the LC molecules tend to deflect to the $y$-axis in the domain. This creates a Freedericksz-type transition \cite{Biscari_CMT2007, Hoogboom_RSA2007, Nochetto_JCP2018}, in which the director field deflects towards the $y$-axis to better align with the electric field and this, in turn, gives rise to a defect region near the sides of the cube, on which $\vNN$ is set to be vertical. Secondly, the Saturn-ring defect observed in Figure \ref{fig:computational_colloid_electric_A} is rotated. Instead of having a rotation axis parallel to the $z$-axis, the ring has a rotation axis parallel to the $y$-axis.  Figure \ref{fig:computational_colloid_electric_B} shows alternative views of the simulation.

\begin{figure}[ht]
\includegraphics[width=0.49\linewidth]{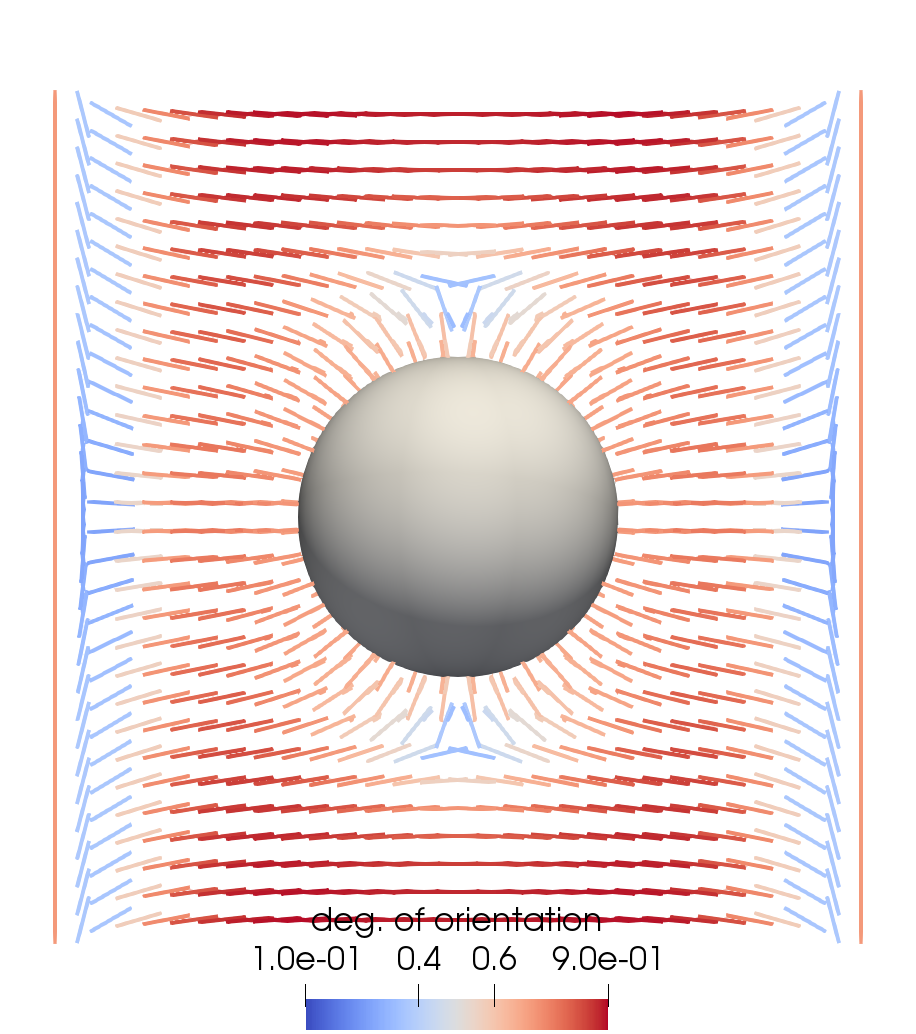} 
\includegraphics[width=0.49\linewidth]{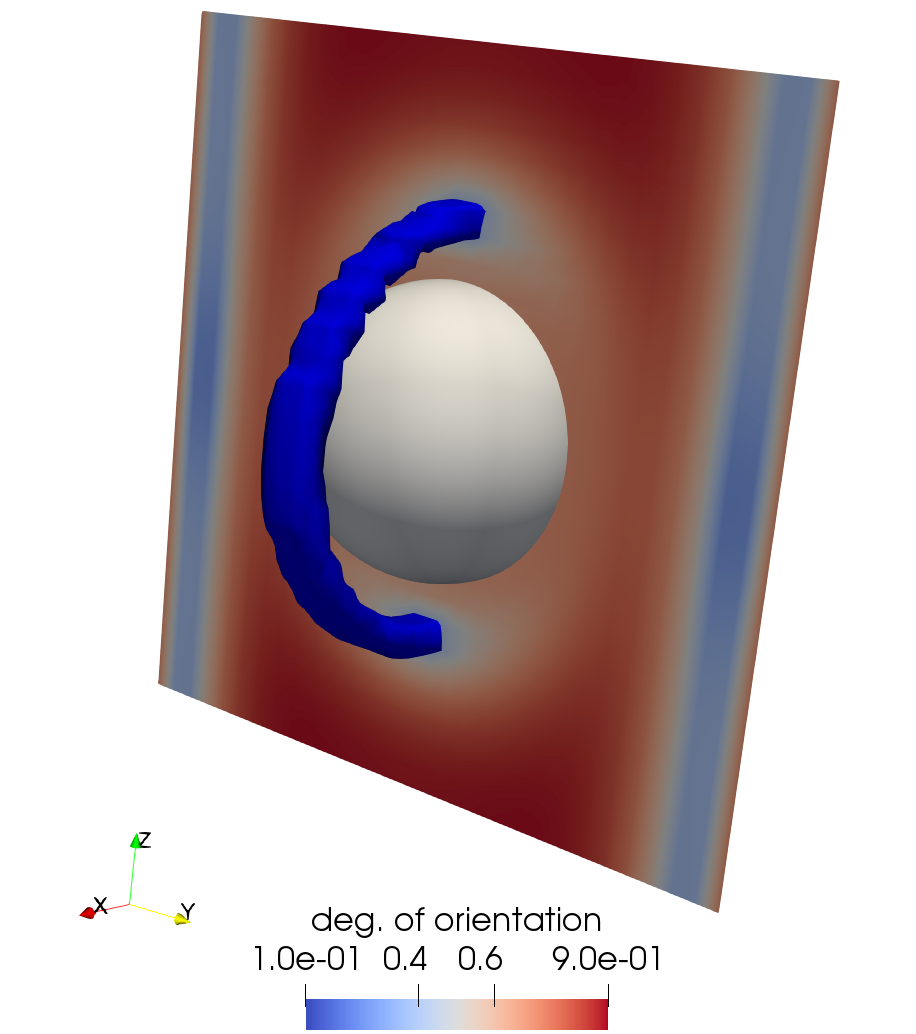}
\caption{Computational results for the uniaxially-constrained Landau-deGennes model under the presence of an electric field and a colloidal inclusion (represented using a phase field approach). The left panel shows the line field $\vNN$ on the slice $\{ x = 0.5 \}$, while the right panel depicts the degree of orientation on the same plane and the isosurface $s =0.4$. As opposed to Figure \ref{fig:computational_colloid}, the two $-1/2$-degree defects are situated on top and bottom of the colloidal particle. The strong electric field also creates a large defect region near the sides of the cube.
}
\label{fig:computational_colloid_electric_A}
\end{figure}

\begin{figure}[ht]
\includegraphics[width=0.49\linewidth]{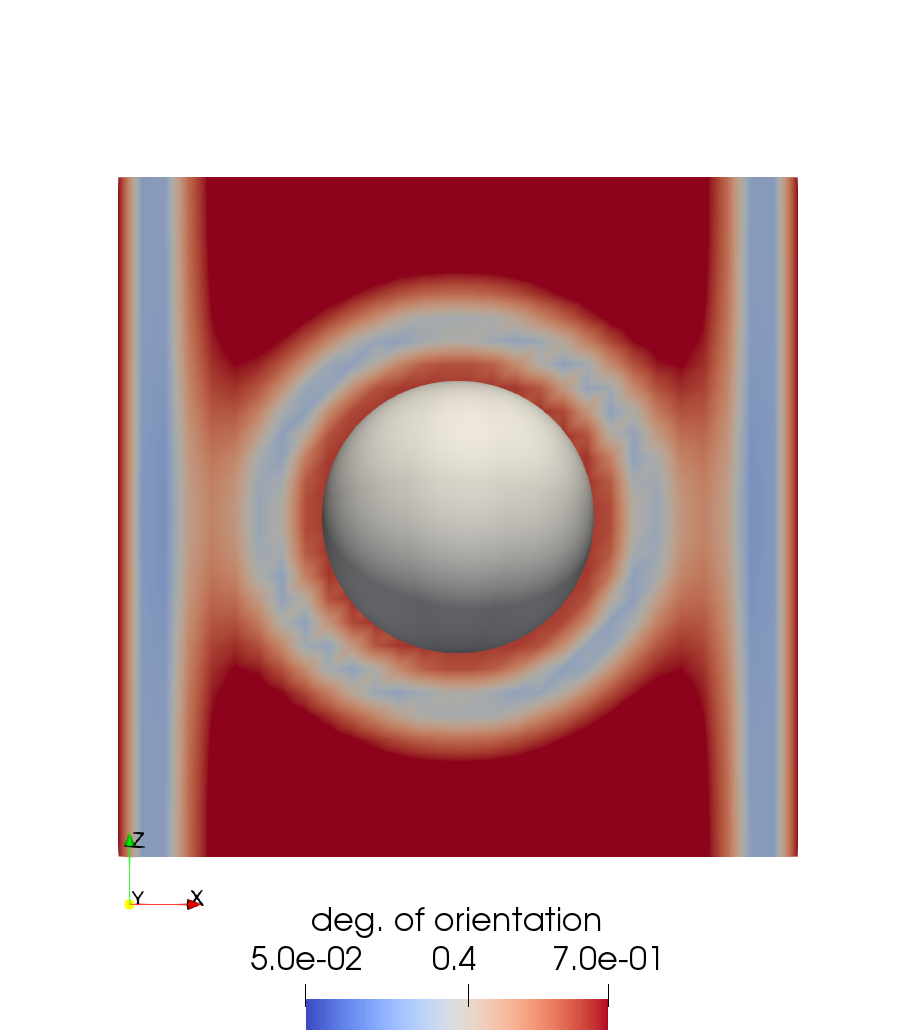}
\includegraphics[width=0.49\linewidth]{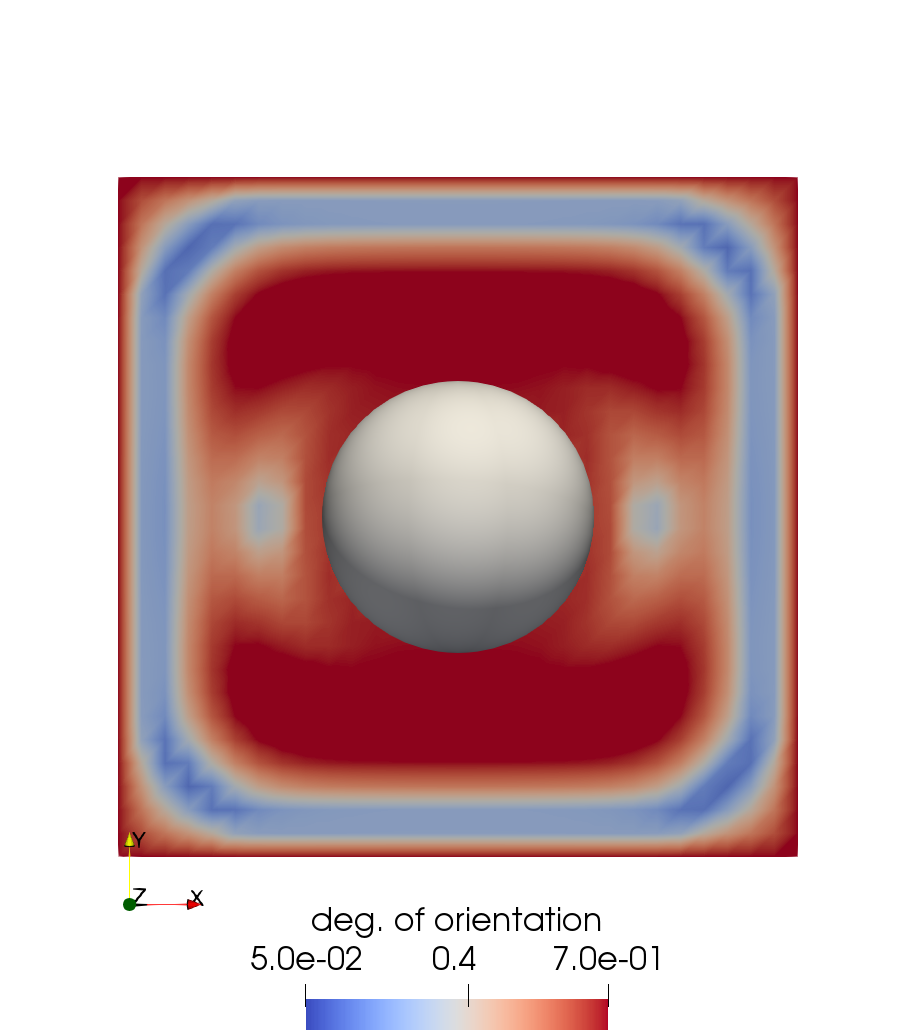}
\caption{Another visualization of the experiment in Figure \ref{fig:computational_colloid_electric_A}.  The left panel shows the degree of orientation in the $\{ y = 0.5 \}$ plane, while the right panel shows the $\{ z = 0.5 \}$ plane.  The Saturn-ring is clearly aligned with the $y$-axis, and one can see a \emph{secondary} line defect near the walls of the domain aligned with the $z$-axis.
}
\label{fig:computational_colloid_electric_B}
\end{figure}

\section{The Landau-deGennes Model With and Without the Uniaxial Constraint}\label{sec:standard_LdG_vs_uniaxial}

Our uniaxially constrained LdG model allows us to probe the fundamental modeling issue raised earlier in Section \ref{sec:fundamentals}.  Does uniaxiality significantly affect the minimizing configuration?  To the best of our knowledge, our method is the first to simulate the LdG model with \emph{uniaxiality enforced as a hard constraint}.  Thus, we can do a direct \emph{quantitative} comparison of the ``standard'' LdG approach against the uniaxially constrained case. 

We revisit the Saturn-ring example in Section \ref{sec:conforming_mesh}.  In particular, we use the boundary conditions in \eqref{eqn:uni_LdG_inclusion_BCs} and the double-well potential in \eqref{eqn:DW_uni_LdG_Saturn-ring_conform_mesh} for the uniaxially constrained model in \eqref{eqn:nematic_Q-tensor_energy_s_ntens}. For the standard (one-constant) LdG model in \eqref{eqn:Landau-deGennes_energy_one_const}, we use the following boundary conditions
\begin{equation}\label{eqn:std_LdG_inclusion_BCs}
\begin{split}
	\vQ &= s^* \left(\vnu \otimes \vnu - \frac{1}{3} \vI \right) \mbox{ on } \Gm_i, \\
	\vQ &= s^* \left( (0,0,1) \otimes (0,0,1) - \frac{1}{3} \vI \right) \mbox{ on } \Gm_o,
\end{split}
\end{equation}
which is consistent with the boundary conditions in \eqref{eqn:uni_LdG_inclusion_BCs}. Moreover, the double-well potential is given by \eqref{eqn:Landau-deGennes_bulk_potential}, \eqref{eqn:LdG_bulk_convex_split}, where
\begin{equation}\label{eqn:std_LdG_DW_coefs}
\begin{split}
	\BulkK = 1.0, \quad \BulkA &= -7.502104, \quad \BulkB = 60.975813, \\
	\BulkC &= 66.519069, \quad \Bulkstab = 552.230967,
\end{split}
\end{equation}
which is consistent with the double well potential \eqref{eqn:DW_uni_LdG_Saturn-ring_conform_mesh}.  The initial guess for the standard LdG model is chosen to be the minimizer $\vQ_{\mathrm{uni}}$ of the uniaxial model.  Both models were simulated using the following set of values for $\Bulkcoef$: $\{ 0.25, 0.16, 0.09, 0.04 \}$.

Table \ref{tbl:LdG_uniaxial_compare_model_diff} shows a comparison of the energy $\Euni[\vQ_{\mathrm{uni}}]$ with $\ELdGone[\vQ_{\mathrm{uni}}]$.  The relative error is small, but not zero, because the two numerical models are different, i.e. the error is purely due to numerical discretization and a finite mesh size.  This table illustrates that the two numerical models are consistently implemented.

\begin{table}[h]
\caption{Comparison of the standard LdG model with the uniaxially constrained model: model difference error.  The relative error between $\Euni[\vQ_{\mathrm{uni}}]$ and $\ELdGone[\vQ_{\mathrm{uni}}]$ is shown.}
\label{tbl:LdG_uniaxial_compare_model_diff}
\begin{center}
\begin{tabular}{|c|c|c|c|c|}
\hline 
$\Bulkcoef$ & $\Euni[\cdot]$ (initial) & $\Euni[\vQ_{\mathrm{uni}}]$ (final) & $\ELdGone[\vQ_{\mathrm{uni}}]$ (initial) & rel. error \\ 
\hline 
0.25 & 7.5990605 & 2.6644532 & 2.6164206 & 0.018358114 \\ 
\hline 
0.16 & 7.5990605 & 2.8031773 & 2.7497279 & 0.019438097 \\ 
\hline 
0.09 & 7.5990605 & 3.0018994 & 2.9413466 & 0.020586758 \\ 
\hline 
0.04 & 7.5990605 & 3.2711983 & 3.2176374 & 0.016646014 \\ 
\hline
0.0225 & 7.5990605 & 3.5063179 & 3.5156034 & -0.002641225 \\ 
\hline
\end{tabular}
\end{center}
\end{table}

The uniaxial solution $\vQ_{\mathrm{uni}}$, for $\Bulkcoef = 0.25$, is depicted in Figure \ref{fig:LdG_Saturn_ring}, in Section \ref{sec:conforming_mesh}.
In Figure \ref{fig:LdG_uniaxial_compare_B}, we show a direct numerical comparison of the minimizer of the standard LdG model $\vQ_{\mathrm{LdG}}$ with $\vQ_{\mathrm{uni}}$.  On the left, we plot the biaxiality parameter given in \eqref{eqn:biaxiality_param}.  
Figure \ref{fig:LdG_uniaxial_compare_B} shows that $\vQ_{\mathrm{LdG}}$ achieves maximum biaxiality near the defect.

On the right of Figure \ref{fig:LdG_uniaxial_compare_B}, we plot the pointwise quantity $|\vQ_{\mathrm{LdG}}(\vx) - \vQ_{\mathrm{uni}}(\vx)|$ with a maximum value approximately $0.228$ (note that $\vQ_{\mathrm{LdG}}$ and $\vQ_{\mathrm{uni}}$ are $O(1)$ tensors).  Figure \ref{fig:LdG_uniaxial_compare_B} clearly shows that the two solutions are quite different near the defect.
\begin{figure}[h]

\begin{center}
\includegraphics[width=0.48\linewidth]{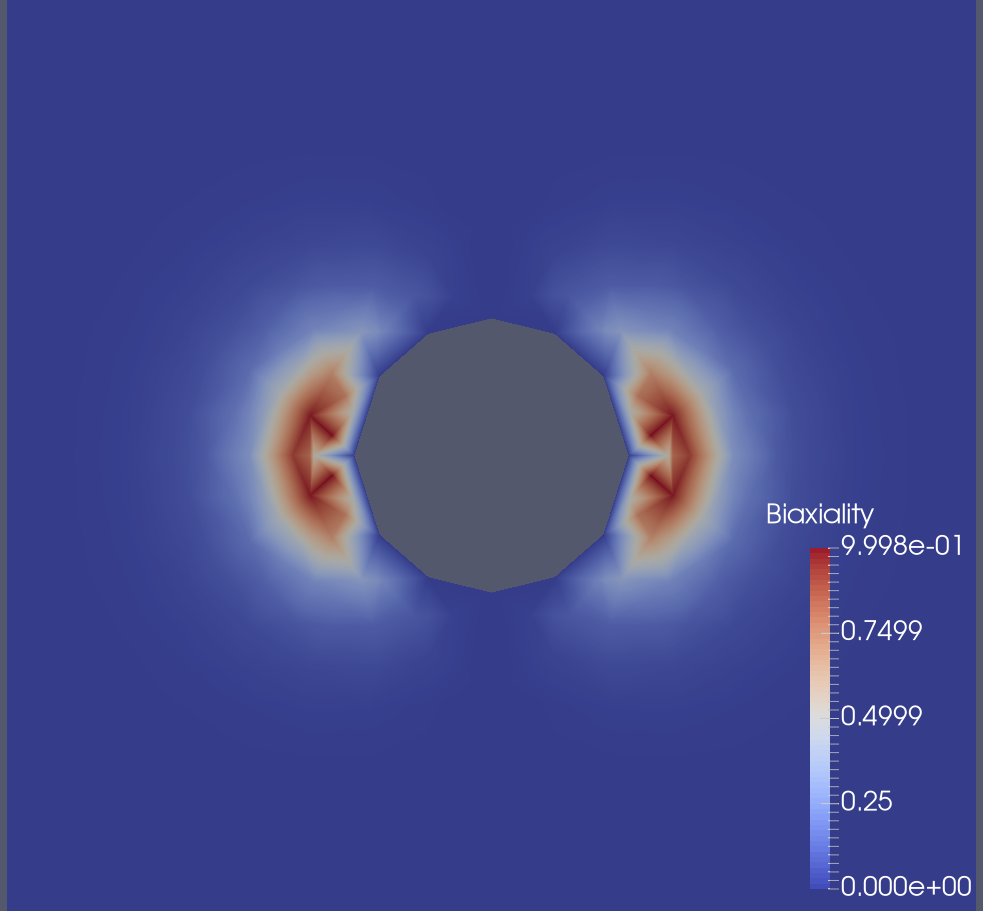}
\hspace{0.01\linewidth}
\includegraphics[width=0.48\linewidth]{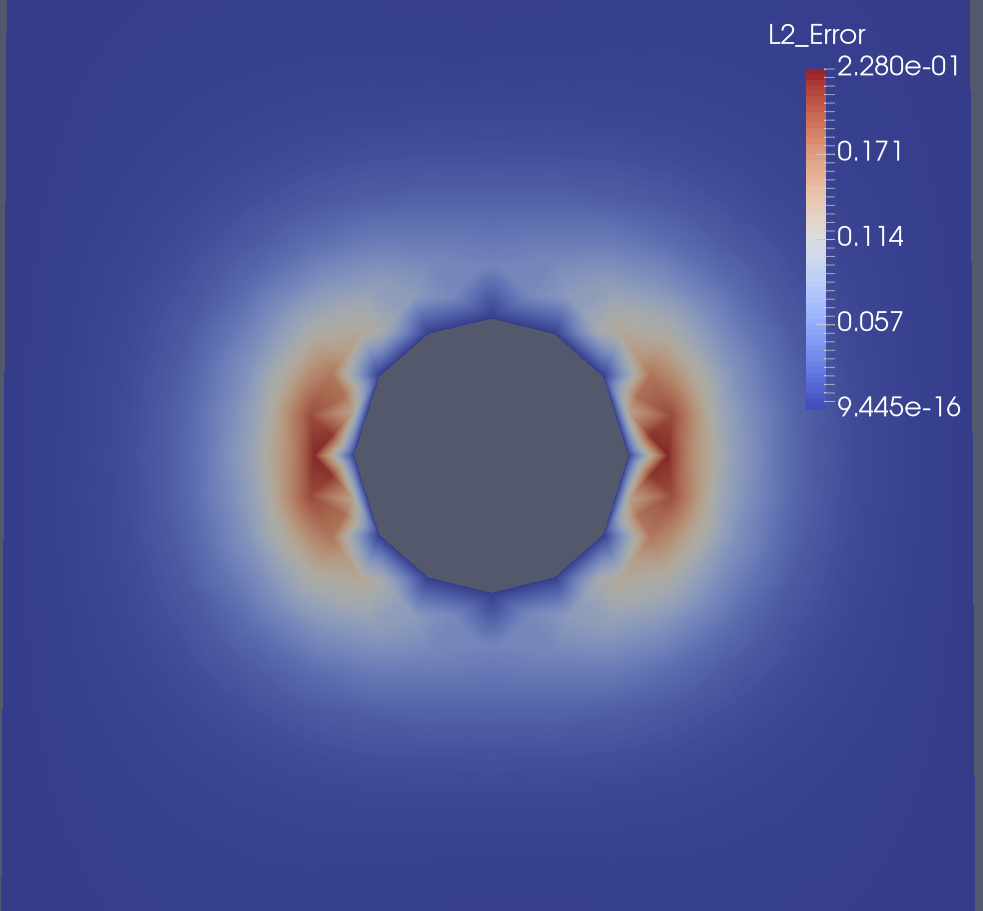}
\end{center}

\caption{\small{Comparison of the standard LdG model with \emph{uniaxially} constrained version, using the Saturn-ring defect as example \cite{Gu_PRL2000} (Section \ref{sec:standard_LdG_vs_uniaxial}). Left: the ``standard'' one-constant LdG model exhibits maximum ``biaxiality'' near the defect.  Right: the pointwise $L^2$ error between the minimizer from the standard model and the minimizer with uniaxial constraint.}}
\label{fig:LdG_uniaxial_compare_B}
\end{figure}

Moreover, the energy of the uniaxial minimizer is significantly higher than the LdG minimizer, as shown in Table \ref{tbl:LdG_uniaxial_compare_model_final}.  The fact that it is higher is not surprising --the uniaxial model is more constrained-- but it is significantly higher, which suggests that the two models could behave quite differently when other physical effects (e.g. electric/magnetic fields) are present.

\begin{table}[h]
\caption{Comparison of the standard LdG model with the uniaxially constrained model: final energy error.  The relative error between $\Euni[\vQ_{\mathrm{uni}}]$ and $\ELdGone[\vQ_{\mathrm{LdG}}]$ is shown.}
\label{tbl:LdG_uniaxial_compare_model_final}
\begin{center}
\begin{tabular}{|c|c|c|c|}
\hline 
$\Bulkcoef$ & $\Euni[\vQ_{\mathrm{uni}}]$ (final) & $\ELdGone[\vQ_{\mathrm{LdG}}]$ (final) & rel. error \\ 
\hline 
0.25 & 2.6644532 & 2.1808923 & 0.22172614 \\ 
\hline 
0.16 & 2.8031773 & 2.2931235 & 0.22242754 \\ 
\hline 
0.09 & 3.0018994 & 2.4689709 & 0.21585046 \\ 
\hline 
0.04 & 3.2711983 & 2.7850147 & 0.17457129 \\ 
\hline
0.0225 & 3.5063179 & 3.0643099 & 0.14424392 \\ 
\hline
\end{tabular}
\end{center}
\end{table}

\section{Conclusion}\label{sec:main_conclusion}

We discussed the modeling of nematic LCs and their numerical simulation. We compared three models (namely, Oseen-Frank, Ericksen and Landau-deGennes) for the equilibrium state of LCs. Because most thermotropic LCs do not exhibit any biaxiality, we focus on uniaxial LCs and compare Ericksen's model with a uniaxially-constrained Landau-deGennes model. For these, we present robust finite element schemes, which $\Gamma$-converge to the continuous problem as the mesh size tends to zero. For the solution of the resulting nonlinear equations, we design gradient flow-type algorithms that are proven to be energy-decreasing.

We presented a variety of numerical experiments, illustrating the discretizations' ability to capture non-trivial orientable and (for the Landau-deGennes model) non-orientable defects. Moreover, we incorporated additional energy terms to model colloidal effects and the effect of external fields, such as electric fields. 
Finally, we gave a detailed numerical study of the effect of imposing the uniaxial constraint (exactly) in the classic Landau-deGennes model, which is a major highlight of this work.

\bibliographystyle{abbrv}
\bibliography{MasterBibTeX}

\end{document}